\documentclass[12pt,english,fleqn,liststotoc,bibtotoc,idxtotoc,BCOR7.5mm,tablecaptionabove]{amsart}
\usepackage[T1]{fontenc}
\usepackage[latin1]{inputenc}
\usepackage{geometry}
\usepackage{color}
\usepackage{amsthm}
\usepackage{amstext}
\usepackage{amssymb}
\usepackage{amsmath}
\usepackage{graphicx}
\usepackage{young}
\usepackage{xypic}
\usepackage{amscd}
\usepackage[section]{placeins}
\usepackage[svgnames]{xcolor}
\usepackage[unicode=true,
bookmarks=true,bookmarksnumbered=true,bookmarksopen=false,
breaklinks=false,pdfborder={0 0 1},backref=false,colorlinks=true]
{hyperref}
\hypersetup{pdftitle={},
	pdfauthor={Nick Early},
	pdfsubject={},
	pdfkeywords={},
	linkcolor=black, citecolor=black, urlcolor=blue, filecolor=blue,  pdfpagelayout=OneColumn, pdfnewwindow=true,  pdfstartview=XYZ, plainpages=false, pdfpagelabels}
\usepackage{breakurl}

\usepackage{scalerel}
\DeclareMathOperator*{\bigboxplus}{\scalerel*{\boxplus}{\sum}}

\makeatletter

\theoremstyle{plain}
\newtheorem{thm}{\protect\theoremname}
\theoremstyle{plain}
\newtheorem{conjecture}[thm]{\protect\conjecturename}
\theoremstyle{plain}
\newtheorem{question}[thm]{\protect\questionname}
\theoremstyle{remark}
\newtheorem{claim}[thm]{\protect\claimname}
\theoremstyle{plain}
\newtheorem{lem}[thm]{\protect\lemmaname}
\theoremstyle{plain}
\newtheorem{prop}[thm]{\protect\propositionname}
\theoremstyle{remark}

\theoremstyle{remark}
\newtheorem{rem}[thm]{\protect\remarkname}
\theoremstyle{definition}
\newtheorem{defn}[thm]{\protect\definitionname}
\theoremstyle{definition}
\newtheorem{example}[thm]{\protect\examplename}
\theoremstyle{plain}
\newtheorem{cor}[thm]{\protect\corollaryname}
\theoremstyle{plain}

\numberwithin{thm}{section}

\usepackage{ifpdf}
\ifpdf

\IfFileExists{lmodern.sty}
{\usepackage{lmodern}}{}

\fi 

\usepackage{multicol}

\newenvironment{nouppercase}{%
	\renewcommand{\uppercasenonmath}[1]{}}{}

\makeatother

\providecommand{\claimname}{\inputencoding{latin9}Claim}
\providecommand{\conjecturename}{\inputencoding{latin9}Conjecture}
\providecommand{\corollaryname}{\inputencoding{latin9}Corollary}
\providecommand{\definitionname}{\inputencoding{latin9}Definition}
\providecommand{\examplename}{\inputencoding{latin9}Example}
\providecommand{\lemmaname}{\inputencoding{latin9}Lemma}
\providecommand{\notename}{\inputencoding{latin9}Note}
\providecommand{\propositionname}{\inputencoding{latin9}Proposition}
\providecommand{\questionname}{\inputencoding{latin9}Question}
\providecommand{\remarkname}{\inputencoding{latin9}Remark}
\providecommand{\theoremname}{\inputencoding{latin9}Theorem}
\providecommand{\problemname}{\inputencoding{latin9}Problem}

\newcommand\twoheaduparrow{\mathrel{\rotatebox{90}{$\twoheaduparrow$}}}
\newcommand\twoheaddownarrow{\mathrel{\rotatebox{270}{$\twoheaddownarrow$}}}

\begin{document}
	\author{Nick Early}
	\thanks{Perimeter Institute for Theoretical Physics and the Institute for Advanced Study. \\
		email: \href{mailto:earlnick@ias.edu}{earlnick@ias.edu}}

\title[Planarity in Generalized Scattering Amplitudes: PK Polytope, Generalized Root Systems, Worldsheet Associahedra]{Planarity in Generalized Scattering Amplitudes: PK Polytope, Generalized Root Systems and Worldsheet Associahedra}

\begin{nouppercase}
	\maketitle
\end{nouppercase}
	\begin{abstract}
In this paper we study the role of planarity in generalized scattering amplitudes, through several closely interacting structures in combinatorics, algebraic and tropical geometry.	

The generalized biadjoint scalar amplitude, introduced recently by Cachazo-Early-Guevara-Mizera (CEGM), is a rational function of homogeneous degree $-(k-1)(n-k-1)$ in $\binom{n}{k}-n$ independent variables, calculated by summing over all solutions of the generalized scattering equations; its poles are in bijection with coarsest positroidal subdivisions of the hypersimplex $\Delta_{k,n}$ and can be constructed directly from the rays of the positive tropical Grassmannian.

We introduce for each pair of integers $(k,n)$ with $2\le k\le n-2$ a system of generalized positive roots which arises as a specialization of the planar basis of kinematic invariants.  We prove that the higher root polytope $\mathcal{R}^{(k)}_{n-k}$ has volume $C^{(k)}_{n-k}$ by exhibiting a flag unimodular triangulation into simplices, in bijection with noncrossing collections of $(k-1)(n-k-1)$ $k$-element subsets of $\{1,\ldots, n\}$.  Here $C^{(k)}_{n-k}$ is the $k$-dimensional Catalan number, which counts for instance the number of standard rectangular Young tableaux.  As an application of our results for generalized roots, we give a bijection between certain positroidal subdivisions of the hypersimplex $\Delta_{3,n}$, which we call tripods, and the set of noncrossing pairs of 3-element subsets that are not weakly separated.

We show that the facets of the Planar Kinematics (PK) polytope, introduced recently by Cachazo and the author, are exactly the $\binom{n}{k}-n$ generalized positive roots.  We show that the PK specialization of the generalized biadjoint amplitude evaluates to $C^{(k)}_{n-k}$.  

Looking forward, we give defining equations, and conjecture explicit solutions using the space $(\mathbb{CP}^{n-k-1})^{\times (k-1)}$ via a notion of compatibility degree for noncrossing collections, for a two parameter family of generalized worldsheet associahedra $\mathcal{W}^+_{k,n}$.  These specialize when $k=2$ to the dihedrally invariant partial compactification of the configuration space $M_{0,n}$ of $n$ distinct points in $\mathbb{CP}^{1}$.  Many detailed examples are given throughout to motivate future work.

	\end{abstract}

	\begingroup
	\let\cleardoublepage\relax
	\let\clearpage\relax
	\tableofcontents
	\endgroup

\section{Introduction}
	In 1948, Richard Feynman introduced a pictorial formalism in Quantum Field Theory (QFT) to organize the expansion of scattering amplitudes as sums of elementary building blocks, labeled by graphs. These graphs, now called Feynman diagrams, specify which singularities of the amplitude are compatible and can be reached simultaneously.

In \cite{CEGM2019}, Cachazo, Early Guevara and Mizera (CEGM) introduced higher analog of scattering equations, and a higher analog
of the biadjoint scalar partial amplitude $m^{(2)}_n (=m_n$, in the usual notation).  In the higher analog, the $(k,n)$ generalized biadjoint amplitude $m^{(k)}_n$ acts as an interpolation between $n$ copies of $m^{(k-1)}_{n-1}$ and $n$ copies of $m^{(k)}_{n-1}$.  There is a generalized notion of Feynman diagrams for the theory, given by \textit{collections} of Feynman diagrams \cite{BC2019} and \cite{Early2019PlanarBasis} as well as \cite{CGUZ2019}, and parametrized by arrangements of metric trees on certain boundaries of a polytope.  The singularities of $m^{(k)}_n$ are extremely complex.  Each of these is a hyperplane defined by the vanishing of some linear function of \textit{Mandelstam} parameters $\mathfrak{s}_J$, but the decomposition is not unique and does not provide much intuition.  Therefore it would be very helpful to have an efficient and compact way to organize them.  They key idea is to look at subdivisions of convex polytopes called hypersimplices $\Delta_{k,n}$ which arise in combinatorial geometry.  In the case at hand, one considers those subdivisions which are regular: they are induced by projecting down the facets in the lower envelope of the convex hull of lifted vertices of $\Delta_{k,n}$.  A particular kind of tropical hypersurface, called a blade \cite{EarlyBlades}, plays a special role here.  The standard blade $((1,2,\ldots, n))$, pinned to vertices of $\Delta_{k,n}$, induces a basis of $\binom{n}{k}-n$ kinematic functions $\eta_J(\mathfrak{s})$ such that $\eta_J(\mathfrak{s})=0$ characterizes a pole of $m^{(k)}_n$.  Moreover, other poles possess particularly simple expansions as integer linear combination of the $\eta_J(\mathfrak{s})$'s.  Then $m^{(k)}_n$ becomes a (finite) sum of simple rational functions which have distinct nonzero iterated residues, where each denominator factor is dual to a \textit{weighted} blade arrangement \cite{Early2020WeightedBladeArrangements}.

Returning to CEGM amplitudes $m^{(k)}_n$: they have an extremely rich structure and the classification poses a significant challenge\footnote{Still it could be worse!  Poles of the CEGM scattering equations formula for $m^{(k)}_n$ are known \cite{AHLT2019} to be in bijection with coarsest regular \textit{positroidal} subdivisions, which are actually well-behaved in the wider context of all regular matroid subdivisions of $\Delta_{k,n}$, that is the full Dressian $\text{Dr}_{k,n}$.}.  A key item on the wishlist is a fully detailed and constructive description of the residues of $m^{(k)}_n$, but this has been considered completely unattainable in any generality.  One problem is that coarsest regular positroidal subdivisions of $\Delta_{k,n}$ have (so far, at least) resisted classification; another source of difficulty seems intriguing: there is a rich structure of linear dependencies among the residues which arises from non-simplicial cones in the secondary fan, the positive tropical Grassmannian $\text{Trop}^+G(k,n)$.  One could in principle try to resolve such linear dependencies, by enlarging the kinematic space and triangulating the maximal cones in $\text{Trop}^+G(k,n)$.  This could be implemented in practice by for instance including certain higher degree polynomials in the Plucker coordinates into the CEGM potential, such as those coming directly from the cluster algebra of the Grassmannian.  This in itself seems like a difficult and potentially interesting problem; but we take a different route.  Instead of enlarging the CEGM potential, we \textit{deform} it to something that is extremely well-behaved.

Thus, the purpose of this work is to construct a deformation $m^{(k),NC}_n$ of $m^{(k)}_n$ which is possible to describe in complete generality.  More precisely, we deform the Plucker coordinates $p_J$ on $G(k,n)$ with certain Laurent polynomials.  After fixing a parametrization, these Laurent polynomials become multi-homogeneous \textit{polynomials} $\tau_J(x_{i,j})$ for $(i,j) \in \lbrack 1,k-1\rbrack \times \lbrack 1,n-k\rbrack$.  We are interested in the Newton polytope of the product $\prod_J\tau_J$; we conjecture a combinatorial formula for its face lattice.  Our deformation $m^{(k),NC}_n$ is at the very least a rich testing ground for properties of CEGM amplitudes, but seems interesting in its own right from the perspective of algebraic, convex, combinatorial and tropical geometry.

We study the $k=3$ case $m^{(3),NC}_n$ in depth and formulate a conjecture which leads to an all (k,n) formula for the rational function expansion of $m^{(k),NC}_n$ as a sum of reciprocals of products of simple linear functions in $\binom{n}{k}-n$ independent parameters, and we obtain supporting evidence for our formula for large values of $(k,n)$; on the other hand, the analogous calculation CEGM amplitudes seems far out of reach.

In parallel, we construct a triangulation of a generalization of the type $A$ root polytope, defined in \cite{CE2020}, denoted $\mathcal{R}^{(k)}_{n}$, which is the convex hull of a set of generalized positive roots which live in the Cartesian product of $k-1$ root systems of type $A_{n-k-1}$.  We exhibit an explicit triangulation of $\mathcal{R}^{(k)}_{n}$ and show that the complex of simplices in the triangulation is isomorphic to the (reduced) noncrossing complex $\mathbf{NC}_{k,n}$.  In particular this means that $\mathcal{R}^{(k)}_n$ has volume the $k$-dimensional Catalan number $C^{(k)}_{n-k}$, where for example $C^{(2)}_{n-2} = 2,5,14,42,\ldots $ for $n=4,5,6,7,\ldots$ and $C^{(3)}_{n-3} = 5,42,462,6006,\ldots, $ for $n=5,6,7,8,\ldots$.  Here $\mathbf{NC}_{k,n}$ was studied by Petersen, Pylyavskyy and Speyer in \cite{PKPS} and in for instance the work of Santos, Stump and Welker in \cite{GrassmannAssociahedron} on the so-called Grassmann associahedron\footnote{Here the noncrossing complex $\mathbf{NC}_{k,n}$ was denoted $\tilde{\Delta}^{NC}_{k,n}$.  We change notation to avoid conflict with the hypersimplex.}.  Then the noncrossing triangulation of $\mathcal{R}^{(k)}_{n}$ should be dual to the Grassmann associahedron.  Our main conjecture makes this duality explicit.

We finally initiate the systematic study of generalized positive roots, in the context of the poset of cones in the positive tropical Grassmannian $\text{Trop}^+G(k,n)$.  In particular, generalized positive roots and the noncrossing complex provides a rich framework to explore the poset of cones in $\text{Trop}^+G(k,n)$, and to explore the poset of regular positroidal subdivisions of $\Delta_{k,n}$.

\subsection{Outline of the rest of the Introduction}
	In this paper, we have included an extended introduction whose aim is to articulate and motivate our objectives for a diverse audience.  The structure is as follows.
	\begin{enumerate}
		\item Section \ref{sec: background and summary} gives a concise summary of our main results, and Section \ref{sec: structure of paper} contains an outline of the paper.
		\item Section \ref{sec: notation and conventions} contains key definitions and notations.  In particular it has the definition of the noncrossing complex $\mathbf{NC}_{k,n}$.
		\item Section \ref{sec: intro II amplitudes} formulates the question we consider from a physical point of view and then motivates our strategy through constructions in combinatorial geometry.

	\end{enumerate}
\subsection{Background and Summary of Results}\label{sec: background and summary}

	In \cite{CEGM2019}, Cachazo, Guevara, Mizera and the author (CEGM) introduced a generalization of the Cachazo-He-Yuan (CHY) definition of biadjoint double partial amplitudes $m_n(\mathbb{I}, \mathbb{I})$ where $\mathbb{I} = (12\cdots n)$ is the standard cyclic order, from integrals over the moduli space of $n$ points on $\mathbb{CP}^{1}$ modulo projective equivalence, localized to points satisfying the scattering equations \cite{CHY2014A,CHY2014B,CHY2014C,Fairly}, to integrals over the space 
	$$X(k,n) = \left\{g\in G(k,n): \prod_J p_J \not=0\right\}\big\slash (\mathbb{C}^\ast)^n$$
	localized to solutions of the generalized scattering equations.  Here $p_J$ is a Plucker coordinate, the $k\times k$ determinant with column set $J$.
	
	In 2013, CHY \cite{CHY2014C} solved the scattering equations for a special choice of kinematics, via a high-energy limit of the type $A_1$ Y-system, proving analytically, by summing all $\lfloor (n-1)/2\rfloor$ solutions to the Planar Kinematics scattering equations, that the biadjoint partial amplitude evaluates to the Catalan number\footnote{Note that we follow the convention of the original paper of CHY \cite{CHY2014C}, where there is an alternating sign.}
	$$m_n(\mathbb{I},\mathbb{I}) = (-1)^{n+1}C^{(2)}_{n-2}$$
	where 
	$$C^{(2)}_{n-2} = 2,5,14,42,132,\ldots $$ 
    with $ n=4,5,6,7,8,\ldots$ the 2-dimensional Catalan numbers.
	
	The generalized biadjoint partial amplitude $m^{(k)}_n(\mathbb{I},\mathbb{I})$ is a highly structured rational function $m^{(k)}(\mathbb{I}_n,\mathbb{I}_n)$ on a dimension $\binom{n}{k}-n$ kinematic space $\mathcal{K}(k,n)$; it is calculated by summing over the critical points of the potential function
	\begin{eqnarray}\label{eq: potential function intro}
		\mathcal{S}_{k,n} & = & \sum_{1\le j_1<\cdots <j_k \le n}\log(p_{j_1,\ldots, j_k})s_{j_1\cdots j_k},
	\end{eqnarray}
	where $s_{j_1\cdots j_k}$ are parameters with distinct (unordered) indices.  It grows quickly in size and complexity as $k$ increases.  The poles alone can be as complicated as any coarsest regular positroidal subdivision of the hypersimplex $\Delta_{k,n}$, see \cite{BC2019,CGUZ2019,Early19WeakSeparationMatroidSubdivision,Early2020WeightedBladeArrangements,LPW2020,SpeyerStumfelsTropGrass,arkani2020positive} and have as such resisted attempts to make a systematic classification.  
	
	In this paper, we study a polytope which is defined from Equation \eqref{eq: potential function intro} at the special kinematic point, the Planar Kinematics (PK) point.  With respect to the cyclic order $(12\cdots n)$, the PK point is given by 
\begin{eqnarray*}\label{planarK}
	s_{12\ldots k}= s_{23\ldots k+1}=\ldots =s_{n1\ldots k-1}& =& 1, \\
	s_{n1\ldots k-2,k} = s_{12\ldots k-1,k+1}=\ldots = s_{n-1,n\ldots k-3,k-1}& = & -1,\nonumber
\end{eqnarray*}
where all other $s_J$ are set to zero.  The PK point has an equivalent characterization in terms of the \textit{planar basis} of linear functions $\eta_J$ on the kinematic space $\mathcal{K}(k,n) \subset \mathbb{R}^{\binom{n}{k}}$, wherein all $\eta_J$ are set to 1.  This basis was introduced in \cite{CHY2014C} for $k=2$ and subsequently generalized to all $k$ in \cite{Early2019PlanarBasis,Early2020WeightedBladeArrangements}.  The planar kinematics scattering equations appear as the set of equations occurring as the first derivatives of the PK potential function
	$$\mathcal{S}^{(PK)}_{k,n} = \sum_{j=1}^n\log\left(\frac{p_{j,j+1,\ldots, j+k-2,j+k-1}}{p_{j,j+1,\ldots, j+k-2,j+k}}\right),$$
	which provides a $(\mathbb{C}^\ast)^n$-invariant analog of the mirror superpotential \cite{Rietsch,RietschWilliams}, see also \cite{karp2019moment}, 
	\begin{eqnarray*}\label{kapi}
		{\mathcal F}_q := \sum_{i=1,\, i\neq n-k}^{n}\frac{p_{i,i+1,\ldots, i+(k-2),i+k}}{p_{i,i+1,\ldots, i+(k-2),i+(k-1)}} + q \frac{p_{n-k,n-k+1,\ldots,n-1,1}}{p_{n-k,n-k+1,\ldots, n-1,n}}.
	\end{eqnarray*}
	
	In \cite{CE2020B}, the PK potential function $\mathcal{S}^{(PK)}_{k,n}$ was defined as a specialization of the potential function arising in the generalized scattering equations \cite{CEGM2019}.  Here the main result was to prove that critical points of $\mathcal{S}^{(PK)}_{k,n}$ lift to (equivalence classes of) critical points of $\mathcal{F}_q$.
	
	In more detail, it was shown that the critical points of $\mathcal{S}^{(PK)}_{k,n}$ can be represented by images of certain $k$-tuples of roots of unity which inject via the Veronese mapping into the torus quotient $G(k,n) \slash (\mathbb{C}^\ast)^n$.  Specifically, the injection is from the set of $k$-tuples of roots of unity which do not sum to zero, into the set of binary Lyndon words with $k$ ones and $n-k$ zeros.  The critical points of the PK potential function are equivalence classes of those critical points of the mirror superpotential studied in \cite{Rietsch} and \cite{RietschWilliams}, coming from $k$-tuples or roots of unity which do not sum to zero.
	
	In this paper, we study the PK potential function in its own right; in \cite{CE2020B} Cachazo and the author introduced the Planar Kinematics (PK) polytope $\Pi^{(k)}_{n-k}$, which is the convex hull of all exponent vectors of monomials appearing in $\exp\left(-\mathcal{S}^{(PK)}_{k,n}\right)$ under a certain parameterization.  More precisely, we need that 
	$$\exp\left(-\mathcal{S}^{(PK)}_{k,n}\right) = \prod_{j=1}^n\frac{p_{j,j+1,\ldots, j+k-2,j+k}}{p_{j,j+1,\ldots, j+k-2,j+k-1}}$$
	becomes a Laurent polynomial when it is evaluated on the image of the positive parameterization \cite{PositiveGrassmannian} of the nonnegative Grassmannian; see Section \ref{sec: positive parameterization} below for a Newton-polytopal reinterpretation.  Based on that finding\footnote{Note the departure from \cite{CE2020B}, $\Pi^{(k)}_{n-k}$ was defined in terms of facet inequalities; in this paper we prove that the latter is equal to the Newton polytope here.}, we make the definition
	\begin{eqnarray}
		\Pi^{(k)}_{n-k} = \text{Newt}\left(\frac{P_1\cdots P_{k-1}Q_1\cdots Q_{n-k-1}}{\prod_{(i,j)\in \lbrack 1,k-1\rbrack \times \lbrack 1,n-k\rbrack}x_{i,j}}\right),
	\end{eqnarray}
	where 
	$$P_i = \sum_{j=1}^{n-k}x_{i,j},$$
	whose Newton polytope is a simplex of dimension $n-k-1$,
	$$\text{Newt}(P_i) = \left\{\sum_{j=1}^{n-k}\alpha_{i,j} e_{i,j}: \sum_{j=1}^{n-k}\alpha_{i,j}=1\right\},$$
	and where
	$$Q_j = \sum_{\{(t_i)\in \{0,1\}^{k-1}:\ t_1\le \cdots \le  t_{k-1}\}}x_{1,j+t_1}x_{2,j+t_2}\cdots x_{k-1,j+t_{k-1}}$$
	has Newton polytope 
	$$\text{Newt}(Q_j) = \left\{\sum_{i=1}^{k-1}\alpha_{i,j}e_{i,j} + \alpha_{i,j+1} e_{i,j+1}: 1\ge \alpha_{1,j} \ge \alpha_{2,j} \ge \cdots \ge \alpha_{k-1,j} \ge 0,\ \alpha_{i,j} + \alpha_{i,j+1}=1\right\},$$
	a simplex of dimension $k-1$.
	
	Our approach this paper to study $\Pi^{(k)}_{n-k}$ relies on a generalization \cite{CE2020B} of the system of positive roots of type $A$ is embedded in a $(k-1)$-fold Cartesian product of root lattices of type Dynkin type $A_{n-k}$.  We observe that the set of linear dependencies for this lattice of higher roots is governed by noncrossing combinatorics of $k$-element sets; we prove that the higher rank root polytope $\mathcal{R}^{(k)}_{n-k}$, admits a unimodular triangulation into $C^{(k)}_{n-k}$ simplices of maximal dimension $(k-1)(n-k-1)$, where $C^{(k)}_{n-k}$ is a $k$-dimensional Catalan number, which counts for instance the number of rectangular $k\times n$ standard Young tableaux:
		\begin{center}
		\begin{tabular}{|c|c|c|c|c|c|}
			\hline
			$k\setminus n$ & $4$ & $5$ & $6$  & $7$ & $8$ \\ \hline 
			$2$      & $2$ & $5$ & $14$ & $42$ & $132$ \\
			$3$      &     & $5$ & $42$ & $462$ & $6\,006$ \\
			$4$      &     &     & $14$ & $462$ & $24\,024$ \\ 
			$5$      &     &     &      &  $42$ & $6\, 006$ \\
			$6$      &     &     &      &       &  $132$   \\
			\hline
		\end{tabular}.
	\end{center}
	
Let us first recall from \cite{CE2020B} the construction of higher positive roots, that is to say the linear functions $\gamma_J$, defined as follows.  Let $e_{i,j}$ be the standard basis for $\mathbb{R}^{(k-1)\times (n-k)}$ with coordinate functions $\alpha_{i,j}$.  For any subset $J$ of $\{1,\ldots, n\}$ define $\alpha_{i,J} = \sum_{t\in J} \alpha_{i,t}$.  When $J = \emptyset$ then $\alpha_{i,\emptyset} = 0$.  

Now for each $k$-element subset $J$ of $\{1,\ldots, n\}$, denote by $\gamma_J:\mathbb{R}^{(k-1)\times (n-k)} \rightarrow \mathbb{R}$ the generalized root
\begin{eqnarray}\label{eq:gamma intro}
	\gamma_J(\alpha) = \sum_{i=1}^{k-1} \alpha_{i,\lbrack j_i-(i-1),j_{i+1}-i-1\rbrack}.
\end{eqnarray}
Choosing new variables $\alpha'_{i,j}$ such that $\alpha_{i,j} = \alpha'_{i,t} - \alpha'_{i,t+1}$, for each $i=1,\ldots, k-1$ we recover one copy of the usual system of positive roots $\alpha'_{i,a} - \alpha'_{i,b}$ for $a<b$.

Note that $\gamma_J$ is identically zero when $J$ is a subinterval $J = \lbrack i,i+(k-1)\rbrack$ of $\{1,\ldots, n\}$.  Usually we omit the argument of $\gamma_J(\alpha)$ and write simply $\gamma_J$. 

Given a $k-1$ tuple of integers $\mathbf{\lambda} = (\lambda_1,\ldots, \lambda_{k-1})$, denote by $\mathcal{H}^{\lambda}_{k,n}$ the codimension $(k-1)$ (affine) subspace of $\mathbb{R}^{(k-1)\times (n-k)}$,
\begin{eqnarray}\label{eq:Hkn}
	\mathcal{H}^{\lambda}_{k,n} & = &  \left\{(\alpha) \in \mathbb{R}^{(k-1)\times (n-k)}: \sum_{j=1}^{n-k}\alpha_{i,j}=\lambda_i,\ i=1,\ldots, k-1 \right\}.
\end{eqnarray}
When $\lambda = (0,\ldots, 0)$ we write just $\mathcal{H}_{k,n}$.

Denote by $\hat{\mathcal{R}}^{(k)}_{n-k}$ the root polytope, the convex hull of the linear functions $\gamma_J$,
$$\hat{\mathcal{R}}^{(k)}_{n-k} = \text{convex hull}\left(\{0\}\cup \left\{\gamma_J: J \in \binom{\lbrack n\rbrack}{k} \right\}\right)$$
and by $\mathcal{R}^{(k)}_{n-k}$ its dimension $(k-1)(n-k-1)$ projection, which is the convex hull of the restricted functions
$$\mathcal{R}^{(k)}_{n-k} = \text{convex hull}\left\{\gamma_J\vert_{\mathcal{H}_{k,n}}: J \in \binom{\lbrack n\rbrack}{k}^{nf} \right\}.$$

Our main result for the PK polytope is the explicit set of (irredundant) facet inequalities.
\begin{thm}
	We have
\begin{eqnarray}\label{eqn: facet inequalities intro}
	\Pi^{(k)}_{n-k} & = & \left\{(\alpha) \in \mathbb{R}^{(k-1)\times (n-k)}: \sum_{j=1}^{n-k}\alpha_{i,j}=0;\  \gamma_J(\alpha)+1\ge 0,\ J \in \binom{\lbrack n\rbrack}{k}^{nf}\right\}
\end{eqnarray}
and each inequality defines a facet.
\end{thm}

In particular it turns out that $\Pi^{(k)}_{n-k}$ lives in the intersection of a codimension $k-1$ subspace of $\mathbb{R}^{(k-1)\times(n-k)}$ with the cube $\lbrack -1,2\rbrack^{(k-1)\times (n-k)}$.

Our main result for the root polytope $\mathcal{R}^{(k)}_{n-k}$ is that it admits a triangulation which is governed by the complex of pairwise noncrossing collections of $k$-element subsets $\mathbf{NC}_{k,n}$, see Definition \ref{defn: noncrossing}.  This generalizes the result due \cite{GelfanGraevPostnikov1997}, for the triangulation of the root polytope, the convex hull of the positive roots $e_i-e_j$ together with the origin, using certain alternating trees.  See also \cite{Postnikov2006}, and \cite{Meszaros2011}.

For any collection $\mathcal{J}  = \{J_1,\ldots, J_{m}\}\in \mathbf{NC}_{k,n}$ of nonfrozen subsets $J_i$, define 
\begin{eqnarray*}
	\lbrack \mathcal{J}\rbrack & = & \text{Convex hull}(\left\{0,\gamma_{J_1},\ldots, \gamma_{J_m}\right\})
\end{eqnarray*}
\begin{thm}
	The set of simplices $\left\{\lbrack \mathcal{J}\rbrack: \mathcal{J} \in \mathbf{NC}_{k,n} \right\}$ defines a flag, unimodular triangulation\footnote{A triangulation is flag if it is the clique complex of its 1-skeleton: every complete subgraph of the 1-skeleton is the 1-skeleton of a simplex in the triangulation.  A triangulation is unimodular if every (maximal dimension) simplex has the same volume $\frac{1}{d!}$, where $d$ is the dimension.} of $\mathcal{R}^{(k)}_{n-k}$.
\end{thm}

We formulate a conjectural formula for the \textit{resolved} biadjoint scalar $\hat{m}^{(3),NC}_n(\mathbb{I},\mathbb{I})$ which also uses the noncrossing complex.

For any $1\le i<j<k\le n$, define a compound determinant
$$A_{ijk} =  \det\left((u_1\times u_2)\times (u_i\times u_{i+1}),u_j,(u_{j+1}\times u_{j+2})\times (u_{k}\times u_1)  \right),$$
for a given matrix with columns $u_1,\ldots, u_n \in \mathbb{C}^3$.
 
For each $2\le j<k\le n$ define 
$$\hat{p}_{1jk} = p_{1jk},$$
while for $2\le i<j<k\le n$ define the \textit{resolved} minor
$$\hat{p}_{ijk} = \frac{A_{ijk}}{p_{1,2,i+1}p_{1,j+1,j+2}}$$
and the \textit{resolved} potential function 
$$\hat{\mathcal{S}}_{3,n} = \sum_{1\le i<j<k \le n}\log(\hat{p}_{ijk})s_{ijk}$$
Finally, denote by $\hat{m}^{(3),NC}_n(\mathbb{I},\mathbb{I})$ the \textit{resolved} amplitude
$$\hat{m}^{(3),NC}_n(\mathbb{I},\mathbb{I}) = \sum_{(x)\in \text{critical points}(\hat{\mathcal{S}}_{3,n}(s))}\frac{1}{\det'\Phi(x)} \prod_{j=1}^n\left(\frac{1}{\hat{p}_{j,j+1,j+2}(x)}\right)^2,$$
where $\det'\Phi$ is the so-called reduced Hessian determinant (see \cite[Equation 2.4]{CEGM2019} for details) of $\hat{\mathcal{S}}_{3,n}$.  See also \cite{CachazoMasonSkinner2014} for general results, and \cite{SturmfelsTelen} for a reformulation using the Euler operator.

\begin{conjecture}
	There exists a set of linear functions $\hat{\eta}_J(s)$ such that 
\begin{eqnarray}\label{CEGM biadjoint scalar defin critical point intro}
	\hat{m}^{(3),NC}_n(\mathbb{I},\mathbb{I}) = \sum_{\{J_1,\ldots, J_{(2)(n-4)}\}\in \mathbf{NC}_{3,n}}\prod_{j=1}^{(2)(n-4)}\frac{1}{\hat{\eta}_{J_j}}.
\end{eqnarray}
\end{conjecture}
See Appendix \ref{sec: numerical amplitude evaluation} for a detailed, numerical calculation for large prime-number kinematics in the case $(k,n) = (3,6)$.

A formula for the \textit{resolved} planar kinematic functions $\hat{\eta}_{J}$ is conjectured in Equation \eqref{eq: kinematic shift 3n}.

Next, we define variables, generalized cross-ratios $u_{ijk}$, which can be extended via the positive parameterization (see Section \ref{sec: positive parameterization} for a reinterpretation which is convenient for our purposes) to provide a combinatorial characterization of a compactification of the configuration space of $n$ generic points in $\mathbb{CP}^{2}$ modulo projective equivalence; conjecturally the strata in the compactification are in bijection with pairwise noncrossing collections of $3$-element subsets of $\{1,\ldots, n\}$.  It is now natural to propose analogous constructions for all $k\ge 2$, which we do, in Section \ref{sec: generalized worldsheet associahedron}.

Planar face polynomials $\tau_J$ are defined in Section \ref{sec: noncrossing complex}.  For each $i=2,3,\ldots, n$, define $\tau_{1,i} = 1.$
With integers $s\ge 0$ and $m\ge 2$, and an $m$-element subset $J = \{j_1,\ldots, j_m\}$ of $\{2,3,\ldots, n\}$, set
$$\tau^{(s)}_{J} = \sum_{\{A \in \mathcal{I}_J: a_1\le a_2\le \cdots \le a_m\}}x_{1+s,a_1}x_{2+s,a_2}\cdots x_{m+s,a_m},$$
where $A = \{a_1,\ldots, a_m\}$, and 
$$\mathcal{I}_J = \lbrack j_1-1,j_2-1\rbrack \times \lbrack j_2-2,j_3-2\rbrack \times \cdots \lbrack j_{m-1}-(m-1),j_m-m\rbrack.$$
We are now ready to define the face polynomials $\tau_J$.
\begin{defn}
	For any $I \in \binom{\lbrack n\rbrack}{k}$, let $s\ge 0$ be such that $I=\lbrack 1,s\rbrack \cup J$, where $J = I\setminus \lbrack 1,s\rbrack$.  Now define 
	$$\tau_{I} = \tau^{(s)}_{j_1-s,j_2-s,\ldots, j_m-s}.$$
	Here $\lbrack 1,0\rbrack = \emptyset$ is the empty set.
\end{defn}
Note that the $\tau_J$ polynomials are not irreducible in general.
\begin{figure}[h!]
	\centering
	\includegraphics[width=0.7\linewidth]{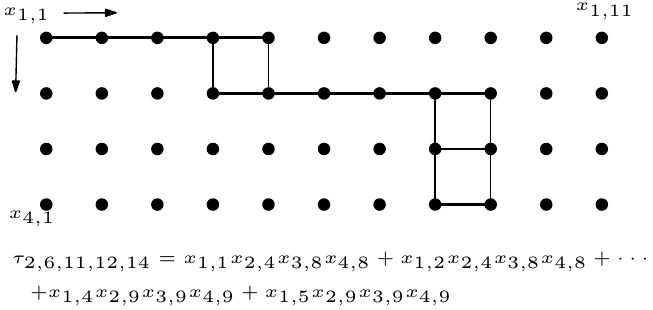}
	\caption{The Staircase: the (irreducible) Planar Face polynomial $\tau_{2,6,11,12,14}$.}
\end{figure}

Then for instance for $k=3$ we have $\tau_{1,2,j}=1$ for all $j=3,\ldots, n$.  For $\{1,j,k\}$ such that $3\le j<k \le n$, then
\begin{eqnarray*}
	\tau_{1,j,k} & := & \sum_{a \in \lbrack j-1,k-2\rbrack} x_{2,a}.
\end{eqnarray*}
Whenever $\{i,j,k\}$ satisfies $2\le i\le j<k\le n$, then
\begin{eqnarray*}
	\tau_{i,j,k} & := & \sum_{\{(a,b)\in \lbrack i-1,j-1\rbrack \times \lbrack j-2,k-3\rbrack,\ a\le b\}} x_{1,a}x_{2,b}
\end{eqnarray*}
We finally define the u-variables for $k=3$ and then formulate the main conjecture.

For any $\{i,j,k\} \subset \{1,\ldots, n\}$ that is not one of the $n$ cyclic intervals $\{j,j+1,j+2\}$, define
\begin{eqnarray*}
	u_{i,j,k}(x) = \begin{cases}
		\frac{\tau_{i,n-1,n}}{\tau_{i-1,n-1,n}}, & (i,j,k) = (i,n-1,n)\\
		\frac{\tau_{i+1,j,k}\tau_{i,j+1,j+2}}{\tau_{i,j,k}\tau_{i+1,j+1,j+2}}, & j+1<k,\ \ k=n,\\
		\frac{\tau_{i+1,j,k}\tau_{i,j,k+1}}{\tau_{i,j,k}\tau_{i+1,j,k+1}}, & k <n.
	\end{cases}
\end{eqnarray*}
When it is clear from context, we will abbreviate $u_{ijk}(x)$ with just $u_{ijk}$.
\begin{conjecture}\label{conjecture: binary equations intro}
	For any $J\in \binom{\lbrack n\rbrack}{3}^{nf}$, we have
	\begin{eqnarray}\label{eq: binary relations intro}
		u_J & =& 1 - \prod_{\{I:\ (I,J) \not\in\mathbf{NC}_{3,n}\}}u^{c_{I,J}}_I,
	\end{eqnarray}
	where for any crossing pair $(i_1i_2i_3,j_1j_2j_3) \not\in \mathbf{NC}_{3,n}$,
	\begin{eqnarray}\label{eq: binary equations exponents intro}
		c_{(i_1,i_2,i_3),(j_1,j_2,j_3)} = \begin{cases}
			2 & \text{if } i_1<j_1<i_2<j_2<i_3<j_3\text{ or } j_1<i_1<j_2<i_2<j_3<i_3\\
			1 & \text{otherwise}.
		\end{cases}
	\end{eqnarray}
\end{conjecture}

We have confirmed explicitly that Equations \eqref{eq: binary relations intro} and \eqref{eq: binary equations exponents intro} hold for u-variables of type $(3,n)$ for all $n\le 15$.  For $k=4$, the u-variables are defined explicitly in Equation \eqref{eqn: u variables 4n}.

In Definition \ref{defn: generalized worldsheet conclusion} (reproduced below) we formulate the all-(k,n) definition of the generalized worldsheet.  It is not difficult to see that it recovers the case $k=3$ above, as well as the dihedral compactification studied in \cite{FrancisBrownDihedral} and \cite{Worldsheet2017} in the case $k=2$.  These all provide examples of positive geometries \cite{PositiveGeometries}.  We also conjecture a parameterization of the solution space, given explicitly in the cases $3,4$ with $k\ge 5$ straightforward to infer.  It would be very interesting to find a combinatorial proof that our parameterization solves the binary equations below!

For any $I,J\in \binom{n}{k}^{nf}$, the \textit{compatibility degree} $c_{I,J}$ is the number of violations of the noncrossing condition in the pair $(I,J)$.  Let us be more precise.
\begin{defn}\label{defn: compatibility degree intro}
	Given $I,J\in \binom{n}{k}^{nf}$, then the compatibility degree $c_{I,J}\ge 0$ is the number of pairs $(\{i_1,i_2\},\{j_1,j_2\})$ with $\{i_a,i_b\} \subseteq I$ and $\{j_c,j_d\} \subseteq J$, 
	such that $(i_a,i_b),(j_c,j_d)$ is not weakly separated and we are in either of the following situations: $b=a+1$ and $d=c+1$ or $i_\ell = j_\ell$ for each $\ell = a+1,a+2,\ldots, b-1$.
\end{defn}

\begin{defn}\label{defn: generalized worldsheet conclusion intro}
	The generalized worldsheet associahedron $\mathcal{W}^+_{k,n}$ is the set of all points $(u_J) \in \lbrack 0,1\rbrack^{\binom{n}{k}-n}$ such that the following set of equations hold: for each $J\in \binom{\lbrack n\rbrack}{k}^{nf}$, one equation
	\begin{eqnarray}\label{eq: binary relations worldsheet associahedron defn All kn intro}
		u_J & =& 1 - \prod_{\{I:\ (I,J) \not\in\mathbf{NC}_{k,n}\}}u^{c_{I,J}}_I,
	\end{eqnarray}
	where $c_{I,J}$ is the compatibility degree from Definition \ref{defn: compatibility degree}.
\end{defn}
\begin{conjecture}\label{conjecture: binary relations All kn intro}
	Equations \eqref{eq: binary relations worldsheet associahedron defn All kn} generate the ideal of relations among the planar face ratios $u_J(\tau)$.
\end{conjecture}
Finally, Root Kinematics, and the Root Kinematics potential function $\mathcal{S}^{(Rt)}_{k,n}$, are formulated in Definition \ref{def: Root Kinematics}, where it is shown that, in a precise sense, it induces a \textit{non-recursive} construction of the positive parameterization of $G(k,n)\slash (\mathbb{C}^\ast)^n$, as the solution to a differential equation; however it requires a nontrivial amount of machinery to do properly and it would not make sense to repeat the discussion here in the Introduction.  Nonetheless we can extract the key piece: $\mathcal{S}^{(Rt)}_{k,n}$ can be reconstituted from its value on the positive parameterization\footnote{We use the notation BCFW, after the recursion of Britto-Cachazo-Feng-Witten \cite{BCFW2005}; see \cite{PositiveGrassmannian} for further discussion.} of the nonnegative Grassmannian \cite{PostnikovGrassmannian,TalaskaWilliams} (Section \ref{sec: positive parameterization}).  This value is:
	\begin{eqnarray}
	\mathcal{S}^{(Rt)}_{k,n}\big\vert_{BCFW} & = & \sum_{(i,j) \in \lbrack 1,k-1\rbrack \times \lbrack 1,n-k\rbrack} \log\left(\frac{x_{i,j}}{\sum_{\ell=1}^{n-k}x_{i,\ell}}\right)\alpha_{i,j},
\end{eqnarray}
where $x_{i,j}$ are homogeneous coordinates on $(\mathbb{CP}^{n-k-1})^{\times(k-1)}$.

\subsection{Structure of the rest of the paper}\label{sec: structure of paper}
The paper is structured as follows.
\begin{itemize}
	\item Section \ref{sec: notation and conventions} contains key notation and conventions, as well as an extended discussion of motivation.
	\item Sections \ref{sec: root polytope triangulation} and \ref{sec: triangulation} contain our main result for the higher root polytopes $\mathcal{R}^{(k)}_{n-k}$, namely that they admit a flag unimodular triangulation into simplices in bijection with noncrossing collections in $\mathbf{NC}_{k,n}$.
	\item Section \ref{sec: regular subdivisions root polytope} studies the use of the planar kinematic invariants $\eta_J$ as height functions over the vertices of root polytopes.  For $k=2$ it works perfectly in that it induces the noncrossing triangulation; however for $k\ge 3$ some subtle issues emerge.
	\item Section \ref{sec: noncrossing complex} defines the Planar Kinematics (PK) associahedron $\mathbb{K}^{(k)}_{n-k}$, of which the PK polytope is a Minkowski subsummand.
	\item Section \ref{sec: positive parameterization} reinterprets the matrix entries of the positive parameterization in terms of Newton polytopes.
	\item Section \ref{sec: noncrossing lattice rays} builds on the earlier noncrossing triangulation of $\mathcal{R}^{(k)}_{n-k}$ to derive a noncrossing representation of the rays of the positive tropical Grassmannian.  In particular, a notion of a noncrossing degree is defined.
	\item Section \ref{sec: minors resolved delta polynomials} formulates the main ingredients for the $k=3$ instance of the generalized worldsheet associahedron, $\mathcal{W}^+_{3,n}$, as a compactification of the configuration space of $n$ generic points in $\mathbb{CP}^{2}$.  
	\item Section \ref{sec:facet PK polytope} brings our PK polytope story to a close, with a complete description of its facets.  The proof invokes the noncrossing expansion coming form the complete simplicial fan in Theorem \ref{thm: noncrossing subdivision Wkn}.
	\item In Section \ref{sec: generalized worldsheet associahedron} we define the generalized worldsheet associahedron $\mathcal{W}^+_{k,n}$ for arbitrary $(k,n)$, in terms of an explicit set of binary equations, with exponents involving a certain compatibility degree for the noncrossing complex.
	\item Appendix \ref{sec: blades} reviews material about blades from previous work.
	\item Appendix \ref{sec: numerical amplitude evaluation} includes a numerical evaluation of the \textit{resolved} amplitude $\hat{m}^{(3),NC}_n$ using the generalized scattering equations \cite{CEGM2019}.  This is provides the first consistency check against our formulation of the generalized worldsheet associahedron $\mathcal{W}^+_{3,n}$.
	\item Appendix \ref{sec: u-variables} formulates the $k=3$ and $k=4$ planar face ratios and lays the groundwork for the computation of noncrossing binary identities for all $k$.  In this way, we define for each $(k,n)$ a generalized worldsheet associahedron which specializes to the dihedrally invariant partial compactification of the configuration space of $n$ distinct points in $\mathbb{CP}^1$, see for instance \cite{KobaNielsen,RobertsPh.D., Brown2006,Worldsheet2017}.
\end{itemize}

	\subsection{Notation and Key Definitions}\label{sec: notation and conventions}

Let us first fix some notation.  

Set $\lbrack n\rbrack = \{1,\ldots, n\}$ and denote by $\binom{\lbrack n\rbrack}{k}$ the set of $k$-element subsets of $\lbrack n\rbrack$.  Call a subset $J$ \textit{frozen} if its indices form a cyclic interval with respect to the cyclic order $(12\cdots n)$; otherwise it is nonfrozen.  Let $\binom{\lbrack n\rbrack}{k}^{nf}$ be the set of all nonfrozen $k$-element subsets of $\{1,\ldots, n\}$.

The \textit{hypersimplex} $\Delta_{k,n}$ is the convex polytope 
$$\Delta_{k,n} = \left\{x \in \lbrack 0,1\rbrack^n: \sum_{j=1}^n x_j=k\right\}.$$
It is the convex hull of the set of all points in the unit cube $\lbrack 0,1\rbrack^n$ with $k$ ones and $n-k$ zeros.

Given subsets $I,J\in \binom{\lbrack n\rbrack}{k}$, write $I\prec J$ if $I$ is lexicographically smaller than $J$.

The weak separation condition in Definition \ref{defn: noncrossing} is due to Leclerc and Zelevinsky in \cite{LeclercZelevinsky}; the formulation of the noncrossing criterion given here is our interpretation of the definition given in \cite{GrassmannAssociahedron}.

\begin{defn}[\cite{GrassmannAssociahedron}]\label{defn: noncrossing}
	A pair $I,J \in \binom{\lbrack n\rbrack}{k}$ is said to be \textit{weakly separated}, with respect to the cyclic order $(1,2,\ldots, n)$, provided that the coordinates in the difference  $e_I-e_J$ of vertices $e_I,e_J\in \Delta_{k,n}$ does not contain the pattern\footnote{In the context of \cite{Early19WeakSeparationMatroidSubdivision}, the sign pattern avoidance has the direct interpretation that the arrangement of $\beta_I,\beta_J$ of the blade $\beta=((1,2,\ldots, n))$ on the vertices $e_I,e_J\in \Delta_{k,n}$, respectively, cuts any octahedral face of $\Delta_{k,n}$ at most once, and consequently the subdivision induced in $\Delta_{k,n}$ is matroidal, and in particular positroidal.} $e_a-e_b+e_c-e_d$ for $a<b<c<d$, up to cyclic rotation.
	
	A pair $I,J \in \binom{\lbrack n\rbrack}{k}$ of $k$-element subsets of $\{1,\ldots, n\}$ is said to be \textit{non-crossing}, with respect to the linear order $1<2<\cdots <n$, provided that for each $1\le a<b\le k$, then either
	\begin{enumerate}
		\item The pair $\{\{i_a,i_{a+1},\ldots, i_{b}\},\{j_a,j_{a+1},\ldots, j_b\}\}$ is weakly separated, or 
		\item The interiors of the respective intervals do not coincide, that is we have
		$$\{i_{a+1},\ldots, i_{b-1}\} \not= \{j_{a+1},\ldots, j_{b-1}\}.$$
	\end{enumerate}
\end{defn}
Clearly there is some redundancy above that could be eliminated, but for our purposes this formulation is preferred.

Denote by $\mathbf{WS}_{k,n}$ the poset of all collections of pairwise weakly separated \textit{nonfrozen} $k$-element subsets, ordered by inclusion; according to the purity conjecture for weakly separated collections states that, among these, the maximal (by inclusion) collections of these each have exactly $(k-1)(n-k-1)$ $k$-element subsets.  The purity conjecture was proved independently in \cite{Danilov Purity}, \cite{Purity OhPostnikovSpeyer}.

Denote by $\mathbf{NC}_{k,n}$ the poset of all collections of pairwise non-crossing \textit{nonfrozen} k-element subsets, ordered by inclusion; again, it is known that the maximal (by inclusion) collections of these each have exactly $(k-1)(n-k-1)$ $k$-element subsets, see \cite{PKPS}, \cite{GrassmannAssociahedron}.

We emphasize that in our definition of the noncrossing complex $\mathbf{NC}_{k,n}$, the $n$ frozen subsets that consist of a single cyclic interval are excluded; sometimes $\mathbf{NC}_{k,n}$ has been called the \textit{reduced} noncrossing complex, as in for instance \cite{GrassmannAssociahedron}.  For us, this reduction has in fact it has a physical origin, namely that the generalized biadjoint scalar \cite{CEGM2019} is a rational function of total degree $-(k-1)(n-k-1)$.  In particular, the planar kinematic invariants $\eta_J$ with $J$ a frozen subset are all identically zero on the kinematic space.

The kinematic space $\mathcal{K}(k,n)$ is a dimension $\binom{n}{k}-n$ subspace of $\mathbb{R}^{\binom{n}{k}}$, given by 
$$\mathcal{K}(k,n) = \left\{(s) \in \mathbb{R}^{\binom{n}{k}}: \sum_{J\ni a}s_J,\ a=1,\ldots, n\right\}.$$
Denote by $\mathcal{K}_D(k,n)$ the following planar cone inside the kinematic space:
\begin{eqnarray}\label{eq: upper kinematic space}
	\mathcal{K}_D(k,n) & = & \left\{(s) \in\mathcal{K}(k,n): s_J \le 0,\ J \in \binom{\lbrack n\rbrack}{k}^{nf},\ s_{j,j+1,\ldots, j+k-1} \ge 0\right\}.
\end{eqnarray}
Call a point $\mathcal{K}_D(k,n)$ \textit{interior} if all inequalities are strict.  This cone will be used in Section \ref{sec: regular subdivisions root polytope} in the construction of certain regular subdivisions of the root polytope $\mathcal{R}^{(k)}_{n-k}$.  The cells in the subdivision form what is called a polyhedral complex.

\begin{defn}
	A polyhedral complex $\mathcal{C}$ is a set of polyhedra, such that
	\begin{enumerate}
		\item Every face of a polyhedron from $\mathcal{C}$ is also in $\mathcal{C}$,
		\item The intersection of any two polyhedra $C_1,C_2 \in \mathcal{C}$ is a face of both $C_1$ and $C_2$.
	\end{enumerate}
	A polyhedral \textit{fan} is a polyhedral complex such that every polyhedron is a cone from the origin; it is simplicial if every polyhedron is simplicial.  A polyhedral fan $\mathcal{C}$ in some $\mathbb{R}^m$ is \textit{complete} if for any $v \in \mathbb{R}^m$, then there exists a unique cone $C\in \mathcal{C}$ such that $v$ is in the relative interior of $C$.
\end{defn}

\subsection{Motivation}\label{sec: intro II amplitudes}

Recent progress in the study of scattering amplitudes digs deep into structures in combinatorial and tropical geometry, revealing two closely related notions of the generalized biadjoint scalar amplitude \cite{CEGM2019} and stringy integrals \cite{AHL2019Stringy}, both of which -- on their common intersection -- are governed by degenerations of point configurations in complex projective space that are subject to a certain condition of \textit{planarity}.  The generalized biadjoint scalar amplitude is known to control only the leading order term in the so-called $\alpha'$-expansion\footnote{Here $\alpha'$ corresponds in the case $k=2$ to the so-called inverse string tension.} of a stringy integral, but does not suffer from convergence issues and can be calculated in somewhat greater generality.  

One immediate lesson is that it may be time to revisit of how physical processes are represented.  What should be the correct notion of the worldsheet in the generalization?  This is not completely obvious given what is now known as it involves selecting the right compactification of the configuration space $X(k,n)$ of $n$ generic points in $\mathbb{CP}^{k-1}$ modulo $PGL(k)$, where $k\ge 2$ is fixed and $X(2,n)$ gives, after a certain planar compactification, the worldsheet for the usual cubic scalar amplitude.  In this paper, we formulate the rules for a compactification whose strata are -- conjecturally -- in bijection with the set of complete subgraphs of the so-called noncrossing complex of $k$-element subsets $\mathbf{NC}_{k,n}$.

One drawback -- or, more optimistically, feature -- is that our compactification ultimately has to break a cyclic symmetry which was an essential feature for the biadjoint scalar partial amplitude $m^{(k=2)}_n$, arising from a certain color-order decomposition using structure constants of some gauge group $U(N)$.

Next, what does it mean for particles to go on-shell in these models which are governed by point configurations in higher dimensional projective spaces?  Is the on-shell condition the end of the story for particle interactions?  The answer to these questions is manifold and involves an intricate and growing network of connections between algebraic, combinatorial and tropical geometry, the scattering equations formalism and the formulation of stringy integrals.  In what follows, we shall refer to both the generalized biadjoint scalar and stringy integrals as instances of \textit{generalized scattering amplitudes}, though some fundamental physical questions remain to be explored before this designation can be made precise.

According to the blade model \cite{Early2019PlanarBasis,Early2020WeightedBladeArrangements} for generalized scattering amplitudes \cite{CEGM2019,AHL2019Stringy}, \textit{planarity} is an emergent property of a simple fundamental building block: the simplex, $x_1\le x_2\le \cdots \le x_n \le x_1+1$.  The key player is the \textit{normal fan} to this simplex, or more precisely the \textit{blade} $((1,2,\ldots, n))$ see \cite{EarlyBlades,OcneanuVideo}, which has a certain cyclic symmetry.  The prototypical nondegenerate blade $((1,2,\ldots, n))$ can be thought of as a discrete curvature; the blade generates physical events, such as when a collection of particles goes \textit{on-shell}.  At first sight the setup of the blade model may seem somewhat too simplistic to describe physical processes; but its flexibility becomes more apparent when it is applied in the context of the biadjoint scalar, using the scattering equations \cite{CHY2014A,CHY2014B}, and the subsequent generalization to higher Grassmannians and their compactified torus quotients, \cite{CEGM2019}.  Taking weighted arrangements of $((1,2,\ldots, n))$ on the vertices of certain convex polytopes, called hypersimplices $\Delta_{k,n}$ and imposing a compatibility condition on octahedral faces of $\Delta_{k,n}$ gives rise to \textit{matroidal weighted blade arrangements} \cite{Early2020WeightedBladeArrangements}, and among these one finds the singularities of these generalized amplitudes.

The blade $((1,2,\ldots, n))$ is a special kind of tropical hypersurface which induces a decomposition of an n-1 dimensional space into $n$ chambers, simplicial cones $C_1,\ldots, C_{n}$, where cyclically adjacent pairs intersect in a copy of the simple roots of $SL_{n-1}$, more precisely, the intersections of cones $C_j\cap C_{j+1}$ have the $n-2$ edges $e_{j+1}-e_{j+2},e_{j+2}-e_{j+3},\ldots, e_{j-2}-e_{j-1}$, compatibly with the cyclic order $(1,\ldots, n)$, and the set of all edges $e_i-e_{i+1}$ of the blade sums to zero as in Figure \ref{fig:bladesbijectionpathdashed intro}.

The relevant question here is \textit{how} planarity emerges: it emerges for generalized amplitudes through its selection of a preferred \textit{planar} basis of linear functions on the kinematic space,
$$\eta_{j_1\cdots j_k} = -\frac{1}{n}\sum_{I\in \binom{\lbrack n\rbrack}{k}}\min\{(e_{t+1}+2e_{t+2}+\cdots +(n-1)e_{t-1})\cdot (e_I-e_J): t=1,\ldots, n\}s_I,$$
which are identified with arrangements of the blade $((1,2,\ldots,n))$ on the vertices $e_J = \sum_{j\in J} e_j$ of $\Delta_{k,n}$.  Here $J = \{j_1,\ldots, j_k\}$, and the $s_I$ are coordinate functions on $\mathbb{R}^{\binom{n}{k}}$.  The planar basis is in duality with height functions which induce certain matroid subdivisions of the hypersimplex $\Delta_{k,n}$.  In case $k=2$, then one has the remarkable identity
$$\eta_{ij} = \sum_{i+1\le a<b\le j} s_{ab},$$
where $s_{ij}$ are coordinate functions on $\mathbb{R}^{\binom{n}{2}}$.  In the case of the cubic scalar theory, one often restricts to the subset of $\mathbb{R}^{\binom{n}{2}}$ where $s_{ij} = p_i\cdot p_j$ with respect to the Minkowski bilinear pairing on $n$ given points in momentum space $p_1,\ldots, p_n$, in which case $\eta_{ij}=0$ expresses the condition that a cyclically consecutive subset of particles goes \textit{on-shell} and we have
$$ (p_{i+1}+p_{i+2}+\cdots +p_j)^2 = \sum_{i+1\le a<b\le j} s_{ab}=0,$$
where we will following the usual convention in physics in saying that the $s_{ij}$ are called Mandelstam invariants\footnote{However, for the purposes of the amplitude one can simply take $s_{i,j}$ to be formal parameters, as we do in the rest of this paper.}.  Note that here as usual $p_j\cdot p_j=0$.
Due to momentum conservation, such singularities exhibit the usual two-fold symmetry
\begin{eqnarray}\label{eq: momentum symmetry}
	\left(p_{i+1}+p_{i+2}+\cdots p_{j}\right)^2 = \left(p_{j+1}+p_{j+2}+\cdots p_{i}\right)^2.
\end{eqnarray}

However, generalized amplitudes are constructed in \cite{CEGM2019} as a generalization of the so-called biadjoint scalar amplitude using the CHY scattering equations formalism \cite{CHY2014A}, now for all $k\ge 2$; now Mandelstam variables are a priori formal parameters $s_{i_1\ldots i_k}$ indexed by $k$ distinct, unordered indices, and right away one faces the perplexing situation of a singularity of the generalized amplitude that exhibits behaviors that are a priori unexpected from a physical point of view; the aim of this paper is to answer some of this challenge.

For instance, already for the $n=6$ point generalized biadjoint scalar amplitude $m^{(3)}(\alpha,\beta)$ introduced in \cite{CEGM2019}, where we fix the same cyclic order $\alpha = \beta = (1,2,\ldots, 6)$, one has the following \textit{three} manifestly different ways to write the same pole $\eta_{1,3,5}$:
\begin{eqnarray}\label{eq: cyclic symmetry 3-split}
	\eta_{135} & = & s_{123}+s_{126}+s_{136}+s_{234}+s_{235}+s_{236}\nonumber\\
	& = & s_{145}+s_{234}+s_{235}+s_{245}+s_{345}+s_{456}\\
	& = & s_{126}+s_{136}+s_{145}+s_{146}+s_{156}+s_{456},\nonumber
\end{eqnarray}
in contradistinction with the more familiar Equation \eqref{eq: momentum symmetry}, which exhibits a binary-type symmetry, suggesting the existence of rich, new structures that require exploration.
\begin{figure}[h!]
	\centering
	\includegraphics[width=1\linewidth]{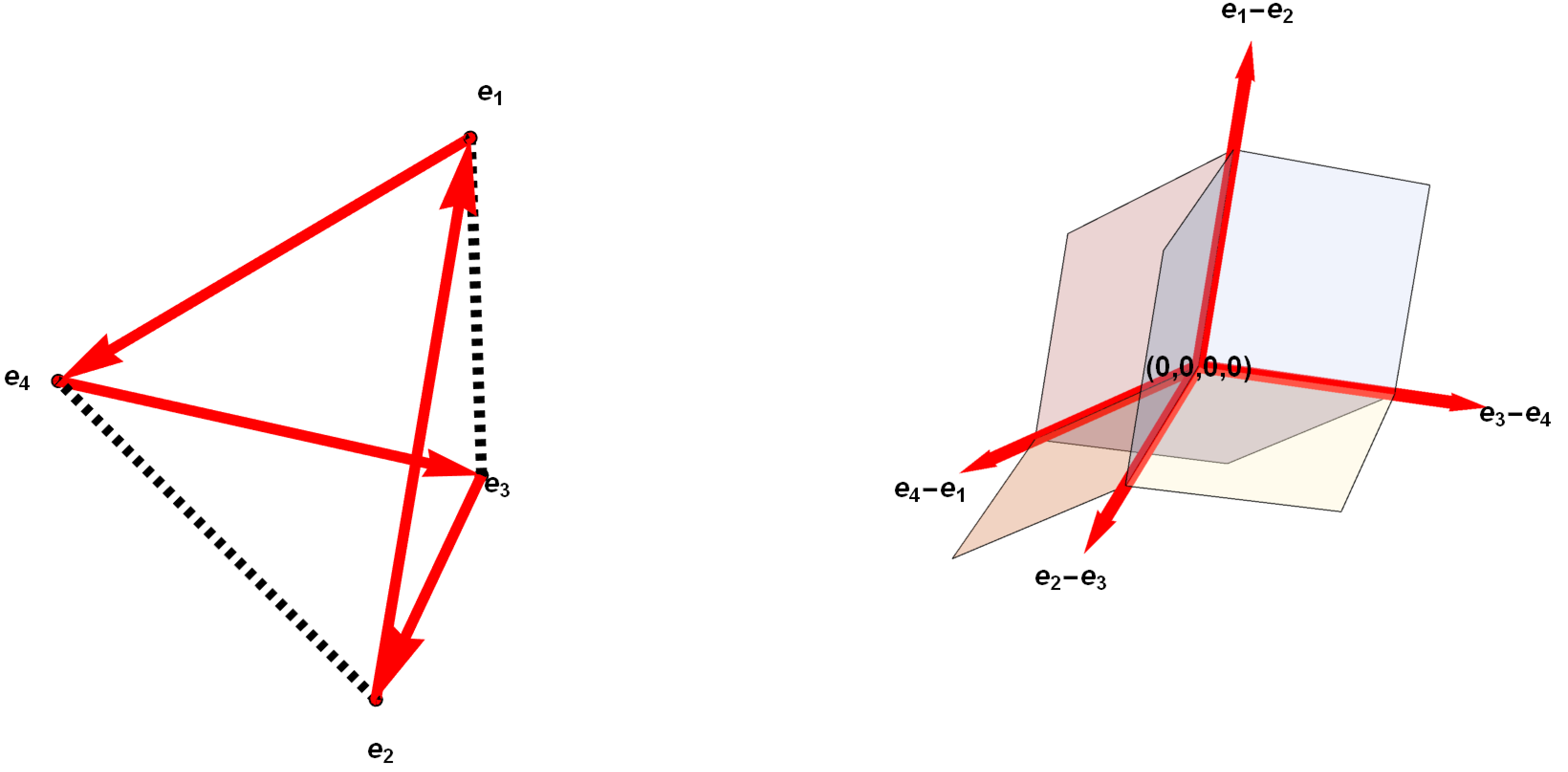}
	\caption{Blades induce a planar (cyclic) order.  Choose a cyclic path (left) on the edges of the standard coordinate simplex; construct the blade $((1,2,3,4))$ (right) from the edge directions.  The blade is also the 2-skeleton of the normal fan to the tetrahedron $x_1\le x_2\le x_3\le x_4 \le x_1+1$.}
	\label{fig:bladesbijectionpathdashed intro}
\end{figure}

But this identity becomes completely transparent and well-motivated in the blade model, in which case we can rewrite Equation \eqref{eq: cyclic symmetry 3-split} rather suggestively as
$$\eta_{((16_1 23_1 45_1))} = \eta_{((23_1 45_1 16_1))} = \eta_{((45_1 16_1 23_1))}.$$
This reflects the cyclic symmetry of the 3-split matroid subdivision of the hypersimplex $\Delta_{3,6}$ that is induced by the blade $((1,2,3,4,5,6))$, pinned to the vertex $e_1+e_3+e_5 = (1,0,1,0,1,0)$.  For details, see \cite{Early19WeakSeparationMatroidSubdivision}; from those general results, one has the following set theoretic identity for the intersections of the two blades $((1,2,3,4,5,6))_{e_{135}}$ and $((16_1 23_1 45_1))$ with the hypersimplex,
$$((1,2,3,4,5,6))_{e_{135}}\cap \Delta_{3,6} = ((16_1 23_1 45_1))\cap \Delta_{3,6},$$
which is analogous to the degree three cyclic symmetry around the vertices in the black tripods in Figures \ref{fig:hexagonbladearrangementfilled} and \ref{fig:weightpermtessellation2}.
\begin{figure}[h!]
	\centering
	\includegraphics[width=0.7\linewidth]{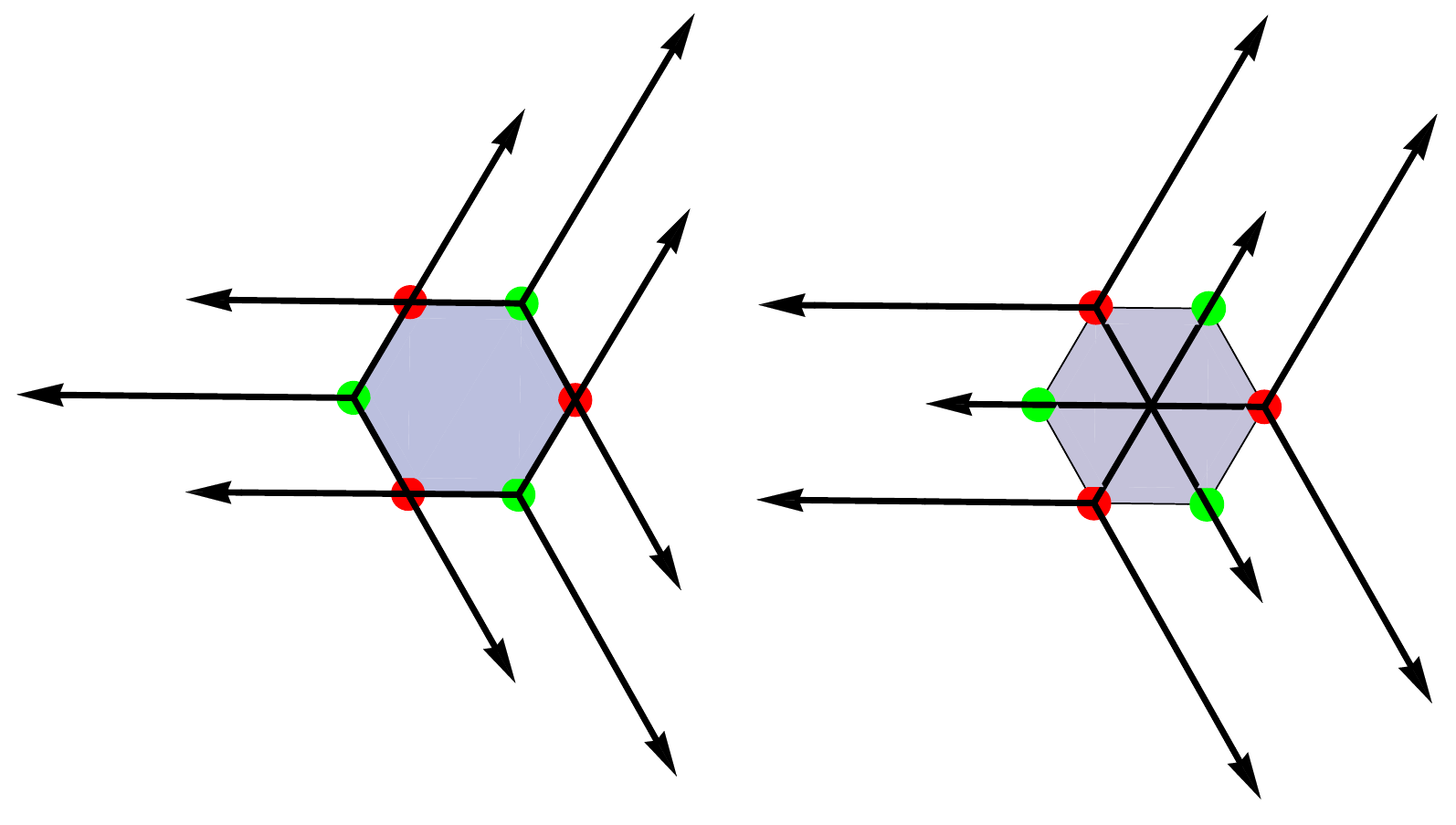}
	\caption{Two arrangements of three copies of the blade $((1,2,3,4,5,6))$ in the plane, visualized in the plane as projected from the hypersimplex $\Delta_{3,6}$ via the projection 
		$(x_1,\ldots, x_6) \mapsto (x_1+x_6,x_2+x_3,x_4+x_5)$, say.  The generalized amplitude singularity induced by the arrangement on the right contributes to the residue of $m^{(3)}(\mathbb{I}_6,\mathbb{I}_6)$ at the generalized on-shell condition $\eta_{135}=0$, giving 
		$$\text{Res}\big\vert_{\eta_{135}=0}\left(m^{(3)}(\mathbb{I}_6,\mathbb{I}_6)\right) = \frac{1}{\eta_{135}}\left(\frac{1}{\eta _{136}}+\frac{1}{\eta _{125}}\right) \left(\frac{1}{\eta _{235}}+\frac{1}{\eta _{134}}\right) \left(\frac{1}{\eta _{356}}+\frac{1}{\eta _{145}}\right).$$}
	\label{fig:hexagonbladearrangementfilled}
\end{figure}
\begin{figure}[h!]
	\centering
	\includegraphics[width=0.7\linewidth]{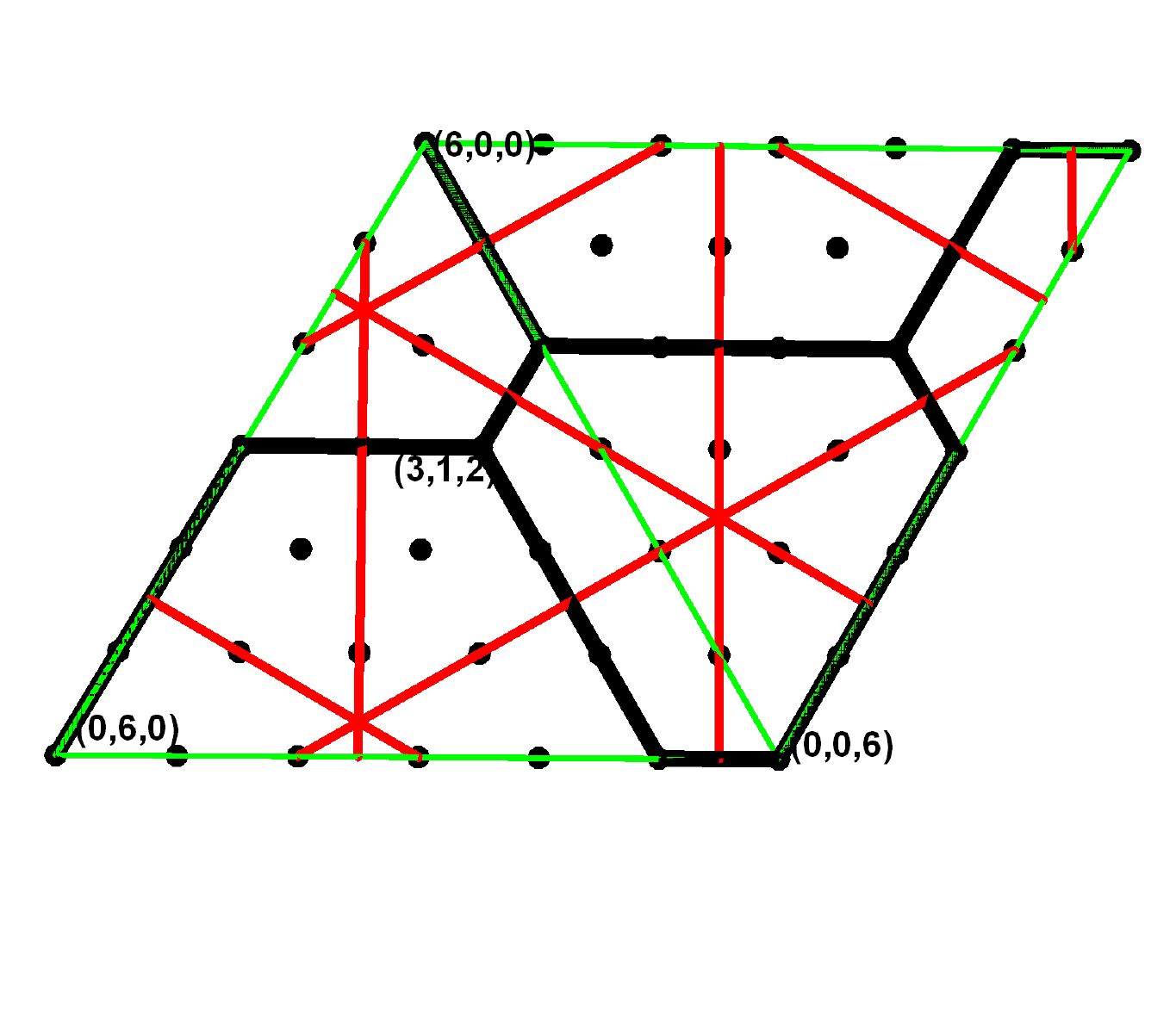}
	\caption{Tessellation of the plane arising from arranging the blades $((1,2,3))$ and $((1,3,2))$ on the root lattice of $SL_3$.  See \cite{Postnikov2006}, as well as \cite{OcneanuVideo}, where such structures have appeared in the context of representation theory.}
	\label{fig:weightpermtessellation2}
\end{figure}

It is interesting that Richard Feynman came rather near to the definition of blade arrangements more than 70 years ago, in his diagrammatic representation of perturbative Quantum Field Theory.  However, the key structures in combinatorial geometry had yet not been been developed formally.

It turns out that (tree-level) Feynman Diagrams are \textit{dual} to blade arrangements on the hypersimplex $\Delta_{2,n}$, not only combinatorially but realizably, in the sense of tropical geometry: in general, start with a subdivision of a polytope into a number of maximal cells.  Then place a vertex in the middle of each cell and connect with an edge any pair of vertices that is separated by an internal facet.  For certain polytopes and certain finest matroid subdivisions of them, the result of this process is a tree, and in the case that the polytope is the hypersimplex $\Delta_{2,n}$, then varying over all such finest regular subdivisions one obtains all $(2n-5)!!$ Feynman diagrams occurring in the cubic scalar at tree-level, and in particular all $C_{n-2}$ planar tree-level Feynman diagrams in the biadjoint scalar amplitude have a prominent role, where $C_{n-2} = 2,5,14,\ldots$, is the Catalan number: such trees are in bijection with a restricted set of subdivisions that are regular and \textit{positroidal} when the cyclic order is the standard one, $\alpha = (12\cdots n)$.

Feynman diagrams have a natural higher dimensional holographic-type generalization on faces of the hypersimplex.  Borges-Cachazo \cite{BC2019} discovered the prototype for Feynman diagrams for generalized scattering amplitudes, consisting of compatible \textit{collections} of Feynman diagrams on the $n$ facets of the hypersimplex $\Delta_{3,n}$.  This was developed further in \cite{CGUZ2019,Early2019PlanarBasis,GuevaraZhang2020}.  See also \cite{HeRenZhang2020}.

However this is only part of the picture: this description amounts to a holographic representation of a dual blade arrangement which is defined in the interior of the hypersimplex $\Delta_{k,n}$; but it is not a priori clear all of the information from the interior of the hypersimplex should be preserved in the process of passing to such a high codimension boundary, particularly for large $k$ and $n$.  Moreover, in the generalized Feynman diagram model there is no preferred notion of a standard basis for the dual kinematic space.  However, blade arrangements not only resolve this ambiguity, but they do so \textit{canonically}, in the sense that there is one such basis for each cyclic order.  Therefore, in the blade model planarity is \textit{emergent}.  A priori, for the generalized biadjoint scalar there is no systematic method to check when two singularities are compatible; also, there are just too many of them!  One has to run a complicated algorithm on the computer every time which becomes infeasible quickly beyond the simplest cases; this suggests again that there is more to the story and it begs for a simpler and more straightforward model to study first.  Indeed, searching for a scattering amplitude whose possible singularities are limited to the $\binom{n}{k}-n$ blades on the hypersimplex $\Delta_{k,n}$, then one is pulled inexorably to the construction in this paper of an amplitude and worldsheet which are so simple, belying the very rich combinatorial structures that they encode, that hand computations are again possible: one can check with a pencil and paper whether two poles are compatible, using a simple noncrossing rule on $k$-tuples of integers which specializes to the Steinmann compatibility relations for poles of the cubic scalar amplitude when $k=2$, and which is distinct from -- and in some respects much better behaved than -- the closely related combinatorial notion of weak separation.  The geometric content of the noncrossing rule is that the generalized root polytope $\mathcal{R}^{(k)}_{n-k}$, introduced in \cite{CE2020B}, should be given a very particular flag unimodular triangulation with $C^{(k)}_{n-k}$ maximal simplices, where $C^{(k)}_{n-k}$ is a richly structured integer in combinatorics, the $k$-dimensional Catalan number, O.E.I.S. number A060854 \cite{oeisMDCatalan}.

Let us conclude the introduction with a blade-theoretic calculation of the biadjoint scalar $m^{(2)}(\mathbb{I}_6,\mathbb{I}_6)$ (see \cite{CHY2014B} for the Definition) at a manifestly planar kinematic point.  Explicitly, we specify values for the basis of the kinematic space which is emergent from the blade model; otherwise in this paper we are concerned with the analogous situation for all generalized amplitudes with $k\ge 2$.

Later on, we will interpret the variables $\alpha_j$ as simple roots and then add a second index, $\alpha_{i,j}$ where $i=1,\ldots, k-1$ and $j=1,\ldots, n-k$.  Many of the enumerative results for the $k=2$ root lattice carries over to $k\ge 3$, including for instance the triangulation of the root polytope, the convex hull of the positive roots and the origin, into simplices which are in bijection with certain alternating trees, see \cite{GelfanGraevPostnikov1997,Postnikov2006,Meszaros2011}.

\begin{figure}
	\centering
	\includegraphics[width=0.7\linewidth]{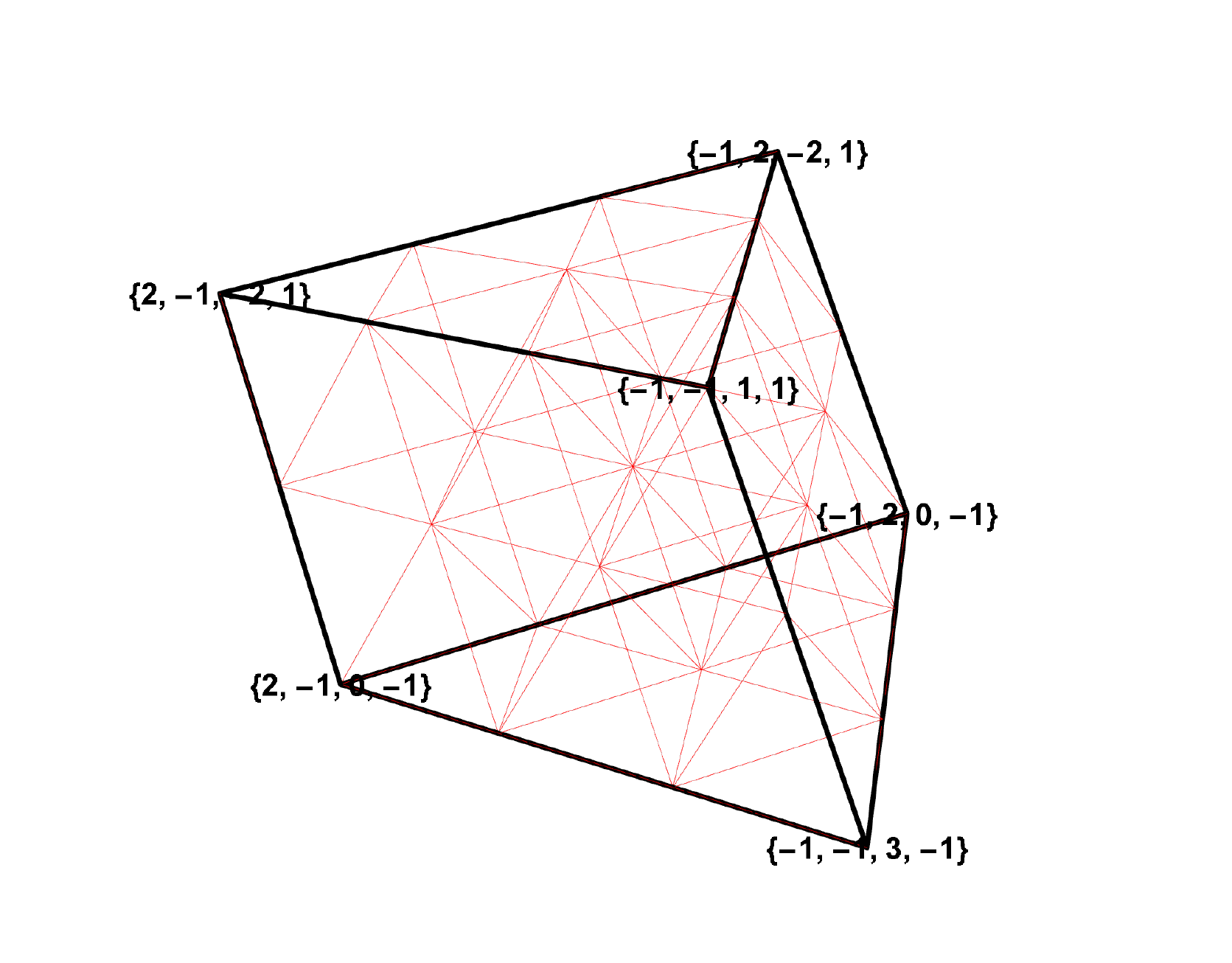}
	\caption{Example \ref{example: degenerate26 amplitude intro}: associahedron for the degenerate $(2,6)$ biadjoint scalar amplitude $m^{(2)}(\mathbb{I}_6,\mathbb{I}_6)$ corresponding to the kinematics in Equation \eqref{eq: kinematics example 26 intro}.  Feynman diagrams are identified with labeled (boundary) vertices.  Red lines are added to indicate compatibility with the root lattice, with directions $e_i-e_j$.  Also the Newton polytope for argument of the log in the potential function at the point $(\alpha)=(0,0,0,0)$ in Equation \eqref{eq: kinematics example 26 intro},
		$$\frac{\left(x_{1,1}+x_{1,2}+x_{1,3}\right)^3 \left(x_{1,3}+x_{1,4}\right)^2}{x_{1,1} x_{1,2} x_{1,3}^2 x_{1,4}}.$$}
	\label{fig:ampdeg26example intro}
\end{figure}

\begin{example}\label{example: degenerate26 amplitude intro}
	Consider the kinematic point $(s)$ specified\footnote{Here we are not giving directly the values of the $s_{ij}$'s, but they are easy to derive, by solving the $9$ equations here, together with 6 equations $\eta_{i\ i+1}=0$, which are automatically satisfied on the kinematic (sub)space $\mathcal{K}_{2,n}$ of $\mathbb{R}^{\binom{n}{2}}$.} by the values of the planar kinematic invariants $\eta_{ij}$,
	\begin{eqnarray}\label{eq: kinematics example 26 intro}
	\begin{array}{ccccccccc}
		\eta_{13} & \eta_{14} & \eta_{15} & \eta_{24} & \eta_{25} & \eta_{26} & \eta_{35} & \eta_{36} & \eta_{46} \\
		\alpha_1+1 & \alpha_{12} + 2 & \alpha_{123} + 1 & \alpha_{2} + 1 & \alpha_{23} + 3 & \alpha_{234} + 2 & \alpha_{3} + 2 & \alpha_{34} + 1 & \alpha_4+1
	\end{array}
	\end{eqnarray}
	where we denote $\alpha_J = \sum_{j\in J} \alpha_j$ for a subset $J$ of $\{1,2,3,4\}$, so for instance $\alpha_{23} = \alpha_2+\alpha_3$.  Here $\alpha_j$ are coordinate functions on $\mathbb{R}^4$.
	
	For the resulting potential function 
	$$\mathcal{S}_{2,6} = \log\left( \frac{x_{1,1}^{\alpha_1}x_{1,2}^{\alpha_2}x_{1,3}^{\alpha_3} x_{1,4}^{\alpha_4}}{(x_{1,1}+x_{1,2}+x_{1,3}+x_{1,4})^{\alpha_{1234}}}\right) + \log\left(\frac{x_{1,1} x_{1,2} x_{1,3}^2 x_{1,4}}{\left(x_{1,1}+x_{1,2}+x_{1,3}\right)^3 \left(x_{1,3}+x_{1,4}\right)^2}\right),$$
	applying for instance the CHY scattering equations formalism can be shown to give, when $\alpha_1+\alpha_2+\alpha_3+\alpha_4=0$, for the biadjoint scalar amplitude\footnote{Abbreviating $m^{(2)}_6 = m^{(2)}(\mathbb{I}_6,\mathbb{I}_6)$, where we abbreviate $\mathbb{I}_6 = (123456)$.  For our purposes, this equation may be taken as a definition.}
	\begin{eqnarray}
	m^{(2)}_6 &= & \frac{1}{\eta _{14} \eta _{15} \eta _{24}}+\frac{1}{\eta _{15} \eta _{24} \eta _{25}}+\frac{1}{\eta _{24} \eta _{25} \eta _{26}}+\frac{1}{\eta _{13} \eta _{15} \eta _{35}}+\frac{1}{\eta _{15} \eta _{25} \eta _{35}}+\frac{1}{\eta _{25} \eta _{26} \eta _{35}}\nonumber\\
	& + & \frac{1}{\eta _{13} \eta _{14} \eta _{15}}+\frac{1}{\eta _{26} \eta _{35} \eta _{36}}+\frac{1}{\eta _{13} \eta _{14} \eta _{46}}+\frac{1}{\eta _{14} \eta _{24} \eta _{46}}+\frac{1}{\eta _{24} \eta _{26} \eta _{46}}+\frac{1}{\eta _{13} \eta _{36} \eta _{46}}\\
	& + & \frac{1}{\eta _{26} \eta _{36} \eta _{46}}+\frac{1}{\eta _{13} \eta _{35} \eta _{36}}, \nonumber
\end{eqnarray}
	the following reduced Feynman diagram expansion around maximal collections of compatible poles,
	\begin{eqnarray}\label{eq: degenerate26amplitude intro}
		m^{(2)}_6& = & \frac{1}{\left(\alpha _1+1\right) \left(\alpha _4+1\right) \left(\alpha _{34}+1\right)}+\frac{1}{\left(\alpha _2+1\right) \left(\alpha _4+1\right) \left(\alpha _{34}+1\right)}+\frac{1}{\left(\alpha _1+1\right) \left(\alpha _2+1\right) \left(\alpha _4+1\right)}\nonumber\\
		& + & \frac{1}{\left(\alpha _1+1\right) \left(\alpha _{34}+1\right) \left(\alpha _{123}+1\right)}+\frac{1}{\left(\alpha _2+1\right) \left(\alpha _{34}+1\right) \left(\alpha _{123}+1\right)}\\
		& + & \frac{1}{\left(\alpha _1+1\right) \left(\alpha _2+1\right) \left(\alpha _{123}+1\right)}.\nonumber
	\end{eqnarray}
	See Figure \ref{fig:ampdeg26example intro} for the polytope cut out in the $\alpha$-space by the requirement that all poles $\eta_{ij}$ must be nonnegative.  Then the six terms correspond directly to the six vertices of the polytope; the function itself is positive on the interior of the polytope, with poles on the facets.
	
\end{example}

In the next section, we state and prove our results for the lattice of points in a subspace $\mathcal{H}_{k,n}$ (see Equation \eqref{eq:Hkn}) rather than the lattice of linear functions in the dual space, keeping in mind that the two are naturally (if not canonically) isomorphic.

\section{Kinematic Space and Planar Basis}
Fix integers $(k,n)$ such that $1\le k\le n-1$.

Recall the notation $\binom{\lbrack n\rbrack}{k}$ for the set of $k$-element subsets of the set $\lbrack n\rbrack = \{1,\ldots, n\}$, and denote by 
$$\binom{\lbrack n\rbrack}{k}^{nf} = \binom{\lbrack n\rbrack}{k} \setminus \left\{\{j,j+1,\ldots ,j+k-1\}: j=1,\ldots, n \right\}$$
the \textit{nonfrozen} $k$-element subsets.  Let $\{e^J: J\in \binom{\lbrack n\rbrack}{k}\}$ be the standard basis for $\mathbb{R}^{\binom{n}{k}}$.

The $k^\text{th}$ hypersimplex in $n$ variables is the $k^\text{th}$ integer cross-section of the unit cube $\lbrack 0,1\rbrack^n$,
$$\Delta_{k,n} = \left\{x\in \lbrack0,1 \rbrack^n: \sum x_j=k \right\}.$$
Henceforth we shall assume that $2\le k\le n-2$.

Recall that the \textit{lineality} space is the n-dimensional subspace
$$\left\{\sum_{J}x_J e^J: x\in\mathbb{R}^n \right\},$$
of $\mathbb{R}^{\binom{n}{k}}$, where we use the notation $x_J = \sum_{j\in J} x_j$.

Then the \textit{kinematic space} is the dimension $\binom{n}{k}-n$ subspace of $\mathbb{R}^{\binom{n}{k}}$, 
\begin{eqnarray}\label{eq: kinematic space}
	\mathcal{K}(k,n) = \left\{(s) \in \mathbb{R}^{\binom{n}{k}}: \sum_{J:\ J\ni j}s_J=0,\ j=1,\ldots, n \right\}.
\end{eqnarray}

Now for any $J\in \binom{\lbrack n\rbrack}{k}^{nf}$, define a linear functional on the kinematic space, or in more physical terminology, a planar kinematic invariant, $\eta_J:\mathcal{K}(k,n) \rightarrow \mathbb{R}$, by
\begin{eqnarray}\label{eq:planar basis element}
	\eta_J(s) & = & -\frac{1}{n}\sum_{I\in \binom{\lbrack n\rbrack}{k}}\min\{L_1(e_I-e_J),\ldots, L_n(e_I-e_J)\}s_I,
\end{eqnarray}
where 
$$L_j(x) = x_{j+1}+2x_{j+2} + \cdots +(n-1)x_{j-1}$$
for $j=1,\ldots, n$ are linear functions on $\mathbb{R}^n$.

Usually instead of $\eta_J(s)$ we write just $\eta_J$ with the understanding that $\eta_J$ is to be evaluated on points $(s) \in \mathcal{K}(k,n)$.

Then we have the property that if $J = \{i,i+1,\ldots, i+k-1\}$ is frozen, then since the graph of $\rho_J:\Delta_{k,n}\rightarrow \mathbb{R}$ does not bend over $\Delta_{k,n}$, it follows that $\eta_J$ is identically zero on $\mathcal{K}(k,n)$.  See \cite{Early2020WeightedBladeArrangements} for details.

A further computation proves linear independence for the set of $\eta_J$ where $J$ is nonfrozen, and we obtain Proposition \ref{prop:planar basis}. 

\begin{prop}[\cite{Early2019PlanarBasis,Early2020WeightedBladeArrangements}]\label{prop:planar basis}
	The set of linear functions
	$$\left\{\eta_J:\mathcal{K}(k,n) \rightarrow \mathbb{R}: J\in \binom{\lbrack n\rbrack}{k}^{nf} \right\}$$
	is a basis of the space of the dual kinematic space $\left(\mathcal{K}(k,n)\right)^\ast$, that is to say, it is a basis of the space of linear functions on $\mathcal{K}(k,n)$.
\end{prop}

\section{Root Polytopes and Triangulations}\label{sec: root polytope triangulation}
In this section, we introduce the generalized \textit{root polytope} $\mathcal{R}^{(k)}_{n-k}$, which is most conveniently defined in the space of linear functions on $\mathcal{H}_{k,n}$.  

Put $\alpha_{i,\lbrack a,b\rbrack} = \sum_{j=a}^b\alpha_{i,j}$.  Let $J = \{j_1,\ldots, j_k\}$, with order $1\le j_1<\cdots <j_k\le n$.  The generalized positive root $\gamma_J$ is the following sum of simple roots $\alpha_{i,j}$:
$$\gamma_J(\alpha) = \sum_{i=1}^{k-1} \alpha_{i,\lbrack j_i-(i-1),j_{i+1}-i-1\rbrack}.$$

\begin{figure}[h!]
	\centering
	\includegraphics[width=0.7\linewidth]{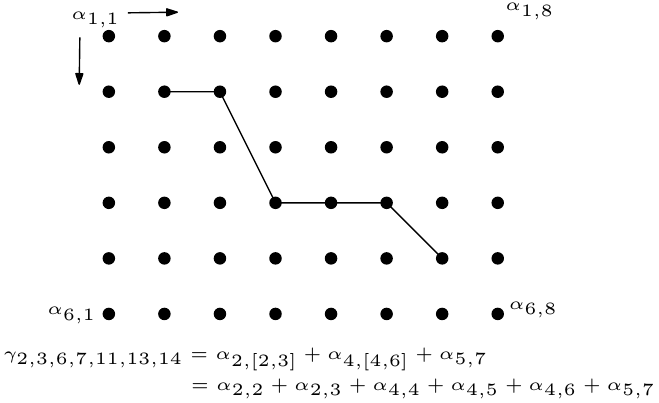}
	\caption{A generalized positive root.}
	\label{fig:higher-root-path}
\end{figure}

\begin{defn}
	The \textit{root polytope} $\mathcal{R}^{(k)}_{n-k}$ is the convex hull
$$\mathcal{R}^{(k)}_{n-k} = \text{convex hull}\left(\left\{\gamma_J: \mathcal{H}_{k,n} \rightarrow \mathbb{R}:  J \in \binom{\lbrack n\rbrack}{k}^{nf}\right\}\right).$$
We also denote by $\hat{\mathcal{R}}^{(k)}_{n-k}$ the convex hull
$$\hat{\mathcal{R}}^{(k)}_{n-k} = \text{convex hull}\left(\{0\}, \left\{\gamma_J: \mathbb{R}^{(k-1)\times (n-k)} \rightarrow \mathbb{R}:  J \in \binom{\lbrack n\rbrack}{k}^{nf}\right\}\right).$$
\end{defn}
Here the polytope $\mathcal{R}^{(k)}_{n-k}$ has dimension $(k-1)(n-k-1)$, while the polytope $\hat{\mathcal{R}}^{(k)}_{n-k}$ has dimension $(k-1)(n-k)$.  These polytopes may be of independent interest; however, in what follows we concentrate on the smaller-dimensional polytope $\mathcal{R}^{(k)}_{n-k}$.

We may decide to embed $\mathcal{R}^{(k)}_{n-k}$ into $\mathcal{H}_{k,n}$ by identifying the space of linear functions on $\mathcal{H}_{k,n}$ with $\mathcal{H}_{k,n}$ itself, with respect to the dual pairing $\alpha_{i,j}(e_{i',j'}) = \delta_{i,i'}\delta_{j,j'}$; then we may take $\mathcal{R}^{(k)}_{n-k}$ to be the projection of the polytope $\hat{\mathcal{R}}^{(k)}_{n-k}$.  To connect with the type $A$ root polytope, see \cite{GelfanGraevPostnikov1997,Postnikov2006,Meszaros2011}, the convex hull of the positive roots $e_{i}-e_j$ for $1\le i<j \le n$, together with the origin, define a projection $p_{k,n}: \mathbb{R}^{(k-1)\times (n-k)} \rightarrow \mathcal{H}_{k,n}$ by $p_{k,n}(e_{i,j}) = f_{i,j}$, where 
$$f_{i,j} = \begin{cases}
	e_{i,j} - e_{i,j+1}, & \ 1\le j\le n-k-1\\
	e_{i,n-k} - e_{i,1}, & \ j=n-k.
\end{cases}$$
Now for each $J \in \binom{\lbrack n\rbrack}{k}$, define
	\begin{eqnarray}\label{eqn: roots}
	\hat{v}_J & = & \sum_{\ell=1}^{k-1}e_{\ell,\lbrack j_\ell-(\ell-1),j_{\ell+1}-\ell-1\rbrack}\\
	v_J & = & \sum_{\ell=1}^{k-1}f_{\ell,\lbrack j_\ell-(\ell-1),j_{\ell+1}-\ell-1\rbrack}\nonumber
\end{eqnarray}
Evidently $v_J$ is zero exactly when $J$ is a single cyclic interval, and $\hat{v}_J$ is zero exactly when $J$ has is an interval of the form $\lbrack i,i+k-1\rbrack$ for $1\le i$ and $i+k-1 \le n$.

Generalized roots satisfy certain \textit{cubical} relations, of which the usual square move, or flip move, is a special case.  We formulate them on the hypersimplex $\Delta_{k,n}$.

For any integers $2\le m\le k$, then any intersection of affine hyperplanes in $\mathbb{R}^n$, say 
$$x_{i_1}+x_{j_1} = x_{i_2}+x_{j_2} = \cdots = x_{i_m}+x_{j_m} = 1,$$
intersects the hypersimplex $\Delta_{k,n}$; we specialize to the case when $1\le i_1<j_1<i_2<j_2<\cdots <i_m<j_m\le n$.  Taking any subset $L = \{\ell_1,\ldots, \ell_{k-m}\} \in \binom{\lbrack n\rbrack \setminus (I\cup J)}{k-m}$, then we define an $m$-dimensional cube in $\Delta_{k,n}$,
$$\mathcal{C}^{(L)}_{I,J} = \left\{x\in \Delta_{k,n}: x_{i_a}+x_{j_a} =1\text{ and } \ x_{\ell_b} = 1\text{ for } a=1,\ldots, m;\ b=1,\ldots, k-m\right\}.$$
\begin{prop}\label{prop:cubical relations 1}
	Given any pair of long diagonals of $\mathcal{C}^{(L)}_{I,J}$ specified by the pairs of vertices of $\Delta_{k,n}$, respectively $(e_{M_1\cup L},\ e_{M_2\cup L})$ and $(e_{M'_1\cup L},\ e_{M'_2\cup L})$, then we have that
	\begin{eqnarray}\label{eqn: noncrossing relation intro}
		v_{M_1 \cup L} + v_{M_2 \cup L} & = & v_{M'_1 \cup L} + v_{M'_2 \cup L}
	\end{eqnarray}
	and correspondingly 
	\begin{eqnarray*}
		\gamma_{M_1 \cup L} + \gamma_{M_2 \cup L} & = & \gamma_{M'_1 \cup L} + \gamma_{M'_2 \cup L}.
	\end{eqnarray*}
	Moreover, for each such cubical cell, there exists a unique pair of antipodal vertices of $\mathcal{C}^{(L)}_{I,J} $ that is noncrossing, namely $(e_{M''_1\cup L}, e_{M''_2 \cup L})$, where 
	\begin{eqnarray*}
		M''_1 & = & \{i_1,j_2,i_3,j_4,\ldots,i_{m-1} ,j_m\}\\
		M''_2 & = & \{j_1,i_2,j_3,i_4,\ldots,j_{m-1} ,i_m\}.
	\end{eqnarray*}
\end{prop}

For instance, recalling that $\gamma_{145} + \gamma_{236} = \alpha_{1,12} + \alpha_{2,23}$, it follows easily that
$$\gamma_{145} + \gamma_{236} = \gamma_{245} + \gamma_{136} = \gamma_{135} + \gamma_{246} = \gamma_{146} + \gamma_{235}.$$
Here $\{145,236\}$ is the only noncrossing (but not weakly separated) pair among these.

For instance, letting $I = \{1,3,5,7,9,11\}$ and $J = \{2,4,6,8,10,12\}$, then the pair 
$$(\{1,4,5,8,9,12\},\{2,3,6,7,10,11\})$$
for $n\ge 12$, is noncrossing and we have the corresponding four-term identity:
$$v_{1,3,5,7,9,11} + v_{2,4,6,8,10,12} = v_{1,4,5,8,9,12} + v_{2,3,6,7,10,11}.$$

Let us conclude with two instructive examples.  With $L = \{6,10\}$, then
$$\gamma_{1,3,5,6,10} + \gamma_{2,4,6,8,10} = \gamma_{1,4,5,6,10} + \gamma_{2,3,6,8,10}.$$
where the pair $\{\{1,4,5,6,10\},\{2,3,6,8,10\}\}$ is noncrossing.  This shows, in particular, that the set $L$ can be interlaced with the indices defining the cube $C^{(L)}_{I,J}$.  This works from a combinatorial point of view due to the requirement that weak separation can be violated when the respective interior intervals are not the same, but what if we didn't know this rule and had only the formal definition of the linear functions $\gamma_J$?  The explicit identities are:
\begin{eqnarray}
	\gamma_{1,3,5,6,10} + \gamma_{2,4,6,8,10} & = & (\alpha _{1,1}+\alpha _{2,2}+\alpha _{4,3}) + (\alpha _{1,2}+\alpha _{2,3}+\alpha _{3,4})\\
	\gamma_{1,4,5,6,10} + \gamma_{2,3,6,8,10} & = & (\alpha _{1,1}+\alpha _{1,2}+\alpha _{4,3}) + (\alpha _{2,2}+\alpha _{2,3}+\alpha _{3,4})
\end{eqnarray}
Finally, letting $I = \{1,3,5,7,9,11\}$ and $J = \{2,4,6,8,10,12\}$, then for the pair 
$$(\{1,4,5,8,9,12\},\{2,3,6,7,10,11\}) \in \mathbf{NC}_{6,12}\setminus \mathbf{WS}_{6,12}$$
we have the following four-term identity:
$$\gamma_{1,3,5,7,9,11} + \gamma_{2,4,6,8,10,12} = \gamma_{1,4,5,8,9,12} + \gamma_{2,3,6,7,10,11}.$$

As a special case of the cubical relations, we recover the usual flip move.

\begin{cor}\label{cor: four term identity gammaJ}
	The linear functions $\gamma_J$ (as well as the vectors $v_J$) satisfy the following additive relation: supposing that $I\cup \{a,b,c,d\} \in \binom{\lbrack n\rbrack}{k+2}$ with $a<b<c<d$, then we have
	\begin{eqnarray}\label{eq: 4-term octahedral}
		\gamma_{Iac} + \gamma_{Ibd} & = & \gamma_{Iad} + \gamma_{Ibc}.
	\end{eqnarray}
	
	Given $a<b<c<d$ such that $\{b\}\cup I\cup \{c\} \in \binom{\lbrack n\rbrack}{k}$ is a single subinterval of $\{1,2,\ldots, n\}$, then as a specialization of Equation \eqref{eq: 4-term octahedral} we obtain Equation \eqref{eq: Jacobi},
	\begin{eqnarray}\label{eq: Jacobi}
		\gamma_{Iac} + \gamma_{Ibd} = \gamma_{Iad}.
	\end{eqnarray}
\end{cor}

When $k=2$ (and $I = \emptyset$), then Equations \eqref{eq: 4-term octahedral} and  \eqref{eq: Jacobi} are equivalent (pulling back via the projection $p_{2,n}$) to the following somewhat more familiar relations among roots, respectively
$$(\alpha_{1,a}-\alpha_{1,c-1}) + (\alpha_{1,b}-\alpha_{1,d-1}) = (\alpha_{1,a}-\alpha_{1,d-1}) + (\alpha_{1,b}-\alpha_{1,c-1}).$$
and
$$(\alpha_{1,a}-\alpha_{1,c-1}) + (\alpha_{1,c-1}-\alpha_{1,d}) = \alpha_{1,a}-\alpha_{1,d}.$$

\section{Triangulation}\label{sec: triangulation}
Denote by $\mathcal{L}^{(k)}_{n-k} \subset \mathcal{H}_{k,n}$ the $(k-1)(n-k-1)$-dimensional integer lattice generated by the elements $f_{i,j} = p_{k,n}(e_{i,j})$, for $(i,j)\in \lbrack 1,k-1\rbrack \times \lbrack 1,n-k-1\rbrack$.  In particular, for our choice of projection $p_{k,n}$ then $\mathcal{L}^{(k)}_{n-k}$ is a Cartesian product of $k-1$ root lattices of type $SL_{n-k}$.

In what follows, we identify the $k-1$ lattices as $v^{(i)}_{j_1j_2} = v_{\lbrack j_1,j_1+i-1\rbrack \cup \lbrack j_2+k-1-i, j_2+k-2\rbrack}$.  For instance, with $(k,n) = (4,9)$, say, then
$$v^{(1)}_{36} = v_{3678} = f_{1,3} + f_{1,4},\ \ v^{(2)}_{36} = v_{3478} = f_{2,3} + f_{2,4},\ \ v^{(3)}_{36} = v_{3458} = f_{3,3} + f_{3,4}.$$
A key property of the elements $f_{i,j} \in \mathcal{H}_{k,n}$ is that for each $i=1,\ldots, k-1$, then $f_{i,1},\ldots, f_{i,n-k}$ span a subspace of dimension $n-k-1$, satisfying the single relation 
$$f_{i,1} + \cdots  + f_{i,n-k} = 0.$$
\begin{lem}
	For each given $i=1,\ldots, k-1$, then the embedding $\{j_1,j_2\} \mapsto \lbrack j_1,j_1+(i-1)\rbrack \cup \lbrack j_2,j_2+(k-(i-1))\rbrack$ has the following property:
	the pair $\{\{j_1,j_2\},\{j_3,j_4\}\}$ is noncrossing (and thus is an edge in $\mathbf{NC}_{2,n}$) if and only if the image
	$$\{\lbrack j_1,j_1+(i-1)\rbrack \cup \lbrack j_2,j_2+(k-(i-1))\rbrack,\lbrack j_3,j_3+(i-1)\rbrack \cup \lbrack j_4,j_4+(k-(i-1))\rbrack\}$$
	forms a noncrossing pair (and thus is an edge in $\mathbf{NC}_{k,n+(k-2)})$.
\end{lem}

According to Proposition \ref{prop: unimodular simplices noncrossing}, given any maximal noncrossing collection 
$$\mathcal{J}= \left\{J_1,\ldots, J_{(k-1)(n-k-1)}\right\}\in\mathbf{NC}_{k,n}$$
then the cone $\langle \mathcal{J}\rangle_+$, defined by 
$$\langle \mathcal{J}\rangle_+ = \left\{\sum_{j=1}^{(k-1)(n-k-1)}t_j v_{J_j}: t_j \ge0 \right\},$$
is simplicial and unimodular (so, the primitive integer generators of its rays, together with the origin, form the vertices of a simplex of volume $1/((k-1)(n-k-1))!$ where $(k-1)(n-k-1)$ is the dimension of the ambient space).

\begin{prop}\label{prop: unimodular simplices noncrossing}
	Given any maximal non-crossing collection $\mathcal{J} = \left\{J_1,\ldots, J_{(k-1)(n-k-1)}\right\} \in\mathbf{NC}_{k,n}$ of nonfrozen subsets $J_i \in \binom{\lbrack n\rbrack}{k}^{nf}$, then the polyhedral cone $\langle \mathcal{J}\rangle_+$ is simplicial of dimension $(k-1)(n-k-1)$.  Moreover, any pair of simplices of the form
	$$\text{Convex hull}\left(\{0\}\cup \left\{ v_J: J\in \mathcal{J}_1 \right\}\right),\ \text{Convex hull}\left(\{0\}\cup \left\{ v_J: J\in \mathcal{J}_2 \right\}\right)$$
	for $\mathcal{J}_1,\mathcal{J}_2 \in \mathbf{NC}_{k,n}$ maximal noncrossing collections, are related by a unique volume-preserving integer linear transformation.

\end{prop}

\begin{proof}
	We shall argue by induction on $k$.
	
	For the base step: both statements for $\mathcal{R}^{(2)}_{n-2}$ (with $n\ge 4$) translate directly to standard properties for the non-crossing triangulation \cite{GelfanGraevPostnikov1997,Postnikov2006} of the type $A_{n-2}$ root polytope into simplices which are put in bijection with certain alternating trees.  See also \cite{Meszaros2011} and the references therein.  The identification with positive roots is as follows: for any $1\le i<j \le n-1$ we have $e_i-e_j = v_{i,j+1}$.
	
	Now let $\mathcal{J} = \left\{J_1,\ldots, J_{(k-1)(n-k-1)}\right\} \in \mathbf{NC}_{k,n}$ be any maximal non-crossing collection.
	
	We will show that the set $\left\{v_J: J\in\mathcal{J} \right\}$
	is a \textit{lattice} basis, that is, any element in the ambient lattice $\mathcal{L}^{(k)}_{n-k}$
	is a linear combination of the given vertices $v_J \in \mathcal{R}^{(k)}_{n-k}$ for $J\in \mathcal{J}$ with integer coefficients.  Indeed, in this case, any pair of maximal simplices labeled by maximal non-crossing collections are related by a volume-preserving (integer-valued) linear transformation.

	We induct on $k$ using a partition of unity $\varphi_{k-2}+\varphi_{-1} = \mathbf{id}$.  
	
	Let $\varphi_{k-2}:\mathcal{H}_{k,n} \rightarrow \mathcal{H}_{k-1,n}$ be the projection onto the first $k-2$ rows, defined by $\varphi_{k-2}: e_{i,j} \mapsto e_{i,j}$ for $1\le i\le k-2$ and $e_{k-1,j} \mapsto 0$, for all $j = 1,\ldots, (n-k)$.  Define a second projection $\varphi_{-1}: \mathcal{H}^{(0)}_{k,n} \rightarrow \mathcal{H}^{(0)}_{2,n}$ onto the last row, by
	$\varphi_{-1}: e_{i,j} \mapsto 0$ for $1\le i\le k-2$ and $e_{k-1,j} \mapsto e_{k-1,j}$, for all $j = 1,\ldots, (k-1)(n-k)$. 
	Define 
	$$\mathcal{J}_{k-2} = \left\{J \in\mathcal{J}: \varphi_{k-2}(v_J) \not=0 \right\} \text{ and } \mathcal{J}_{-1} = \left\{J \in\mathcal{J}: \varphi_{-1}(v_J) \not=0 \right\}.$$
	
	Let $v \in \mathcal{L}^{(k)}_{n-k}$ be an arbitrary lattice point.  We claim that there exists a unique set of integers 
	$$\{c_J\in\mathbb{Z}: J \in \mathcal{J}\}$$
	such that 
	$$v = \sum_{J \in \mathcal{J}} c_J v_J.$$
	Now $\mathcal{J}_{k-2}$ is a collection of $k$-element subsets $J = \{j_1,\ldots, j_k\}$ wherein the largest two indices satisfy $j_{k-1} + 1 = j_k$.  Similarly, $(\mathcal{J}_{-1}\setminus \mathcal{J}_{k-2})$ is a collection of $k$-element subsets of the form $J = \lbrack a,b\rbrack \cup \{j\}$ wherein only the largest two indices $b$ and $j$ are not adjacent.
	
	Then, inducting on $k$, it follows that the (noncrossing) collections $\mathcal{J}_{k-2}$ and respectively $(\mathcal{J}_{-1}\setminus \mathcal{J}_{k-2})$, the sets
	$$\{v_J: J\in\mathcal{J}_{k-2} \},\ \ \{v_J: J\in(\mathcal{J}_{-1}\setminus \mathcal{J}_{k-2})\},$$
	are lattice bases for their respective sublattices, 
	$$\mathbb{Z}\left\{v_J: J \in \mathcal{J}_{k-2}\right\}\simeq \mathcal{L}^{(k-1)}_{n-k} \text{ and } \mathbb{Z}\left\{v_J: J \in (\mathcal{J}_{-1}\setminus \mathcal{J}_{k-2})\right\} \simeq \mathcal{L}^{(2)}_{n-k}.$$
	Consequently there exist unique integers $a_J,b_J$ such that 
	$$\varphi_{k-2}(v) = \sum_{J\in \mathcal{J}_{k-2}} a_J v_J,\ \ \varphi_{-1}(w) = \sum_{J\in \mathcal{J}_{-1}\setminus \mathcal{J}_{k-2}} b_J v_J,$$
	hence
	\begin{eqnarray}
		v & = & \varphi_{k-2}(w) + \varphi_{-1}(w)\\
		& = &  \sum_{J\in \mathcal{J}_{k-2}} a_J v_J + \sum_{J\in \mathcal{J}_{-1}\setminus \mathcal{J}_{k-2}} b_J v_J,\nonumber\\
		& = &  \sum_{J \in \mathcal{J}} c_J v_J,
	\end{eqnarray}
	where $c_J = a_J$ if $J \in \mathcal{J}_{k-2}$ and $c_J = b_J$ if $J \in \mathcal{J}_{-1}\setminus \mathcal{J}_{k-2}$.
\end{proof}

\begin{thm}\label{thm: noncrossing subdivision Wkn}
	Given any $v\in \mathcal{H}_{k,n}$, then there exists a unique noncrossing collection $\mathcal{J} \in \mathbf{NC}_{k,n}$ such that $v$ is in the relative interior of $\langle \mathcal{J}\rangle_+$.  In particular, the set of cones
	$$\left\{\langle \mathcal{J}\rangle_+: \mathcal{J} \in \mathbf{NC}_{k,n}\right\}$$
	defines a complete simplicial fan.	
\end{thm}

\begin{proof}
	We show that any point in $\mathcal{H}_{k,n}$ lies in the relative interior of a unique simplicial cone $\langle \mathcal{J}\rangle_+$ where $\mathcal{J} \in \mathbf{NC}_{k,n}$ is a noncrossing collection. 
	
	We induct on $k$, using the partition of unity $\varphi_{k-2}+\varphi_{-1} = \mathbf{id}$ to prepare the induction step as in the proof of Proposition \ref{prop: unimodular simplices noncrossing}.  Let us repeat the construction for sake of proximity.

	Denote by $\varphi_{k-2}:\mathcal{H}_{k,n} \rightarrow \mathcal{H}_{k,n}$ the projection $\varphi_{k-2}: e_{i,j} \mapsto e_{i,j}$ for $1\le i\le k-2$ and $e_{k-1,j} \mapsto 0$, for all $j = 1,\ldots, (n-k)$.	Define a second projection $\varphi_{-1}: \mathcal{H}_{k,n} \rightarrow \mathcal{H}_{k,n}$  by
$\varphi_{-1}: e_{i,j} \mapsto 0$ for $1\le i\le k-2$ and $e_{k-1,j} \mapsto e_{k-1,j}$, for all $j = 1,\ldots, (n-k)$.  Evidently we have $\varphi_{k-2}(\mathcal{H}_{k,n}) \simeq \mathcal{H}_{k-1,n-1}$ and $\varphi_{-1}(\mathcal{H}_{k,n})\simeq \mathcal{H}_{2,n-(k-2)}$.  Finally, again set 
$$\mathcal{J}_{k-2} = \left\{J \in\mathcal{J}: \varphi_{k-2}(v_J) \not=0 \right\} \text{ and } \mathcal{J}_{-1} = \left\{J \in\mathcal{J}: \varphi_{-1}(v_J) \not=0 \right\}.$$

	
	Given $J = \{j_1,\ldots, j_k\} \in \binom{\lbrack n\rbrack}{k}^{nf}$, it follows readily that  
	\begin{eqnarray}\label{eq: decomposed wj noncrossing}
		\varphi_{k-2}(v_J) & = & v_{j_1,\ldots, j_{k-1},j_{k-1}+1}\\
		\varphi_{-1}(v_J) & = & v_{j_{k-1}-(k-2),j_{k-1} - (k-3),\ldots, j_{k-1},j_{k}},\nonumber
	\end{eqnarray}
	noting that if $J$ is a single cyclic interval than $v_{J}=0$.

	Here, in $v_{j_1,\ldots, j_{k-1},j_{k-1}+1}$ at least the last two labels are consecutive, while in the element $$v_{j_{k-1}-(k-2),j_{k-1} - (k-3),\ldots, j_{k-1},j_{k}}$$ the first $k-1$ labels are consecutive.  Finally, note that $v_J$ can be recovered as
	\begin{eqnarray*}
		\varphi_{k-2}(v_J) +\varphi_{-1}(v_J)  & = &  v_{j_1,\ldots, j_{k-1},j_{k-1}+1}+v_{j_{k-1}-(k-2),j_{k-1} - (k-3),\ldots, j_{k-1},j_{k}}\\
		& = & v_J.
	\end{eqnarray*}	
	compatibly with the relation $\varphi_{k-2} + \varphi_{-1} = \mathbf{id}$.
	
	Now choose an arbitrary point $v \in \mathcal{H}_{k,n}$.  We claim that there exists a unique noncrossing collection $\mathcal{J}\in \mathbf{NC}_{k,n}$, possibly not maximal, such that 
	$$v = \sum_{J \in \mathcal{J}} t_J v_J.$$
	with all parameters $t_J\in \mathbb{R}_{>0}$ positive. 

	Let us assume for our inductive hypothesis that, for any pair of integers $(\ell,m)$ with $2\le \ell\le m-2$ with $\ell <k$ and $m<n$, then the set of cones
	$$\left\{\langle \mathcal{J}\rangle_+: \mathcal{J} \in \mathbf{NC}_{\ell,m}\right\}$$
	gives rise to a complete simplicial fan.  
	
	Then there exist unique noncrossing collections $\mathcal{I}_{k-2},\mathcal{I}_{-1}\in \mathbf{NC}_{k,n}$ such that 
	\begin{eqnarray*}
		\varphi_{k-2}(v) & = & \sum_{I \in \mathcal{I}_{k-2}}t'_I v_I
	\end{eqnarray*}
	and
	\begin{eqnarray*}
		\varphi_{-1}(v) & = & \sum_{I \in \mathcal{I}_{-1}}t''_I v_I,
	\end{eqnarray*}
	characterized by the requirement that all coefficients $t'_{I},t''_{I}>0$ must be positive.  Letting $\mathcal{I} = \mathcal{I}_{k-2} \cup \mathcal{I}_{-1}$, then
	\begin{eqnarray*}
		v & = & \varphi_{k-2}(v) + \varphi_{-1}(v)\\
		& = &  \sum_{I \in \mathcal{I}} t_I v_I,
	\end{eqnarray*}
	where all coefficients $t_I$ are strictly positive; here even though $\mathcal{I}_{k-2}$ and $\mathcal{I}_{-1}$ are separately noncrossing, in general their union $\mathcal{I}_{k-2} \cup \mathcal{I}_{-1}$ is not.  
	
	Therefore our remaining task is to take advantage of the particular formula for the elements in Equation \eqref{eq: decomposed wj noncrossing} to express $v$ as a linear combination indexed by a noncrossing collection.
	
	Let us now consider any pair $\{I_j,J_\ell\}$ for $I_j \in \mathcal{I}_{k-2}$ and $J_\ell\in \mathcal{I}_{-1}$; we have three cases.  We show that there is only one situation in which $\{I_j,J_\ell\}$ can be crossing, and then we show how to remove that crossing; iterating this process will ultimately identify which face of the triangulation contains the point $v$.  Here, each $I_j \in \binom{\lbrack n\rbrack}{k}^{{nf}}$ is by assumption an interval of length $k-1$ together with a single label, of the form
	$$I_j = \lbrack a_j,a_j+k-2\rbrack \cup \{b_j\}$$
	with $b_j \in \lbrack a_j+k,n\rbrack$, so that 
	$$v_{I_j\cup \{b_j\}} = f_{k-1,\lbrack a_j, b_j-k \rbrack}.$$
	As for the sets $J_\ell$, their largest two labels are adjacent:
	$$J_\ell = \{j_1,\ldots, j_{k-2},j_{k-1},j_{k-1}+1\}.$$

	We give an algorithm which locates the point $v$ in a simplex with vertices labeled by a noncrossing collection.
	
	We proceed by cases.
	
	\begin{enumerate}
		\item If $j_{k-1} < a_j+k-2$, then we immediately conclude that $\{I_j,J_\ell\}$ is noncrossing, since any pair of $k-1$ element sets, one of which is an interval, is automatically noncrossing.  For instance,
		$$(I_j,J_\ell) = (\{6,7,8,15\},\{1,3,5,6\})$$
		$$(I_j,J_\ell) = (\{6,7,8,15\},\{1,4,6,7\})$$
		$$(I_j,J_\ell) = (\{6,7,8,15\},\{2,4,7,8\})$$
		are all noncrossing.
		\item If $j_{k-1} = a_j+k-2$, then $\{I_j,J_\ell\}$ is crossing and we have that
		\begin{eqnarray}\label{eq: induction noncrossing}
			v_{J_\ell} + v_{I_j}  & = & v_{j_1,\ldots, j_{k-2},j_{k-1},b_j},
		\end{eqnarray}
		noting that both pairs 
		$$\{J_\ell,\{j_1,\ldots, j_{k-2},j_{k-1},b_j\}\}\text{ and } \{I_j,\{j_1,\ldots, j_{k-2},j_{k-1},b_j\}\}$$ 
		are noncrossing.
		
		We now consider three subcases based on the relative values of the parameters $t'_{I_j}$ and $t''_{J_\ell}$:
		\begin{enumerate}
			\item If we have $t:=t'_{I_j} = t''_{J_\ell}$ then the crossing disappears immediately as
			\begin{eqnarray}\label{eq: induction noncrossing2}
				\cdots + t(v_{J_\ell} + v_{I_j}) + \cdots   & = & \cdots + tv_{j_1,\ldots, j_{k-2},j_{k-1},b_j} + \cdots,
			\end{eqnarray}
			and we have new noncrossing collections
			$$\mathcal{I}^{(2)} = \mathcal{I}\setminus I_j,$$
			$$\mathcal{J}^{(2)} = \mathcal{J}\setminus J_\ell,$$
			and
			$$\mathcal{Q}^{(2)} = \{Q_1\},$$
			where $Q_1 = \{j_1,\ldots, j_{k-2},j_{k-1},b_j\}$.
			\item If $t'_{I_j} > t''_{J_\ell}$ then we write
			\begin{eqnarray}\label{eq: induction noncrossing3}
				\cdots +t'_{I_j}v_{I_j} + t''_{J_\ell}v_{J_\ell} + \cdots  & = & \cdots + t'_{I_j}v_{I_j} + t''_{J_\ell} (v_{j_1,\ldots, j_{k-2},j_{k-1},b_j} - v_{I_j}) + \cdots \\
				& = & \cdots + (t'_{I_j} - t''_{J_\ell})v_{I_j} + t''_{J_\ell} (v_{j_1,\ldots, j_{k-2},j_{k-1},b_j}) + \cdots \nonumber
			\end{eqnarray}
			and all coefficients are again positive.  Then we have new noncrossing collections
			$$\mathcal{I}^{(2)} = \mathcal{I},$$
			$$\mathcal{J}^{(2)} = \mathcal{J}\setminus J_\ell,$$
			and
			$$\mathcal{Q}^{(2)} = \{Q_1\},$$
			where $Q=\{j_1,\ldots, j_{k-2},j_{k-1},b_j\}$.
			\item The case when $t'_{I_j} < t''_{J_\ell}$ is similar.  We then have new noncrossing collections
			$$\mathcal{I}^{(2)} = \mathcal{I}\setminus I_j,$$
			$$\mathcal{J}^{(2)} = \mathcal{J},$$
			and
			$$\mathcal{Q}^{(2)} = \{Q_1\},$$
			where again $Q=\{j_1,\ldots, j_{k-2},j_{k-1},b_j\}$.
		\end{enumerate}
		For example,
		$$(I_j,J_\ell) = (\{6,7,8,15\},\{1,3,8,9\})$$
		is crossing on the 4-tuple $3<7<9<15$, since the middle interval is the same, $8$, in both $I_j$ and $J_\ell$, and we have
		$$v_{6,7,8,15} + v_{1,3,8,9} = v_{1,3,8,15}.$$
		\item If $j_{k-1} > a_j+k-2$, then $(I_j,J_\ell)$ must be noncrossing: it could be that $\{I_j,J_\ell\}$ is not weakly separated for some 4-tuple $a<b<c<d$, but the second condition in Definition \ref{defn: noncrossing} always fails to hold and the pair is indeed noncrossing.
		
		For example, while 
		$$(I_j,J_\ell) = (\{6,7,8,15\}, \{1,6,9,10\})$$
		is not weakly separated still the pair is noncrossing since in every instance of non-weak separation the respective interior intervals do not coincide.  For instance, it fails weak separation on the 4-tuple $1<6<10<15$, but the interior intervals do not coincide: $\{6,9\} \not= \{7,8\}$.
	\end{enumerate}
	
	We outline the completion of the algorithm: order the collection $\mathcal{I}_{-1}$ lexicographically, as $\mathcal{I}_{-1} = \{J_1\prec \cdots \prec J_r\}$, say, where $1\le r\le n-k-1$.  Similarly order $\mathcal{I}_{k-2}$ lexicographically as $\mathcal{I}_{k-2} = \{I_1\prec \cdots \prec  I_{s}\}$ with $1\le s\le (k-2)(n-k-1)$.  Then the Cartesian product $\mathcal{I}_{k-2}\times \mathcal{I}_{-1}$ inherits a lexicographic order in the standard way.  Each iteration of steps (1 - 4) either leaves the collection unchanged or it removes a crossing.  Once the iteration terminates we are left with a collection with no crossings, that is to say a noncrossing collection
	$$\mathcal{I}'_{k-2} \cup \mathcal{I}'_{-1} \cup \mathcal{Q}',$$
	where the three sets $\mathcal{I}'_{k-2}, \mathcal{I}'_{-1}, \mathcal{Q}'$ are disjoint, which completes the proof.
	
\end{proof}

The generalized root polytope $\mathcal{R}^{(k)}_{n-k}$ admits a triangulation which is isomorphic as a poset to the noncrossing complex $\mathbf{NC}_{k,n}$.

For any collection $\mathcal{J}  = \{J_1,\ldots, J_{m}\}\in \mathbf{NC}_{k,n}$ of nonfrozen subsets $J_i$, define 
\begin{eqnarray*}
	\lbrack \mathcal{J}\rbrack  = \lbrack \{v_J: J \in \mathcal{J}\}\rbrack & = & \text{Convex hull}\left\{0,p_{k,n}(v_{J_1}),\ldots, p_{k,n}(v_{J_{m}})\right\}\\
\end{eqnarray*}

\begin{figure}[h!]
	\centering
	\includegraphics[width=0.35\linewidth]{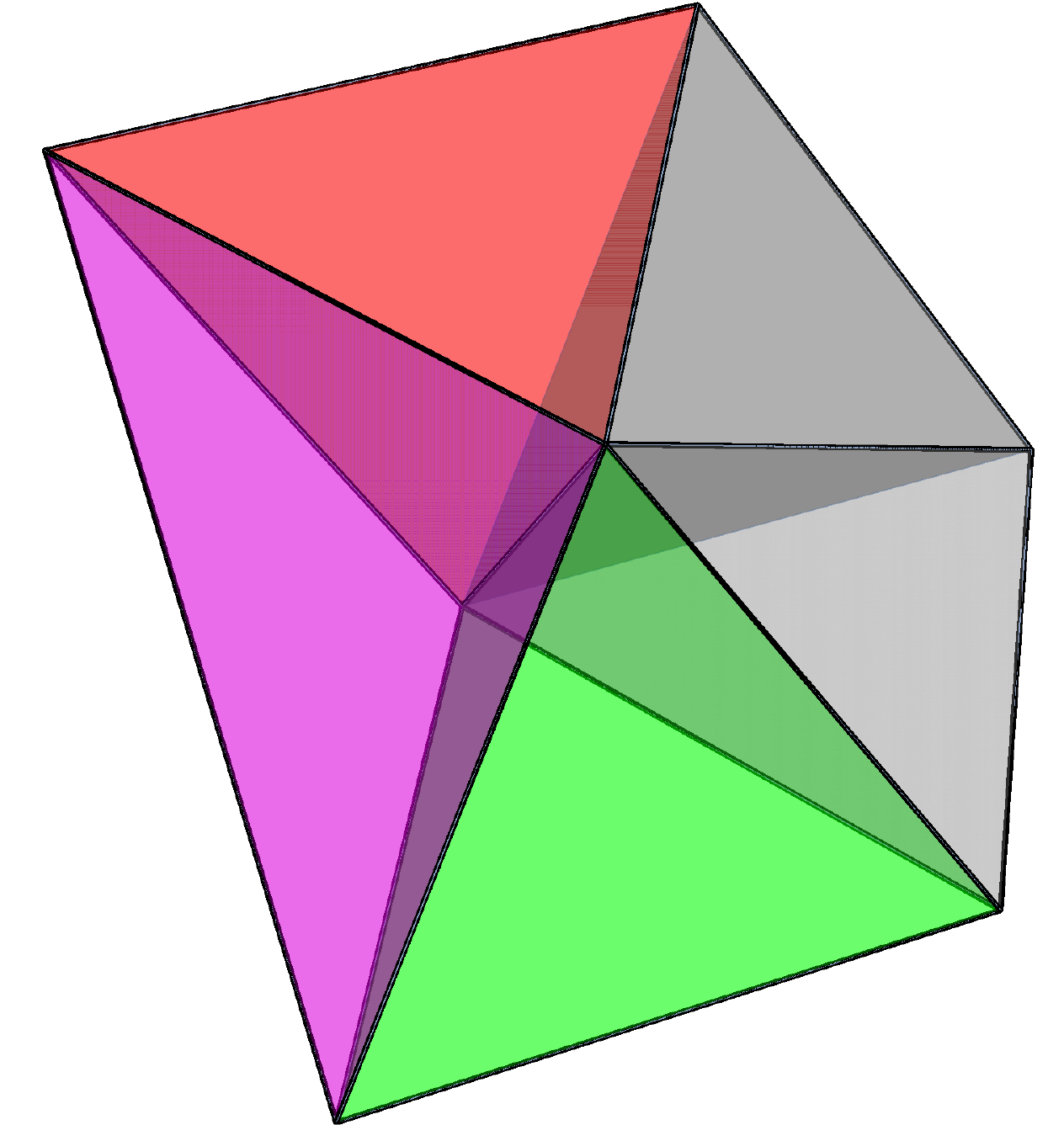}
	\caption{Illustration of Corollary \ref{cor: triangulation root polytope}: triangulation of $\hat{\mathcal{R}}^{(2)}_{3}$, see also \cite{EarlyPermutohedraKinematicSpace} for a detailed study of the $k=2$ associahedron and root polytope using the biadjoint scalar theory.}
	\label{fig:triangulationstandardplate}
\end{figure}

\begin{cor}\label{cor: triangulation root polytope}
	The set of simplices $\left\{\lbrack \mathcal{J}\rbrack: \mathcal{J} \in \mathbf{NC}_{k,n} \right\}$ defines a flag, unimodular triangulation\footnote{A triangulation is flag if it is the clique complex of its 1-skeleton: every complete subgraph of the 1-skeleton is the 1-skeleton of a simplex in the triangulation.  A triangulation (of a lattice polytope) is unimodular if every simplex has the same volume $\frac{1}{d!}$, where $d$ is the dimension.} of $\mathcal{R}^{(k)}_{n-k}$.  The relative volume of $\mathcal{R}^{(k)}_{n-k}$ is $C^{(k)}_{n-k}$.
	
\end{cor}

\begin{proof}
	This follows immediately from Theorem \ref{thm: noncrossing subdivision Wkn} by intersecting the simplicial cones with the polytope $\mathcal{R}^{(k)}_{n-k}$ and using that $\mathbf{NC}_{k,n}$ is a clique complex consisting of all complete subgraphs of the graph with vertices $\{j_1,\ldots, j_k\} \in \binom{\lbrack n\rbrack}{k}^{nf}$.
\end{proof}

\begin{defn}
	Call the triangulation from Corollary \ref{cor: triangulation root polytope} the \textit{noncrossing} triangulation of $\mathcal{R}^{(k)}_{n-k}$.
\end{defn}

\begin{example}\label{ex: triangulating}
	To see the prototypical situation for Case (2) in Theorem \ref{thm: noncrossing subdivision Wkn} we work in the triangulation of $\mathcal{R}^{(3)}_{3}$: we have the identity
	$$v_{134} + v_{235}=v_{135}$$
	for the (crossing) pair of vertices $v_{134}, v_{235} \in \mathcal{R}^{(3)}_{3}$.  

	Let us now look at what happens on the level of rational functions.  Then, using $\gamma_{135} = \alpha_{1,1} + \alpha_{2,2} = \gamma_{134} + \gamma_{235}$, we find 
	$$\frac{1}{\gamma_{134}\gamma_{235}} =  \frac{\gamma_{134} + \gamma_{235}}{\gamma_{135}\gamma_{134}\gamma_{235}} = \frac{1}{\gamma_{135}}\left(\frac{1}{\gamma_{134}} + \frac{1}{\gamma_{235}}\right).$$
	Here both $\{135,134\}$ and $\{135,235\}$ are noncrossing.
	
	However, more complex relations are possible: of course it could be that a simplex has to split into three or more simplices before every collection is noncrossing:
	\begin{eqnarray*}
		\frac{1}{\gamma_{134}\gamma_{235}\gamma_{236}} & = & \frac{1}{\gamma_{135}}\left(\frac{1}{\gamma_{134}} + \frac{1}{\gamma_{235}}\right)\frac{1}{\gamma_{236}}\\
		& = & \frac{1}{\gamma_{135}}\left(\frac{1}{\gamma_{136}}\left(\frac{1}{\gamma_{134}} + \frac{1}{\gamma_{236}}\right) + \frac{1}{\gamma_{235}}\frac{1}{\gamma_{236}}\right).
	\end{eqnarray*}
	We finally illustrate the expansion of one of the terms in the inductive expansion from Theorem \ref{thm: noncrossing subdivision Wkn}.  Here $\mathcal{I}_{k-2} = \{134,145\}$ and $\mathcal{I}_{-1} = \{236,346\}$.  We give the expansion of 
	\begin{eqnarray*}
		\left(\frac{1}{\gamma_{134}\gamma_{145}}\right)\left(\frac{1}{\gamma_{236}\gamma_{346}}\right).
	\end{eqnarray*}
	The pairs, in lexicographic order, are $(134,236) \prec (134,346) \prec (145,236) \prec (145,346)$.  Amongst these, the crossing pairs are 	
	$$(134,236) \prec (145,346).$$
	We then have $\gamma_{134} + \gamma_{236} = \gamma_{136}$ and $\gamma_{145} + \gamma_{346} = \gamma_{146}$, hence in two consecutive steps,
	\begin{eqnarray*}
		\left(\frac{1}{\gamma_{134}\gamma_{145}}\right)\left(\frac{1}{\gamma_{236}\gamma_{346}}\right) & = & \left(\frac{1}{\gamma_{136} \gamma_{134}} + \frac{1}{\gamma_{136} \gamma_{236}}\right)\frac{1}{\gamma_{145}\gamma_{346}}\\
		& = & \left(\frac{1}{\gamma_{136} \gamma_{134}} + \frac{1}{\gamma_{136} \gamma_{236}}\right)\left(\frac{1}{\gamma_{146}\gamma_{145}} + \frac{1}{\gamma_{146}\gamma_{346}}\right).
	\end{eqnarray*}
	The interpretation here is that we have subdivided the 4-dimensional simplex
	$$\text{convex hull} \left\{0,v_{134},v_{145},v_{236},v_{346} \right\}$$
	into four simplices,
	\begin{eqnarray*}
		& \text{convex hull} \left\{0,v_{134},v_{136},v_{145},v_{146} \right\},\ \text{convex hull}\left\{0,v_{134},v_{136},v_{146},v_{346} \right\} & \\
		& \text{convex hull} \left\{0,v_{136},v_{236},v_{145},v_{146} \right\},\ \text{convex hull}\left\{0,v_{136},v_{236},v_{146},v_{346} \right\}. & 
	\end{eqnarray*}
	\begin{figure}[h!]
		\centering
		\includegraphics[width=0.7\linewidth]{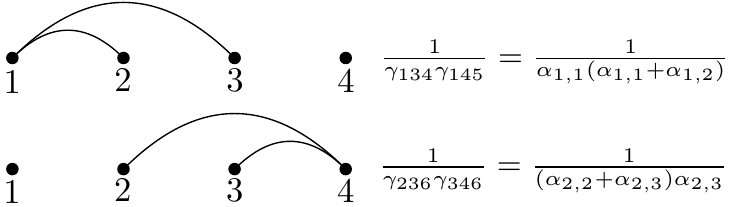}
		\caption{Alternating tree and rational function representation of the setup for step (2) in Theorem \ref{thm: noncrossing subdivision Wkn}, as in Example \ref{ex: triangulating}.  See \cite{Postnikov2006,Meszaros2011}.}
		\label{fig:treespair}
	\end{figure}

\end{example}

\section{Planar Kinematics Height Functions on Generalized Root Polytopes}\label{sec: regular subdivisions root polytope}

In this section, using the planar basis of kinematic invariants as the height function, we induce regular triangulations on $\hat{\mathcal{R}}^{(k)}_{n-k}$; we then project this to induce a triangulation of $\mathcal{R}^{(k)}_{n-k}$.

For each kinematic point $(s_0)\in \mathcal{K}_D(k,n)$, define a height function 
$$\mathfrak{t}_{s_0}:\text{verts}(\hat{\mathcal{R}}^{(k)}_{n-k})\rightarrow \mathbb{R}_{\ge 0}$$ 
by
$$\mathfrak{t}_{s_0}(v_J) = \eta_J(s_0),$$
where we lift the origin $0\in\mathbb{R}^{(k-1)\times (n-k)}$ to height $\mathfrak{t}_{s_0}(0) = 0$.  Notice then that at least $k$ vertices of $\hat{\mathcal{R}}^{(k)}_{n-k}$ have height $\eta_{j,j+1,\ldots, j+k-1}(s_0) = 0$, where the interval is cyclic.  These include: $0$ and the $k-1$ frozen vertices $\hat{v}_{\lbrack 1,i\rbrack \cup \lbrack n-k+i,n\rbrack} = \sum_{j=1}^{n-k}e_{i,j}$ for $i=1,\ldots, k-1$.  Now recall that under the projection $p_{k,n}$, these $k$ vertices are identified and are all mapped to the origin $0 \in \mathcal{H}_{k,n}$; therefore we may project the regular triangulation of $\hat{\mathcal{R}}^{(k)}_{n-k}$ to a regular triangulation of $\mathcal{R}^{(k)}_{n-k}$.

Letting the kinematic point vary in $(s_0)\in \mathcal{K}_D(k,n)$, we get a family of continuous piecewise linear surfaces over $\hat{\mathcal{R}}^{(k)}_{n-k}$, respectively $\mathcal{R}^{(k)}_{n-k}$, with heights constructed from the planar kinematic invariants $\eta_J(s)$, such that each given vertex $v_J$ is lifted to a height $\eta_{J}(s_0)$, where $v_J \in \mathcal{R}^{(k)}_{n-k}$ are vertices of the generalized root polytope.

Now let us suppose that $(s_0) \in \mathcal{K}_D(k,n)$ is interior, so that we have the strict inequality $(s_0)_J<0$ for each nonfrozen $J \in \binom{\lbrack n\rbrack}{k}^{nf}$.

Given an affine hyperplane of the form $x_{a}+x_b+x_c+x_d = 2$ in $\Delta_{k,n}$ and a vertex $e_J \in \Delta_{k,n}$ such that
$$e_J,e_J + (e_{a}-e_b) + (e_c-e_d),e_J + (e_{a}-e_b),e_J + (e_c-e_d) \in \Delta_{k,n},$$
define the octahedral curvature
$$\mathcal{O}_{(a,b,c,d)}(\rho_{e_J})(e_I) = (\rho_{e_J}(e_I)  + \rho_{e_J + (e_{a}-e_b) + (e_c-e_d)}(e_I)) - \left(\rho_{e_J + (e_{a}-e_b)}(e_I)  +\rho_{e_J + (e_c-e_d)}(e_I) \right).$$

\begin{lem}
	Given a nonfrozen vertex $u = e_J\in \Delta_{k,n}$ and a cyclic order $a<b<c<d$ such that all four of the following are in $\Delta_{k,n}$:
	$$e_J,e_J + (e_{a}-e_b) + (e_c-e_d),e_J + (e_{a}-e_b),e_J + (e_c-e_d) \in \Delta_{k,n},$$
	then there exists a vertex $e_I\in \Delta_{k,n}$ such that the value at $e_I$ of $\mathcal{O}_{(a,b,c,d)}(\rho_{e_J})(e_I)$ is strictly negative:
	$$\mathcal{O}_{(a,b,c,d)}(\rho_{e_J})(e_I) <0.$$
	In fact, for all vertices $e_I\in \Delta_{k,n}$, then $\mathcal{O}_{(a,b,c,d)}(\rho_{e_J})(e_I) \le 0$, and 
	$$\mathcal{O}_{(a,b,c,d)}(\rho_{e_J})(e_{j,j+1,\ldots, j+(k-1)})=0$$
	for all $n$ (frozen) vertices $e_{j,j+1,\ldots, j+(k-1)}$.
\end{lem}

\begin{proof}
	As the subdivision induced by the blade $\beta_J$ of $\Delta_{k,n}$ is positroidal, the height function $\rho_J$ satisfies the positive tropical Plucker relations
	$$\rho_{u}(v) + \rho_{u + (e_{a}-e_b) + (e_c-e_d)}(v) = \min\{\rho_{u + (e_{a}-e_b)}(v) + \rho_{u + (e_c-e_d)}(v), \rho_{u-(e_b-e_c)}(v) + \rho_{u - (e_d-e_a)}(v) \}.$$
	After subtracting $\rho_{e_J}(e_I) + \rho_{e_J + (e_{a}-e_b) + (e_c-e_d)}(e_I)$ from both sides it follows immediately that $$\mathcal{O}_{(a,b,c,d)}(\rho_{e_J})(e_I) \le 0$$ for all vertices $e_I\in \Delta_{k,n}$.  It is easy to check that 
	$$\mathcal{O}_{(a,b,c,d)}(\rho_{e_J})(e_{j,j+1,\ldots, j+(k-1)})=0$$
	for $j=1,\ldots, n$.

	Let us now be completely explicit in the core case $I=J$.
	
	We now compute $\mathcal{O}_{abcd}(\rho_u)(u)$ term by term.  First, note that $\rho_u(u) = \rho_0(0)$ for any $u\in \mathbb{R}^n$.  Then
	\begin{eqnarray*}
		\rho_0(0) & = & 0 \nonumber\\
		\rho_{e_a-e_b+e_c-e_d}(0) & = & -\frac{1}{n}\min\{-(n+(a-b)+(c-d)), -((a-b)+(c-d))\}\\
		& = & \frac{1}{n}(n+(b-a)+(d-c))\\
		\rho_{e_a-e_b}(0) & = & \frac{1}{n}(n+b-a)\\
		\rho_{e_c-e_d}(0) & = & \frac{1}{n}(n+d-c),
	\end{eqnarray*}
	hence after combining  and canceling,
	$$\mathcal{O}_{abcd}(\rho_{e_J})(e_J) = -1.$$
\end{proof}

\begin{cor}\label{cor: octahedral commutator positive}
	Let $(s_0)\in \mathcal{K}_D(\Delta_{k,n})$ be given.  Given a nonfrozen vertex $e_J \in \Delta_{k,n}$ with $J \in \binom{\lbrack n\rbrack}{k}^{nf}$ and distinct integers $\{a,b,c,d\}$ such that $a<b<c<d$ cyclically, and such that all four of the following are in $\Delta_{k,n}$:
	$$e_J,e_J + (e_{a}-e_b) + (e_c-e_d),e_J + (e_{a}-e_b),e_J + (e_c-e_d) \in \Delta_{k,n},$$
	then
	$$\eta_{e_J}+ \eta_{e_J+e_a-e_b + e_c-e_d} - \eta_{e_J+e_a-e_b}- \eta_{e_J+e_c-e_d} \ge 0,$$
	and the inequality is strict for all such $(a,b,c,d)$ if $(s_0)$ is interior, so  $s_J <0$ for all $J \in \binom{\lbrack n\rbrack}{k}^{nf}$.
\end{cor}

\begin{proof}
	This follows immediately from the construction of $\eta_J$ in terms of $\rho_J$: we have
	$$\eta_J = -\frac{1}{n}\sum_{I\in \binom{\lbrack n\rbrack}{k}}\min\{L_1(e_I-e_J),\ldots, L_n(e_I-e_J)\}s_I,$$
	so the identity is negated and becomes
	\begin{eqnarray*}
		& & \eta_{e_J}+ \eta_{e_J+e_a-e_b + e_c-e_d} - \eta_{e_J+e_a-e_b}- \eta_{e_J+e_c-e_d} \\
		& = & \left(\mathfrak{h}_{e_J}+ \mathfrak{h}_{e_J+e_a-e_b + e_c-e_d} - \mathfrak{h}_{e_J+e_a-e_b}- \mathfrak{h}_{e_J+e_c-e_d}\right) \cdot \left(\sum_I s_I e^I\right)\\
		& = & \sum_{I}\mathcal{O}_{abcd}(\rho_{e_J})(e_I)s_I\\
		& \ge & 0,
	\end{eqnarray*}
	since $\mathcal{O}_{abcd}(\rho_{e_J})(e_I) \le 0$ for all $I, J \in \binom{\lbrack n\rbrack}{k}^{nf}$ and $s_I\le 0$ for all $I\in \binom{\lbrack n\rbrack}{k}^{nf}$, noting that in the only case when $s_I$ is potentially positive, that is when $I$ is frozen, then $\mathcal{O}_{abcd}(\rho_{e_J})(e_I) = 0$.  The result follows.
\end{proof}
\begin{example}\label{example: noncrossing inequalities}
	On the kinematic space $\mathcal{K}(2,9)$ we have the relations
	\begin{eqnarray*}
		\eta_{14} + \eta_{25} - \eta_{15} - \eta_{24} & = &  -s_{25}\\
		\eta_{15} + \eta_{37} - \eta_{17} - \eta_{35} & = & -s_{26}-s_{27} -s_{36} - s_{37},
	\end{eqnarray*}
	while on the kinematic space $\mathcal{K}(3,9)$ we have 
	\begin{eqnarray*}
		\eta_{136} + \eta_{147} - \eta_{137} - \eta_{146} & = & -s_{147}-s_{478}-s_{479}\\
		\eta_{157} +\eta_{358} - \eta_{158} - \eta_{357} & = & -s_{238}-s_{248}-s_{258}-s_{348}-s_{358}.
	\end{eqnarray*}
	Thus, in all of the above situations, restricting the planar kinematic invariants $\eta_J$ to the cones respectively $\mathcal{K}_D(2,9)$ and $\mathcal{K}_D(3,9)$ makes all four of the above linear combinations nonnegative.
	
	However, in the noncrossing complex there is an important caveat.  That is, one also needs the following quantities to be positive in order to induce the noncrossing triangulation of $\mathcal{R}^{(3)}_{3}$:
	\begin{eqnarray*}
		\eta_{136} + \eta_{245} - \eta_{145} - \eta_{236} & = &  s_{236} - s_{245}\\
		\eta_{146} + \eta_{235} - \eta_{145} - \eta_{236} & = & s_{236} - s_{146}\\
		\eta_{135} + \eta_{246} - \eta_{145} - \eta_{236} & = &  -s_{146} - s_{245} - s_{246},
	\end{eqnarray*}
	and the signs of the first two quantities may change depending on where $(s_0)$ resides in $\mathcal{K}_D(3,6)$.
\end{example}

\begin{cor}\label{cor: extension octahedral commutator inequality}
	Fixing $(s_0) \in \mathcal{K}_D(k,n)$, then for a vertex $e_J\in \Delta_{k,n}$ and $a<b<c<d$ as in Corollary \ref{cor: octahedral commutator positive}, we have that 
	$$\eta_{e_J}(s_0) - \eta_{e_J+e_a-e_b + e_c-e_d}(s_0) - \eta_{e_J+e_a-e_b}(s_0)- \eta_{e_J+e_c-e_d}(s_0) \ge 0$$
	and the inequality is strict when $(s_0)$ is interior.
\end{cor}

Corollary \ref{cor: octahedral cell noncrossing subdivision} shows that every octahedral cell in $\mathcal{R}^{(k)}_{n-k}$ has positive curvature with respect to the height function $\mathfrak{t}_{s_0}$ for interior points $(s_0) \in \mathcal{K}_D(k,n)$.  This is a key step used in order to show that there exists a (regular) triangulation of $\mathcal{R}^{(k)}_{n-k}$ whose faces are in bijection with noncrossing collections in $\mathbf{NC}_{k,n}$.  It is not enough, however, because some noncrossing pairs are not determined by octahedral commutators, see Example \ref{example: octahedral remark} for further discussion and an example.

\begin{cor}\label{cor: octahedral cell noncrossing subdivision}
	Any half-octahedral cell
	$$\text{convex hull}\left\{0,v_{Iac},v_{Iad},v_{Ibc},v_{Ibd} \right\}$$
	in $\mathcal{R}^{(k)}_{n-k}$ has a canonical extension to an octahedron which we denote by
	$$\Delta_{2,4}(Iabcd) = \text{convex hull}\left\{0,v_{Iac},v_{Iad},v_{Ibc},v_{Ibd},(v_{Iac} + v_{Ibd} = v_{Iad}+v_{Ibc}) \right\}.$$
	Moreover, if $(s_0) \in \mathcal{K}_D(k,n)$ is interior, then $\mathfrak{t}_{s_0}$ induces a subdivision of $\Delta_{2,4}(Iabcd)$ whose internal face is the square
	$$\text{conv}\left\{0,v_{Iad},v_{Ibc},(v_{Iac} + v_{Ibd} = v_{Iad}+v_{Ibc}) \right\}.$$
\end{cor}
\begin{proof}
	The first claim is immediate; as for the second, it follows by translating Corollary \ref{cor: extension octahedral commutator inequality} that $\mathfrak{t}_{s_0}$ the average height of the lifts via $\mathfrak{t}_{s_0}$ of the vertices $v_{Iac}$ and $v_{Ibd}$ is larger than that for the pair $v_{Iad},v_{Ibc}$, and consequently the octahedron folds across the square
	$$\text{conv}\left\{0,v_{Iad},v_{Ibc},(v_{Iac} + v_{Ibd} = v_{Iad}+v_{Ibc}) \right\},$$
	as desired.
\end{proof}

Let us now be completely explicit what Corollaries \ref{cor: extension octahedral commutator inequality} and \ref{cor: octahedral cell noncrossing subdivision} look like on the level of rational functions.

Supposing that $I\cup \{a,b,c,d\} \in \binom{\lbrack n\rbrack}{k+2}$ with $a<b<c<d$, then from Corollary \ref{cor: four term identity gammaJ} we have the octahedral flip
\begin{eqnarray}\label{eq: octahedral flip gamma}
	\gamma_{Iac} + \gamma_{Ibd} = \gamma_{Iad} + \gamma_{Ibc}.
\end{eqnarray}
When all four linear functions $\gamma_{Iij}$ are nonzero we obtain, generically,
\begin{eqnarray}\label{eq:cubic 4-term relationA}
	\frac{1}{\gamma_{Ibd}\gamma_{Iad}\gamma_{Ibc}} + \frac{1}{\gamma_{Iac}\gamma_{Iad}\gamma_{Ibc}} & = & \frac{1}{\gamma_{Iac}\gamma_{Ibd}\gamma_{Ibc}} + \frac{1}{\gamma_{Iac}\gamma_{Ibd}\gamma_{Iad}}
\end{eqnarray}
that is,
\begin{eqnarray}\label{eq:cubic 4-term relation}
	\frac{1}{\gamma_{Iad}\gamma_{Ibc}} \left(\frac{1}{\gamma_{Iac}}+ \frac{1}{\gamma_{Ibd}} \right) & = & \frac{1}{\gamma_{Iac}\gamma_{Ibd}} \left(\frac{1}{\gamma_{Iad}}+ \frac{1}{\gamma_{Ibc}} \right).
\end{eqnarray}
Notice that the index sets of the Laurent monomials on each side of Equation \eqref{eq:cubic 4-term relationA} satisfy the following property: on the left-hand side we have the pairwise noncrossing sets
$$\{Iad,Ibd,Ibd\}\ \text{ and } \{Iac,Iad,Ibc\},$$
while on the right we have the two collections
$$\{Iac,Ibd,Ibc\}\ \text{ and } \{Iac,Ibd,Iad\},$$
both of which have the crossing pair $\{ac,bd\}$.  The identity in Equations \eqref{eq:cubic 4-term relationA} and \eqref{eq:cubic 4-term relation} amount to flipping the subdivision of the octahedron $\Delta_{2,4}(Iabcd)$.

Fix a nonzero point $(s_0) \in \mathcal{K}_D(k,n)$ and define 
$$\eta_J(\alpha) = \gamma_J(\alpha) + \eta_J(s_0).$$
Let us now investigate inhomogeneous analogs of Equations \eqref{eq: octahedral flip gamma}, \eqref{eq:cubic 4-term relationA} and \eqref{eq:cubic 4-term relation}.

We see immediately that
\begin{eqnarray*}
	(\eta_{Iac}(\alpha) + \eta_{Ibd}(\alpha)) - (\eta_{Iad}(\alpha) + \eta_{Ibd}(\alpha)) & = & (\eta_{Iac}(s_0) + \eta_{Ibd}(s_0)) - (\eta_{Iad}(s_0) + \eta_{Ibd}(s_0)),
\end{eqnarray*}
and assuming again that all $\eta_{Iij}(\alpha) \not=0$, then after rearrangement we arrive at the inhomogeneous relation
\begin{eqnarray*}
	\frac{1}{\eta_{Iad}(\alpha)\eta_{Ibc}(\alpha)} \left(\frac{1}{\eta_{Iac}(\alpha)}+ \frac{1}{\eta_{Ibd}(\alpha)} \right) & = & \frac{1}{\eta_{Iac}(\alpha)\eta_{Ibd}(\alpha)} \left(\frac{1}{\eta_{Iad}(\alpha)}+ \frac{1}{\eta_{Ibc}(\alpha)} \right)\\
	&  + & \frac{c_{Iabcd}}{\eta_{Iac}(\alpha)\eta_{Iad}(\alpha)\eta_{Ibc}(\alpha)\eta_{Ibd}(\alpha)},
\end{eqnarray*}
where the numerator
$$c_{Iabcd} := \left(\eta_{Iac}(s_0) + \eta_{Ibd}(s_0) - \eta_{Iad}(s_0) - \eta_{Ibc}(s_0)\right)$$
in the last line is by Corollary \ref{cor: extension octahedral commutator inequality} nonnegative.

\begin{example}\label{example: octahedral remark}
	The following flip characterizes the intersection of two adjacent simplices in the noncrossing triangulation of $\mathcal{R}^{(3)}_{3}$:
	$$v_{136} \Leftrightarrow v_{245}.$$
	There are two pairs of simplices in the triangulation of $\mathcal{R}^{(3)}_{3}$ which have facets characterized by this flip, namely
	$$\text{conv}\left(\left\{0,v_{145},v_{146},v_{245},v_{236}\right\}\right)\Leftrightarrow \text{conv}\left(\left\{0,v_{145},v_{146},v_{136},v_{236}\right\}\right)$$
	and
	$$\text{conv}\left(\left\{0,v_{145},v_{245},v_{235},v_{236}\right\}\right)\Leftrightarrow \text{conv}\left(\left\{0,v_{145},v_{235},v_{136},v_{236}\right\}\right) .$$
	Note that there is an intriguing sign dependence on the choice $(s_0) \in \mathcal{K}_D(3,6)$:
	$$\eta_{136} + \eta_{245} - \eta_{145} - \eta_{236} = s_{236} - s_{245}.$$
	
	One would like a set of linear functions $\hat{\eta}_J$ where all such ``flip moves'' are unambiguously nonnegative whenever $(s_0) \in \mathcal{K}_D(k,n)$, but a priori it is not obvious how to arrange for this!  However, Appendix \ref{sec:kinematic shift Appendix} contains a possible resolution to this sign ambiguity, at least in the case $(3,n)$ for $n\le 9$, via a certain kinematic shift $\eta_{a,b,c} \mapsto \hat{\eta}_{a,b,c}$.
\end{example}

\section{Fibered Simplices and Planar Faces: Building the PK associahedron $\mathbb{K}^{(k)}_{n-k}$}\label{sec: noncrossing complex}
In this section, we compute the Newton polytopes of the polynomials $\tau_J$.  	We give a conjectural construction of the PK associahedron $\mathbb{K}^{(k)}_{n-k}$ as the Minkowski sum of certain \textit{planar faces} $\mathcal{F}^{(i)}_{J}$ of the \textit{fibered simplex} $\Omega^{(k)}_{n-k}$.

The polynomials $\tau_J$ are constructed as follows.  For each $i=2,3,\ldots, n$, define $\tau_{1,i} = 1.$
With integers $s\ge 0$ and $m\ge 2$, and an $m$-element subset $J = \{j_1,\ldots, j_m\}$ of $\{2,3,\ldots, n\}$, set
$$\tau^{(s)}_{J} = \sum_{\{A \in \mathcal{I}_J: a_1\le a_2\le \cdots \le a_m\}}x_{1+s,a_1}x_{2+s,a_2}\cdots x_{m+s,a_m},$$
where $A = \{a_1,\ldots, a_m\}$, and 
$$\mathcal{I}_J = \lbrack j_1-1,j_2-1\rbrack \times \lbrack j_2-2,j_3-2\rbrack \times \cdots \lbrack j_{m-1}-(m-1),j_m-m\rbrack.$$
We are now ready to define the face polynomials $\tau_J$.
\begin{defn}
	For any $I \in \binom{\lbrack n\rbrack}{k}$, let $s\ge 0$ be such that $I=\lbrack 1,s\rbrack \cup J$, where $J = I\setminus \lbrack 1,s\rbrack$.  Now define 
	$$\tau_{I} = \tau^{(s)}_{j_1-s,j_2-s,\ldots, j_m-s}.$$
	Here $\lbrack 1,0\rbrack = \emptyset$ is the empty set.
\end{defn}

\begin{figure}[h!]
	\centering
	\includegraphics[width=0.7\linewidth]{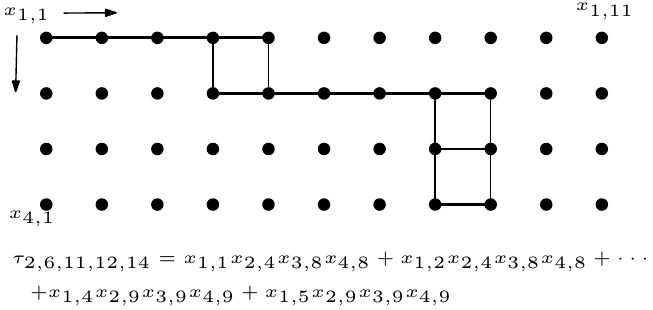}
	\caption{The Staircase: Planar Face polynomial $\tau_{2,6,11,12,14}$}
	\label{fig:planar-facekn-grid-3-14}
\end{figure}

The main question which we pose is the following: do there exist integers $c_J \ge 0$ such that for $J \in \binom{\lbrack n\rbrack}{k}$ the Newton polytope for the product of all polynomials $\tau_J$ is cut out by the set of facet inequalities $\gamma_J(\alpha) \ge c_J$?

In that case we would have an equality
$$\text{Newt}\left(\prod_{J\in \binom{\lbrack n\rbrack}{k}} \tau_J\right) = \left\{(\alpha)\in \mathcal{H}^{(\lambda_1,\ldots, \lambda_{k-1})}_{k,n}: \gamma_J \ge c_J \text{ for each } J \in \binom{\lbrack n\rbrack}{k}^{nf}\right\},$$
where 
$$\lambda_i = c_{\lbrack 1,i\rbrack \cup \lbrack n-i,n\rbrack},$$
all inequalities being facet-defining.

In Section \ref{sec:facet PK polytope} we exhibit such integers for one of the most degenerate (but very important!) subcases, the Planar Kinematics (PK) polytope $\Pi^{(k)}_{n-k}$.

Given polytopes $P,Q \subset \mathbb{R}^m$, define their Minkowski sum 
$$P\boxplus Q = \left\{p+q: (p,q)\in P\times Q \right\}.$$

For any subinterval $\lbrack j_1,j_2\rbrack = \{j_1,j_1+1,\ldots, j_2\}$ of $\{1,\ldots, n-k\}$ and any $1\le i\le k-1$, define a simplex
$$\Delta^{(i)}_{\lbrack j_1,j_2\rbrack} = \left\{\sum_{\ell\in \lbrack j_1,j_2\rbrack}x_{i,\ell} e_{i,\ell}: \sum_{\ell=1}^{n-k} x_{i,\ell} = 1 \right\}.$$

With $2\le k\le n-2$ as usual, define 
\begin{eqnarray}\label{eq: fibered simplex}
	\Omega^{(k)}_{n-k} = \left\{(x_{i,j})\in \lbrack 0,1\rbrack^{(k-1)\times(n-k)}: \ \ x_{i,\lbrack 1,j\rbrack} \ge x_{i+1,\lbrack 1,j\rbrack},\ \ \sum_{j=1}^{n-k} x_{i,j} = 1\right\}.
\end{eqnarray}
Note that the fibered simplices $\Omega^{(k)}_{n-k}$ for $k\ge3$ have new edges which are not edges of the ambient hypersimplex $\Delta_{(k-1),(k-1)(n-k)}$: the middle inequalities $ x_{a,\lbrack 1,b\rbrack} \ge x_{a+1,\lbrack 1,b\rbrack}$ prevent $\Omega^{(k)}_{n-k}$ from being a generalized permutohedron.  

Here for instance
$$\Omega^{(2)}_{n-2} = \Delta_{1,n-2} = \left\{x\in \lbrack 0,1\rbrack^n: \sum_{j=1}^{n-2}x_j=1  \right\}$$
is the usual $(n-3)$-dimensional simplex.  

Notice that in fact $\Omega^{(k)}_{n-k}$ is a subpolytope of the Cartesian product of simplices $\Delta^{(1)}_{\lbrack 1,n-k\rbrack} \times \cdots \times \Delta^{(k-1)}_{\lbrack 1,n-k\rbrack} $.  More interestingly, the vertex set of $\Omega^{(k)}_{n-k}$ is a \textit{section} of the projection $e_{i,j} \mapsto e_j$.  By this we mean that the projection sends vertices of $\Omega^{(k)}_{n-k}$ bijectively onto the vertices of a dilated simplex of dimension $n-k-1$, that is $(k-1)\Delta_{\lbrack 1,n-k\rbrack}$.

As usual assume that $2\le k\le n-2$.  Fixing $m\in \{2,\ldots, k\}$, given any 
$$(i,J)\in \{1,2,\ldots, k-(m-1)\} \times \binom{\lbrack (n-2) - (k-m)\rbrack}{m},$$
we define the \textit{planar face}
\begin{eqnarray}\label{eq: planar face}
	& & \mathcal{F}^{(i)}_J\\
	& = &\left\{(\alpha) \in \bigboxplus_{\ell=i}^{i+m-1}(\Delta^{(i+(\ell-1))}_{\lbrack j_\ell-(\ell-1),j_{\ell+1}-(\ell-1)\rbrack}):\  \alpha_{a,\lbrack 1,b\rbrack} \ge \alpha_{a+1,\lbrack 1,b\rbrack},\ (a,b)\in \lbrack 1,k-2\rbrack \times \lbrack 1,n-k-1\rbrack\right\}.\nonumber
\end{eqnarray}
For convenience let us introduce an intermediate notation for the planar face $\mathcal{F}^{(i)}_J$ in Equation \eqref{eq: planar face},
$$\mathcal{F}^{\lbrack i,i+m-1\rbrack}_{\lbrack j_1,j_2\rbrack,\lbrack j_2-1,j_3\rbrack,\ldots, \lbrack j_{m-1}-(m-1), j_m - m\rbrack}.$$

Remark that since the Minkowski summands live in mutually orthogonal subspaces, $\mathcal{F}^{(i)}_J$ is in fact at top level a Cartesian product of simplices, then cut by some additional inequalities.

\begin{prop}
	Given 
	$$(i,J) \in  \lbrack k-(m-1)\rbrack \times \binom{\lbrack (n-2) - (k-m)\rbrack}{m},$$	
	then there exists an integer vector $v$ such that the translation $v+\mathcal{F}^{(i)}_J$ of $\mathcal{F}^{(i)}_J$ is a face of $\Omega^{(k)}_{n-k}$.
\end{prop}

In a slight abuse of terminology, we shall say simply that $\mathcal{F}^{(i)}_J$ \textit{is} a planar face of $\Omega^{(k)}_{n-k}$, though in actuality this is true only up to translation.

\begin{defn}\label{defn: PK associahedron}
	Denote by $\mathbb{K}^{(k)}_{n-k}$ the Planar Kinematics (PK) associahedron: it is the Minkowski sum of all planar faces  $\mathcal{F}^{(i)}_J$ of $\Omega^{(k)}_{n-k}$,
	$$\mathbb{K}^{(k)}_{n-k} = \bigboxplus_{(i,J)}\mathcal{F}^{(i)}_J,$$
	where $(i,J)$ ranges over the set 
	$$(i,J)\in \bigsqcup_{m=2}^k\left(\{1,2,\ldots, k-(m-1)\} \times \binom{\lbrack (n-2) - (k-m)\rbrack}{m}\right).$$
\end{defn}
Clearly, when $k=2$ then $m=2$ and $\mathbb{K}^{(2)}_{n-2}$ is the usual associahedron, in the metric realization due originally to Loday \cite{Loday2003}, see also \cite{Postnikov2006}: the indexing set for the Minkowski summands is then $\{1\}\times \binom{\lbrack n-2\rbrack}{2}$.

The planar faces $\mathcal{F}^{(i)}_J$ (noting that by convention we are \textit{not} including vertices) have a particularly nice enumeration which helped establish the connection to binary geometries, see \cite{AHL2019Stringy,AHLT2019}.

\begin{prop}\label{prop: number Minkowski summands}
	The number of  Minkowski summands in $\mathbb{K}^{(k)}_{n-k}$ is 
	$$\sum_{m=2}^{k}(k-(m-1))\binom{(n-2)-(k-m)}{m} = \binom{n}{k} - k(n-k) -1.$$
\end{prop}

\begin{defn}\label{defn:deltaiJ}
	Given integers $k,m,n$ as usual with $2\le k\le n-2$ and $2\le m\le k$, for any $(i,J) \in \lbrack k-(m-1)\rbrack \times \binom{\lbrack (n-2) - (m-1)\rbrack}{m} $, define the \textit{face polynomial} $\delta^{(i)}_J \in \mathbb{C}\lbrack \mathbf{x}\rbrack$ to be the following sum over the vertices of the face $\mathcal{F}^{(i)}_J$:
	\begin{eqnarray}\label{eq: delta polynomials}
		\delta^{(i)}_J = \sum_{v\in \text{verts}\left(\mathcal{F}^{(i)}_J\right)} \mathbf{x}^v.
	\end{eqnarray}
\end{defn}

\begin{example}
	For $(k,n) = (3,6)$, then for $\Omega^{(3)}_3$ there are two families of $\binom{3}{2}$ planar faces $$\{\mathcal{F}^{(i)}_{ab}: i=1,2;\ \{a,b\}\in \{12,13,23\}\},$$ which are simplices, and four faces $\mathcal{F}^{(1)}_{abc}$, two of which are simplices.  The first two families are 
	$$\mathcal{F}^{(1)}_{12} = \Delta^{(1)}_{\lbrack 1,2\rbrack},\ \mathcal{F}^{(1)}_{13} = \Delta^{(1)}_{\lbrack 1,3\rbrack},\ \mathcal{F}^{(1)}_{13} = \Delta^{(1)}_{\lbrack 2,3\rbrack}$$
	and
	$$\mathcal{F}^{(2)}_{12} = \Delta^{(2)}_{\lbrack 1,2\rbrack},\ \mathcal{F}^{(2)}_{13} = \Delta^{(2)}_{\lbrack 1,3\rbrack},\ \mathcal{F}^{(2)}_{13} = \Delta^{(2)}_{\lbrack 2,3\rbrack}$$
	with face polynomials respectively
	$$\delta^{(1)}_{12} = x_{1,12},\ \ \delta^{(1)}_{13} = x_{1,123},\ \ \delta^{(1)}_{23} = x_{1,23}$$
	and
	$$\delta^{(2)}_{12} = x_{2,12},\ \ \delta^{(2)}_{13} = x_{2,123},\ \ \delta^{(2)}_{23} = x_{2,23}.$$
	Here our notation means that for instance $\delta^{(2)}_{23} =x_{2,23} = x_{2,2}+x_{2,3}$.
	
	The remaining $\binom{4}{3}$ planar faces are, keeping in parallel both the shorthand notation $\mathcal{F}^{(i)}_J$ and the more explicit notation,
	\begin{eqnarray*}
		\mathcal{F}^{(1)}_{123} = \mathcal{F}^{\lbrack 1,2\rbrack}_{\lbrack 1,2\rbrack,\lbrack 1,2\rbrack},& & \mathcal{F}^{\lbrack 1,2\rbrack}_{124} = \mathcal{F}^{\lbrack 1,2\rbrack}_{\lbrack 1,2\rbrack,\lbrack 1,3\rbrack},\\
		\mathcal{F}^{(1)}_{134} = \mathcal{F}^{\lbrack 1,2\rbrack}_{\lbrack 1,3\rbrack,\lbrack 2,3\rbrack},& & \mathcal{F}^{\lbrack 1,2\rbrack}_{234} = \mathcal{F}^{\lbrack 1,2\rbrack}_{\lbrack 2,3\rbrack,\lbrack 2,3\rbrack},
	\end{eqnarray*}
	with face polynomials respectively	
	\begin{eqnarray*}
		\delta^{(1)}_{123} & = & x_{1,1}x_{2,12}+x_{1,2}x_{2,2},\\
		\delta^{(1)}_{124} & = & x_{1,1}x_{2,123}+x_{1,2}x_{2,23},\\
		\delta^{(1)}_{134} & = & x_{1,12}x_{2,2}+x_{1,3}x_{2,3},\\
		\delta^{(1)}_{234} & = & x_{1,2}x_{2,23}+x_{1,3}x_{2,3}.
	\end{eqnarray*}
	For $(k,n) = (4,8)$, say, consider the planar face $\mathcal{F}^{(2)}_{124}$ of $\Omega^{(4)}_{4}$:
	$$\mathcal{F}^{(2)}_{124} = \mathcal{F}^{\lbrack 2,3\rbrack}_{\lbrack 1,2\rbrack,\lbrack 1,3\rbrack}.$$
	The corresponding face polynomial is given by
	$$\delta^{(2)}_{124} = x_{2,1}x_{3,123} + x_{2,2}x_{3,23} = x_{2,1}x_{3,1}+x_{2,1}x_{3,2} + x_{2,1}x_{3,3} + x_{2,2}x_{3,2}+x_{2,2}x_{3,3}.$$
	Finally, for $(k,n) = (4,n)$ with $n\ge 8$ say,
	$$\mathcal{F}^{(1)}_{1346} = \mathcal{F}^{\lbrack 1,3\rbrack}_{\lbrack 1,3\rbrack,\lbrack 2,3\rbrack,\lbrack 2,4\rbrack}$$
	and 
	$$\delta^{(1)}_{1346} = x_{1,12}x_{2,2}x_{3,234} + x_{1,3}x_{2,3}x_{3,34}.$$
\end{example}
\begin{rem}
	Note that in the construction of the u-variables for $k=3,4$ in Appendix \ref{sec: u-variables} we do allow the first two indices in $J$ to coincide; this is done intentionally, in order to simplify the definition by avoiding a plethora of subcases, but the side-effect is the introduction of more common (monomial) factors in the numerator and denominator of the $u_J$'s for certain subsets $J$.  To illustrate the cancellation, consider $(k,n) = (3,6)$, where we have
	\begin{eqnarray*}
		u_{236} & = & \frac{\tau_{336}\tau_{245}}{\tau_{236} \tau_{345}}\\
		& = & \frac{\left(x_{1,1} x_{2,2}+x_{1,2} x_{2,2}\right) \left(x_{1,2} x_{2,2}+x_{1,2} x_{2,3}\right)}{x_{1,2} x_{2,2} \left(x_{1,1} x_{2,1}+x_{1,1} x_{2,2}+x_{1,2} x_{2,2}+x_{1,1} x_{2,3}+x_{1,2} x_{2,3}\right)}\\
		& = & \frac{\left(x_{1,1}+x_{1,2}\right) \left(x_{2,2}+x_{2,3}\right)}{x_{1,1} x_{2,1}+x_{1,1} x_{2,2}+x_{1,2} x_{2,2}+x_{1,1} x_{2,3}+x_{1,2} x_{2,3}}\\
		& = & \frac{\delta^{(1)}_{12}\delta^{(2)}_{23}}{\delta^{(1)}_{124}}.
	\end{eqnarray*}
	
	To find a better (non-redundant!) scheme to define $u$-variables directly in terms of the face variables $\delta^{(i)}_J$ introduced in Definition \ref{defn:deltaiJ} would require some extra analysis that is not needed for this paper; the problem is left to future work.
\end{rem}

\section{Positive Parameterization From Fibered Simplices: Parameterizing $X(k,n)$}\label{sec: positive parameterization}

Let us reinterpret the matrix entries of the so-called positive parameterization\footnote{We thank Freddy Cachazo for explanations and for sharing Mathematica code \cite{CachazoDiscussions}.}, as sums of monomials over vertex sets of fiber simplices $\Omega^{(k)}_{n-k}$, constructing in this way an embedding 
$$(\mathbb{CP}^{n-k-1})^{\times (k-1)} \hookrightarrow X(k,n)$$
of a Cartesian product of projective spaces into $X(k,n)$.

We first define a $(k-1)\times (n-k-1)$ polynomial-valued matrix $M_{k,n}$ with entries $m_{i,j}(x)$, with $(i,j) \in \lbrack 1,k-1\rbrack \times \lbrack 1,n-k\rbrack$, defined by 
\begin{eqnarray*}
	m_{i,j}(\{x_{a,b}: (a,b) \in \lbrack i,k-1\rbrack \times \lbrack 1,j\rbrack\}) & = & \sum_{\mathbf{v}\in \text{verts}\left(\Omega^{(i+1)}_{j}\right)} \mathbf{x}^{\mathbf{v}'},
\end{eqnarray*}
where we define
$$x^{\mathbf{v}'} = \prod_{(a,b) }x^{v_{a+(i-1),b}}_{a,b}.$$
For instance, fixing $(k,n) = (4,n)$ then 
\begin{eqnarray*}
	m_{1,2} & = & x_{1,1} x_{2,1} x_{3,1}+x_{1,1} x_{2,1} x_{3,2}+x_{1,1} x_{2,2} x_{3,2}+x_{1,2} x_{2,2} x_{3,2}\\
	m_{2,3} & = & x_{2,1} x_{3,1}+x_{2,1} x_{3,2}+x_{2,2} x_{3,2}+x_{2,1} x_{3,3}+x_{2,2} x_{3,3}+x_{2,3} x_{3,3}\\
	m_{3,4} & = & x_{3,1}+x_{3,2}+x_{3,3}+x_{3,4}.
\end{eqnarray*}
We emphasize that the polynomials $m_{i,j}$ for $k\ge 3$ do not in general give rise to planar faces of the form $\mathcal{F}^{(i)}_J$.  The first such instance is for $(3,6)$, specifically the matrix entry $m_{1,3}$.

For the embedding $(\mathbb{CP}^{n-k-1})^{\times k-1} \hookrightarrow X(k,n)$, we construct a $(k-1)\times (n-k)$ matrix with $M_{k,n}$ as its upper right block:
$$\begin{bmatrix}
	1 &  &  & 0 & m_{1,1}& \cdots  & m_{1,n-k} \\
	& \ddots &  &  &\vdots & \ddots & \vdots \\
	&  & 1 &  & m_{k-1,1} & & m_{k-1,n-k} \\
	0&  &  & 1 & 1 & \cdots & 1 \\
\end{bmatrix}.
$$

For instance, for rank $k=3$ we have
\begin{eqnarray}
	M_{3,6} & = & \begin{bmatrix}
		m_{1,1} & m_{1,2} & m_{1,3}\\
		m_{2,1} & m_{2,2} & m_{2,3}
	\end{bmatrix}\\	
	& = & \begin{bmatrix}
		x_{1,1}x_{2,1} & x_{1,1}x_{2,12}+x_{1,2}x_{2,2}  & x_{1,1}x_{2,123} + x_{1,2}x_{2,23} + x_{1,3}x_{2,3} \\
		x_{2,1} & x_{2,12} & x_{2,123}\nonumber
	\end{bmatrix}
\end{eqnarray}
and for the embedding we have
$$\begin{bmatrix}
	1 & 0 & 0 & x_{1,1}x_{2,1} & x_{1,1}x_{2,12}+x_{1,2}x_{2,2} & x_{1,1}x_{2,123} + x_{1,2}x_{2,23} + x_{1,3}x_{2,3} \\
	0& 1 & 0 & x_{2,1} & x_{2,12} & x_{2,123}\\
	0 & 0 & 1 & 1 & 1 & 1 
	\nonumber
\end{bmatrix}.$$

\section{Root Kinematics Potential Function and a Noncrossing Degree for the Positive Tropical Grassmannian}\label{sec: noncrossing lattice rays}

In this section, we use Theorem \ref{thm: noncrossing subdivision Wkn} to construct an injection from the set of rays of the positive tropical Grassmannian into the noncrossing triangulation of the root polytope $\mathcal{R}^{(k)}_{n-k}$.  As an immediate consequence we obtain a noncrossing \textit{degree} on rays. 

We also obtain a bijection between the set of pairs in $\mathbf{NC}_{3,n}$ that are \textit{not} weakly separated, and the certain matroidal weighted blade arrangements (see \cite{Early2020WeightedBladeArrangements}), called \textit{tripods} $\tau_{(u,v,w),(u',v',w')}$, that is 
$$\tau_{(u,v,w),(u',v',w')}= -\beta_{uvw}+\beta_{u'vw} + \beta_{uv'w} + \beta_{uvw'}.$$
The notation will be explained below.  Briefly, each such tripod induces a regular positroidal subdivision of $\Delta_{3,n}$.  It is expected to be coarsest (and to therefore define a ray of the positive tropical Grassmannian, and a pole of $m^{(3)}_n$), but to prove this in complete generality, which would be essential for possible physical applications to achieve a complete description of the poles, may require additional insights.

Here is the main construction of this section.

For an overview of the blade model, see Appendix \ref{sec: blades}, and in particular Theorem \ref{Early2020WeightedBladeArrangements}.  

\begin{defn}\label{defn:projection blade to weight}
	Define a projection $\varphi_{\mathcal{R}}:\mathcal{B}(k,n) \rightarrow \mathbb{R}^{(k-1)\times (n-k)}$ by 
	$$\beta_J \mapsto v_J.$$
\end{defn}
Lest this projection seem unmotivated, let us point out its origin: namely the \textit{root} kinematics potential function and scattering equations, which we now introduce.  Given $J\in \binom{\lbrack n\rbrack}{k}^{nf}$, define a cube
$$\mathcal{U}(J) = \left\{e_J\in \Delta_{k,n}: \text{for each $j\in\{1,\ldots, n\}$ such that }  (j,j+1) \in J\times J^c\right\}$$
where addition is cyclic modulo $n$.

Then the planar cross-ratio $w_J$ is given by 
$$w_J = \prod_{M\in \mathcal{U}(J)} p_M^{k-\#(M\cap J)+1},$$
where $p_J$ is the minor with column set $J$ of a given $k\times n$ matrix.

\begin{defn}\label{def: Root Kinematics}
	For any $(\alpha) \in \mathbb{R}^{(k-1)\times (n-k)}$, let $(s(\alpha)) \in \mathcal{K}(k,n)$ be the point characterized as follows.  If $J \in \binom{\lbrack n\rbrack}{k}^{nf}$ then put $\eta_J(s(\alpha))=0$.  Otherwise, put
	$$\eta_J(s(\alpha)) = \gamma_J(\alpha).$$
	The \textit{Root Kinematics subspace} $\mathcal{K}^{(Rt)}(k,n)$ of $\mathcal{K}(k,n)$ is given by 
	$$\mathcal{K}^{(Rt)}(k,n)=\left\{(s(\alpha))\in \mathcal{K}(k,n): (\alpha) \in \mathbb{R}^{(k-1)\times (n-k)}\right\}.$$
\end{defn}

We define the \textit{Root Kinematics potential function} $\mathcal{S}^{(Rt)}_{k,n}$, by 
\begin{eqnarray}
	\mathcal{S}^{(Rt)}_{k,n} & = & \sum_{J \in \binom{\lbrack n\rbrack}{k}^{nf}}\log(w_J)\gamma_J(\alpha).
\end{eqnarray}

We give two examples.

\begin{example}
	The potential functions $\mathcal{S}^{(Rt)}_{3,6}$ and $\mathcal{S}^{(Rt)}_{4,8}$ have the following simple expressions, and it is easy to see that both simplify to the form in Equation \eqref{eq: root potential} when evaluated on the positive parameterization:
	\begin{eqnarray*}
		\mathcal{S}^{(Rt)}_{3,6} & = & \alpha_{1,1} \log \left(\frac{p_{156} p_{234}}{p_{134} p_{256}}\right)+\alpha_{1,2} \log \left(\frac{p_{124} p_{156} p_{345}}{p_{134} p_{145} p_{256}}\right)+\alpha_{1,3} \log \left(\frac{p_{125} p_{456}}{p_{145} p_{256}}\right)\\
		& + & \alpha_{2,1} \log \left(\frac{p_{126} p_{134}}{p_{124} p_{136}}\right)+\alpha_{2,2} \log \left(\frac{p_{123} p_{126} p_{145}}{p_{124} p_{125} p_{136}}\right)+\alpha_{2,3} \log \left(\frac{p_{123} p_{156}}{p_{125} p_{136}}\right)
	\end{eqnarray*}
	and 
	\begin{eqnarray*}
		\mathcal{S}^{(Rt)}_{4,8} & = & \alpha_{1,1} \log \left(\frac{p_{1678} p_{2345}}{p_{1345} p_{2678}}\right)+\alpha_{1,2} \log \left(\frac{p_{1245} p_{1678} p_{3456}}{p_{1345} p_{1456} p_{2678}}\right)+\alpha_{1,3} \log \left(\frac{p_{1256} p_{1678} p_{4567}}{p_{1456} p_{1567} p_{2678}}\right)\\
		& + & \alpha_{1,4} \log \left(\frac{p_{1267} p_{5678}}{p_{1567} p_{2678}}\right)+\alpha_{2,1} \log \left(\frac{p_{1278} p_{1345}}{p_{1245} p_{1378}}\right)+\alpha_{2,2} \log \left(\frac{p_{1235} p_{1278} p_{1456}}{p_{1245} p_{1256} p_{1378}}\right)\\
		& + & \alpha_{2,3} \log \left(\frac{p_{1236} p_{1278} p_{1567}}{p_{1256} p_{1267} p_{1378}}\right)+\alpha_{2,4} \log \left(\frac{p_{1237} p_{1678}}{p_{1267} p_{1378}}\right)+\alpha_{3,1} \log \left(\frac{p_{1238} p_{1245}}{p_{1235} p_{1248}}\right)\\
		& + & \alpha_{3,2} \log \left(\frac{p_{1234} p_{1238} p_{1256}}{p_{1235} p_{1236} p_{1248}}\right)+\alpha_{3,3} \log \left(\frac{p_{1234} p_{1238} p_{1267}}{p_{1236} p_{1237} p_{1248}}\right)+\alpha_{3,4} \log \left(\frac{p_{1234} p_{1278}}{p_{1237} p_{1248}}\right)
	\end{eqnarray*}
\end{example}

\begin{claim}\label{claim: Root Kinematics potential function embedding}
	For each $(\alpha) \in \mathbb{R}^{(k-1)\times (n-k)}$, the function $\mathcal{S}^{(Rt)}_{k,n}$ has a unique critical point in the torus quotient $G(k,n) \slash (\mathbb{C}^\ast)^{\times n}$.  Allowing $(\alpha)$ to vary over $(\mathbb{CP}^{(n-k-1)})^{\times (k-1)}$ gives an embedding $(\mathbb{CP}^{(n-k-1)})^{\times (k-1)} \hookrightarrow G(k,n) \slash (\mathbb{C}^\ast)^{\times n}$, and we obtain a manifestly $PGL(k)$ invariant characterization of the positive parameterization, as the solution to a differential equation.
\end{claim}

\begin{proof}[Sketch of proof]
	By evaluating $\mathcal{S}^{(Rt)}_{k,n}$ on the image of the positive parameterization one finds that 
	\begin{eqnarray}\label{eq: root potential}
		\mathcal{S}^{(Rt)}_{k,n}\big\vert_{BCFW} & = & \sum_{(i,j) \in \lbrack 1,k-1\rbrack \times \lbrack 1,n-k\rbrack} \log\left(\frac{x_{i,j}}{\sum_{\ell=1}^{n-k}x_{i,\ell}}\right)\alpha_{i,j},
	\end{eqnarray}
	which is familiar from algebraic statistics.  Indeed, the right hand side decomposes as a sum of $k-1$ log-likelihood functions, each of which is known to have maximum likelihood (ML) degree 1 \cite{Huh2013,HuhSturmfels}, with unique critical point of the form given above.
\end{proof}

In what follows, to connect with the positive tropical Grassmannian $\text{Trop}^+G(k,n)$ we rely on Theorem \ref{Early2020WeightedBladeArrangements}, recalled in Appendix \ref{sec: blades} from \cite{Early2020WeightedBladeArrangements}, which constructs an embedding into the set of weighted blade arrangements $\mathfrak{B}_{k,n}$ with image the \textit{matroidal} weighted blade arrangements, $$\text{Trop}^+G(k,n) \simeq \mathcal{Z}(k,n).$$  In Proposition \ref{prop: map Trop to NC} we take into account Theorem \ref{thm: noncrossing subdivision Wkn}, which uses the noncrossing complex $\mathbf{NC}_{k,n}$ in the construction of a complete simplicial fan in $\mathcal{H}_{k,n}$.

\begin{prop}\label{prop: map Trop to NC}
	Suppose that $\pi \in \mathbb{Z}^{\binom{n}{k}}$ is a vector in the direction of a ray of the positive tropical Grassmannian.  Let 
	$$\beta_{\mathbf{c}} = \sum_{J\in \binom{\lbrack n\rbrack}{k}^{nf}}c_J\beta_J$$ be the image of $\pi$ in $\mathcal{Z}(k,n)$, where $\mathbf{c} \in \mathbb{Z}^{\binom{n}{k}-n}$.
	
	Then there exists a unique noncrossing collection $\{J_1,\ldots, J_m\} \in \mathbf{NC}_{k,n}$ and positive integers $a_1,\ldots, a_m \in \mathbb{Z}_{>0}$ such that 
	$$\varphi_\mathcal{R}(\beta_{\mathbf{c}}) =  \sum_{j=1}^m a_j v_{J_j}.$$
\end{prop}
\begin{proof}
	We restrict the linear projection $\mathcal{B}(k,n) \rightarrow \mathcal{H}_{k,n}$, characterized by $\beta_J \mapsto v_J$, to  $\mathcal{Z}(k,n)$.  We study the image of an arbitrary matroidal weighted blade arrangement $\beta_{\mathbf{c}} \in \mathcal{Z}(k,n)$; it maps under $\varphi_\mathcal{R}$ to an integer point $v_{\mathbf{c}} \in \mathcal{H}_{k,n}$.  Now this is in one of the (simplicial) cones in the complete simplicial fan from Theorem \ref{thm: noncrossing subdivision Wkn}.  Consequently, $v_\mathbf{c}$ lies in (the relative interior of) a cone
	$$\left\{ \sum_{j=1}^m t_j v_{J_j} \in \mathcal{H}_{k,n}: t_j\ge0 \right\}$$
	generated by some roots $v_{J_1},\ldots, v_{J_m}$, where $\{J_1,\ldots, J_m\} \in \mathbf{NC}_{k,n}$ is a noncrossing collection that can be calculated explicitly from $v_{\mathbf{c}}$; here existence of the linear combination follows from completeness of the fan, that is to say, every point in the ambient space lies in a unique cone.  Moreover, integrality comes from Proposition \ref{prop: unimodular simplices noncrossing}.  Uniqueness follows because every cone in the fan is simplicial, so the rays of any given cone form a basis for the ambient subspace.
\end{proof}

\begin{defn}
	Let $\beta_{\mathbf{c}} \in \mathcal{Z}(k,n)$ be a matroidal weighted blade arrangement.  Then, the noncrossing degree of $\beta_{\mathbf{c}}$ is the size $d$ of the noncrossing collection $\{J_1,\ldots, J_d\}$ that indexes the expansion of its image in a noncrossing collection of $v_J$'s under the map $\varphi_\mathcal{R}$, as in
	$$\varphi_\mathcal{R}(\beta_{\mathbf{c}}) = \sum_{j=1}^d a_jv_{J_j},$$
	where $a_j>0$ are strictly positive.
\end{defn}

\begin{rem}
	We urge caution when working with the noncrossing degree: while every ray of $\text{Trop}^+G(k,n)$ is assigned a (weighted) noncrossing collection, the same is not true for higher dimensional cones, and in particular for maximal cones -- and thus finest positroidal subdivisions of the hypersimplex $\Delta_{k,n}$.  Indeed, we shall see in what follows that the noncrossing expansion is not constant on the four bipyramids in $\text{Trop}^+G(3,6)$.  This is demonstrated on the level of the amplitude itself in Example \ref{example: am36 free triangulation computation}.
	
	There are four special rays in $\text{Trop}^+G(3,6)$ which induce four positroidal 3-splits of the hypersimplex $\Delta_{3,6}$; these four are permuted transitively by the group generated by the cycle $(123456)$ and the reflection $j\mapsto 7-j$, say.  These 3-splits are induced by the following hypersimplicial blades (see Definition \ref{defn:blade}):
	\begin{eqnarray}\label{example: 36A}
		((12_1 34_1 56_1)), & & ((16_1 23_1 45_1))\\
		((12_1 56_1 34_1)), & & ((16_1 45_1 23_1)).\nonumber
	\end{eqnarray}
	The corresponding matroidal weighted blade arrangements are respectively
	\begin{eqnarray}\label{example: 36B}
		\beta_{246}, & & \beta_{135}\\
		-\beta_{246} + \beta_{124} + \beta_{346} + \beta_{256},  & &  -\beta_{135} + \beta_{235} + \beta_{145} + \beta_{136}.\nonumber
	\end{eqnarray}
	Then, for instance, to see that in Equation  \eqref{example: 36A} and \eqref{example: 36B}, the blade $((12_1 56_1 34_1))$ intersects the hypersimplex $\Delta_{3,6}$ in the $(n-2)$ skeleton of the same positroidal subdivision that is induced by the weighted blade arrangement $-\beta_{246} + \beta_{124} + \beta_{346} + \beta_{256}$, one verifies that they induce the same subdivision on each of the six faces $\partial_j(\Delta_{3,6})\simeq \Delta_{2,5}$ for $j=1,\ldots, 6$.  Then,
	$$\partial\left(-\beta_{246} + \beta_{124} + \beta_{346} + \beta_{256}\right) = \beta^{(1)}_{2,4}+\beta^{(2)}_{1,4}+\beta^{(3)}_{4,6}+\beta^{(4)}_{3,6}+\beta^{(5)}_{2,6}+\beta^{(6)}_{2,5},$$
	see \cite{Early2020WeightedBladeArrangements} for details.
	
	Note that the reflection symmetry $j \mapsto n+1-j$ for positroidal subdivisions is already broken in Equation \eqref{example: 36B}; but we emphasize that a priori Equations \eqref{example: 36A} and \eqref{example: 36B} describe a priori very different objects that happen to behave the same on the faces $\partial_j(\Delta_{3,6})$.  For instance, $((12_1 34_1 56_1))$ is the $4$ skeleton of a polyhedral fan with three maximal cones.  Intersecting these cones with the hypersimplex $\Delta_{3,6}$ induces a 3-split.  On the other hand, the blade $\beta_{246} = ((1,2,3,4,5,6))_{246}$ is the 4-skeleton of a simplicial fan with 6 maximal cones, translated to the vertex $e_{246} \in \Delta_{3,6}$.  From the general results of \cite{Early19WeakSeparationMatroidSubdivision} it follows that they have the same intersection with $\Delta_{3,6}$:
	$$((1,2,3,4,5,6))_{246} \cap \Delta_{3,6} = ((12_1 34_1 56_1))\cap \Delta_{3,6}.$$
	The story for the matroidal \textit{weighted} blade arrangements is somewhat more subtle.  We refer the reader to \cite{Early2020WeightedBladeArrangements} for details.
	
	Now, in the embedding of Trop$^+G(3,6)$ into $\mathcal{Z}(3,6)$ afforded by Theorem \ref{Early2020WeightedBladeArrangements}, the rays that induce the four positroidal 3-splits of $\Delta_{3,6}$ have the following primitive generators:
	\begin{enumerate}
		\item Two of noncrossing degree 1:
		$$\beta_{135} \mapsto v_{135},$$
		$$\beta_{246} \mapsto v_{246}$$
		\item Two of noncrossing degree 2:
		\begin{eqnarray*}
			-\beta_{135} + \beta_{235} + \beta_{145} + \beta_{136} & \mapsto & e_{1,1}+e_{1,2}+e_{2,2}+e_{2,3}\\
			& = & v_{145} + v_{236},\\
			-\beta_{246} + \beta_{124} + \beta_{346} + \beta_{256} & \mapsto & e_{1,3}+e_{2,1}\\
			& = & v_{124} + v_{356}.
		\end{eqnarray*}
		These evidently correspond to the two pairs in $\mathbf{NC}_{3,6}$ which are not weakly separated:
		$$\{145,236\}\text{ and } \{124,356\}.$$
	\end{enumerate}
\end{rem}

\begin{defn}
	When $\{u,v,w\} \in \binom{n}{3}$ is a subset such that $u,v,w$ are not cyclically adjacent, given any 
	$$(u',v',w')\in \lbrack u+1,v-1\rbrack\times \lbrack v+1,w-1\rbrack \times \lbrack w+1,u-1\rbrack,$$ then we call the element $$\tau_{(u,v,w),(u',v',w')}= -\beta_{uvw}+\beta_{u'vw} + \beta_{uv'w} + \beta_{uvw'}\in\mathfrak{B}_{3,n}$$
	a \textit{tripod}.
\end{defn}
Note that $\tau_{I,J} \not= \tau_{J,I}$.

In our notation above the intervals are cyclic, so that for instance with $n=6$ we have $\lbrack 5,1\rbrack = \{5,6,1\}$.

Remark that in fact tripods are in fact matroidal: they are in $\mathcal{Z}(3,n)$, and they induce positroidal subdivisions that are (conjecturally\footnote{One could of course make computational proofs of coarseness for given $n$, but we would like to understand the general geometric and combinatorial structure and this seems possibly more subtle.}) coarsest.

\begin{prop}\label{prop:tripod noncrossing classification}
	Any tripod $\tau_{I,J} \in \mathcal{Z}(3,n)$ has noncrossing degree 2.  Moreover, there is a bijection between tripods $\tau_{I,J}$ and noncrossing pairs $\{I',J'\}$ in $\mathbf{NC}_{3,n}$ that are not weakly separated.
\end{prop}

\begin{proof}
	This is a straightforward computation; we state the results.
	
	The two tripods associated to a given 6-element set $J = \{j_1,\ldots,j_6\} \in \binom{\lbrack n\rbrack}{6}$ with $j_1<\cdots <j_6$, are $\tau_{j_1j_3j_5,j_2j_4j_6}$ and $\tau_{j_2j_4j_6,j_1j_3j_5}$.  They satisfy 
\begin{eqnarray*}
	\varphi_{\mathcal{R}}(\tau_{j_1j_3j_5,j_2j_4j_6}) & = & \varphi_{\mathcal{R}}(-\beta_{j_1j_3j_5} + \beta_{j_2j_3j_5} + \beta_{j_1j_4j_5} + \beta_{j_1j_3j_6}) \\
	& =&  v_{j_1j_4j_5} + v_{j_2j_3j_6}
\end{eqnarray*}
	and
	\begin{eqnarray*}
	\varphi_{\mathcal{R}}(\tau_{j_2j_4j_6,j_1j_3j_5}) & = & \varphi_{\mathcal{R}}(-\beta_{j_2j_4j_6} + \beta_{j_1j_2j_4} + \beta_{j_3j_4j_6} + \beta_{j_2j_5j_6}) \\
	& =&  v_{j_1j_2j_4} + v_{j_3j_5j_6}
	\end{eqnarray*}
	Now one can check that the edges in $\mathbf{NC}_{3,n}\setminus \mathbf{WS}_{3,n}$ are exactly the $2\binom{n}{6}$ pairs of the form 
	$$\{j_1j_2j_4,j_3j_5j_6\},\ \ \{j_1j_4j_5,j_2j_3j_6\}$$
	where $J$ ranges over all $J = \{j_1,\ldots, j_6\} \in \binom{\lbrack n\rbrack}{3}$.
\end{proof}
\begin{example}
	Joining the two tripods
	$$-\beta _{259}+\beta _{125}+\beta _{359}+\beta _{269}\  \text{ and }  -\beta_{268}+\beta _{269}+\beta _{278}+\beta _{568},$$
	together on their common blade $\beta_{269}$ gives 
	$$\left(-\beta _{259}+\beta _{125}+\beta _{359}+\beta _{269} -\beta _{268}+\beta _{278}+\beta _{568}\right),$$
	which projects to an element in the interior of a three-dimensional cone in the normal fan to $\Pi_{3,9}(s_0)$ for interior $(s_0)$:
	\begin{eqnarray*}
		\left(-\beta _{259}+\beta _{125}+\beta _{359}+\beta _{269} -\beta _{268}+\beta _{278}+\beta _{568}\right)& \mapsto & e_{1,3}+e_{1,4}+e_{1,5}+e_{2,1}+e_{2,2}+e_{2,5}+e_{2,6}\\
		& = & \left(v _{125}+v _{378}+v _{569}\right),
	\end{eqnarray*}
	and it is not difficult to verify that 
	$$\{\{1,2,5\},\{3,7,8\},\{5,6,9\}\} \in \mathbf{NC}_{3,9}\setminus \mathbf{WS}_{3,9},$$
	and we infer that the noncrossing degree of any embedding into $\mathcal{Z}(3,n)$ of such an element is 3.  Compare to Figure 1 in \cite{Early2020WeightedBladeArrangements}.
	
	For $\mathcal{Z}(3,10)$, we look toward a pole of the generalized biadjoint scalar amplitude $m^{(3)}(\mathbb{I}_{10},\mathbb{I}_{10})$ which expands in terms of kinematic invariants as $\eta_{\mathbf{c}}$, say, and correspondingly a matroidal weighted blade arrangement $\beta_{\mathbf{c}}$ where coefficients are not only units $c_J \in \{\pm1\}$, shown in what follows to define a lattice point in the interior of a five-dimensional cone in the normal fan to $\Pi_{3,10}(s_0)$, again, as usual, where $(s_0)$ is interior.  As a weighted blade arrangement it becomes
	\begin{eqnarray*}
		& &\left( 2 \beta _{124}-2 \beta _{2410}-\beta _{259}+\beta _{2510}-\beta _{268}+\beta _{2610}+\beta _{278}+\beta _{289}+2 \beta _{3410}-\beta _{358}+\beta _{359}+\beta _{368}+\beta _{458}\right)\\
		& \mapsto &3 f_{1,3}+3 f_{1,4}+2 f_{1,5}+f_{1,6}+2 f_{2,1}+f_{2,4}+2 f_{2,5}+2 f_{2,6}+2 f_{2,7}\\
		& = & \left(2v _{124}+v _{3610}+v _{378}+v _{389}+v _{4510}\right),
	\end{eqnarray*}
	where we note that the index sets for the monomials,
	$$\{\{1,2,4\},\{3,6,10\},\{3,7,8\},\{3,8,9\},\{4,5,10\}\} \in \mathbf{NC}_{3,10}\setminus \mathbf{WS}_{3,10},$$
	define a noncrossing collection.  Consequently this (coarsest) matroidal weighted blade arrangement (and thus the corresponding ray of $\text{Trop}^+G(3,10)$) has noncrossing degree 5.  See Figure \ref{fig:layered-example-bad-310} for the embedding of the matroidal weighted blade arrangement on the 1-skeleton of the hypersimplex $\Delta_{3,10}$.
	
\end{example}
\begin{figure}[h!]
	\centering
	\includegraphics[width=0.7\linewidth]{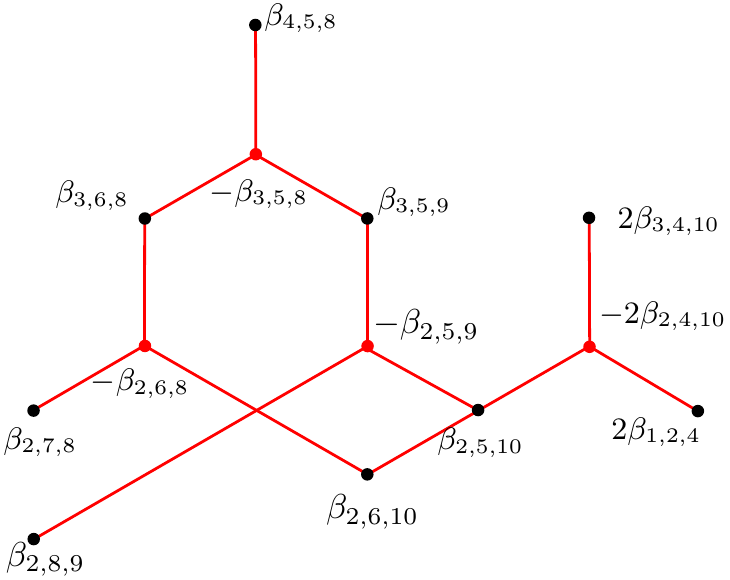}
	\caption{Pinning the matroidal weighted blade arrangement to the 1-skeleton of the hypersimplex $\Delta_{3,10}$.  Matroidal blade arrangements provide a novel representation of a tropical linear space, here in dimension nine.}
	\label{fig:layered-example-bad-310}
\end{figure}

\begin{example}\label{example: 48 large}
	Let us now study some of the building blocks of matroidal weighted blade arrangements introduced recently in \cite{Early2020WeightedBladeArrangements}, here the fan of matroidal weighted blade arrangements $\mathcal{Z}(4,n)$ on the hypersimplex $\Delta_{4,n}$; however let us change our calculation and use planar kinematic invariants by replacing the blade $\beta_J$ with the linear function $\eta_J$ on $\mathcal{K}(4,n)$.  In this example we put $n=8$ and so our kinematic space is $\mathcal{K}(4,8)$.  Then we have the relations
	\begin{eqnarray*}
		\left(-2 \eta _{1,3,5,7}+\eta _{1,3,5,8}+\eta _{1,3,6,7}+\eta _{1,4,5,7}+\eta _{2,3,5,7}\right)\big\vert_{Rt} & = & \left(\gamma _{1,4,5,8}+\gamma _{2,3,6,7}\right)\\
		\left(-2 \eta _{2,4,6,8}+\eta _{1,2,4,6}+\eta _{2,4,7,8}+\eta _{2,5,6,8}+\eta _{3,4,6,8}\right)\big\vert_{Rt} & = & \left(\gamma _{1,2,4,6}+\gamma _{3,5,7,8}\right),
	\end{eqnarray*}
	and it is easy to see that the pairs 
	$$\{\{1,4,5,8\},\{2,3,6,7\}\} \text{ and } \{\{1,2,4,6\},\{3,5,7,8\}\}$$
	on the right-hand sides of the two equations are each noncrossing (but not weakly separated).  
	
	For another example, consider the following pole of $m^{(4)}_8$, evaluated on root kinematics:
	\begin{eqnarray*}
	\eta_{\mathbf{c}}\vert_{Rt}& = & \left(\eta_{1357} - \eta_{2357} - \eta_{1457} - \eta_{1367} - \eta_{1358} +\eta_{3457} + \eta _{1235}+\eta _{1378}+\eta _{1458}+\eta _{1567}+\eta _{2367}\right)\vert_{Rt}\\
	& = & \gamma _{1235}+\gamma _{1567}+\gamma _{3478},
\end{eqnarray*}

	see Figure \ref{fig:pole48reverse}.	 Here again the collection 
	$$\{1235,1567,3478\}$$
	is noncrossing but not weakly separated, and $\eta_{\mathbf{c}}$ has noncrossing degree 3.
	\begin{figure}[h!]
		\centering
		\includegraphics[width=0.55\linewidth]{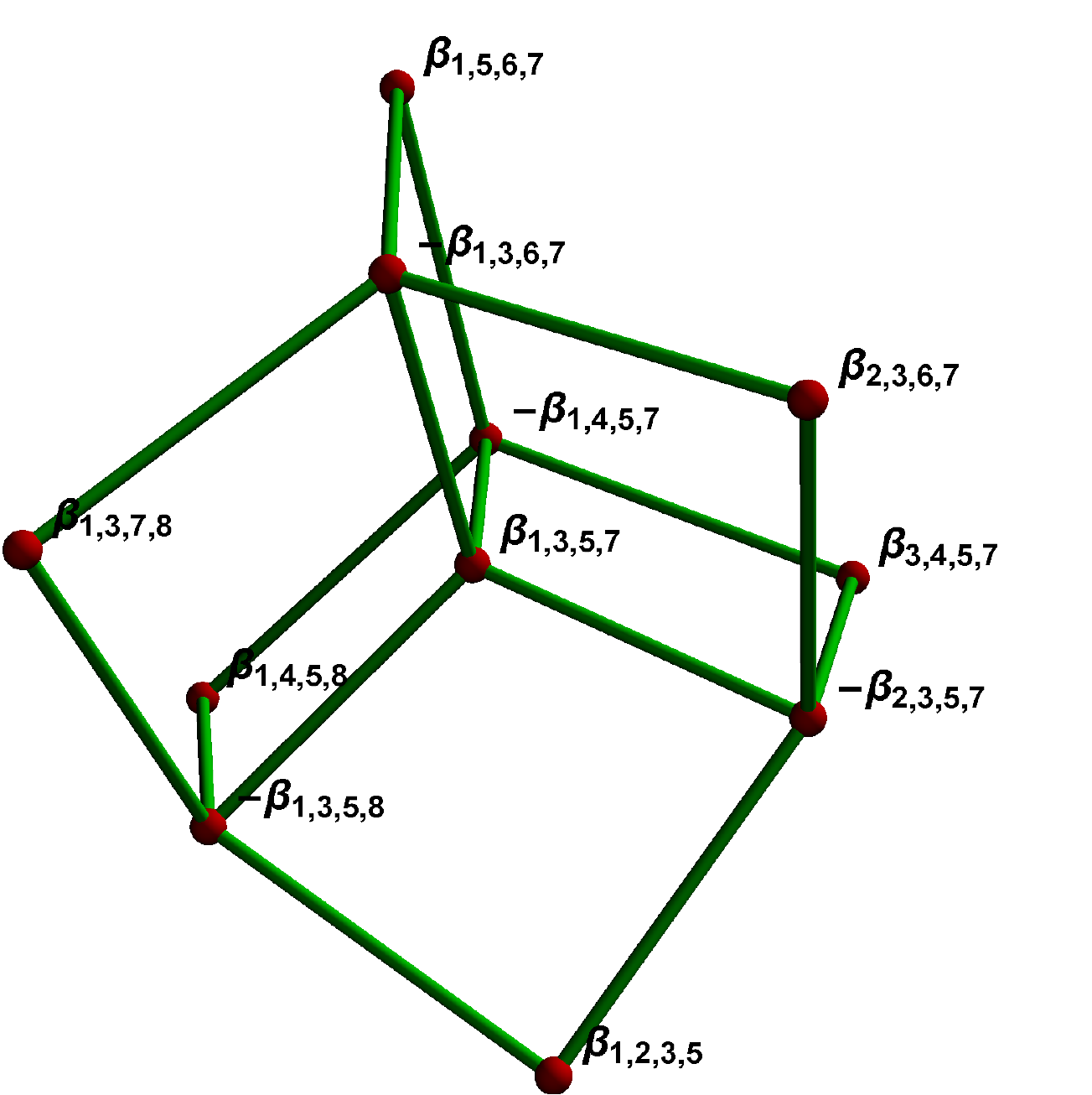}
		\caption{Matroidal weighted blade arrangement in Example \ref{example: 48 large} on $\Delta_{4,8}$; the corresponding pole of the generalized biadjoint scalar amplitude $m^{(4)}(\mathbb{I}_8,\mathbb{I}_8)$ reduces on the Root Kinematics subspace to the linear functional with noncrossing expansion $\gamma _{1235}+\gamma _{1567}+\gamma _{3478}$.}
		\label{fig:pole48reverse}
	\end{figure}
	
	Finally, for a quite nontrivial example see Figure \ref{fig:zonotopaltessellationpole48d} for the pinning onto the 1-skeleton of the hypersimplex $\Delta_{4,8}$ of the matroidal weighted blade arrangement that induces the pole (for a tabulation of poles, see for instance \cite{CGUZ2019,HeRenZhang2020,GuevaraZhang2020}; for generic kinematics, these determined matroidal weighted blade arrangements),
	\begin{eqnarray*}
	\beta_{\mathbf{c}} & = & \left(-\beta _{1247}+\beta _{1248}+\beta _{1257}+\beta _{2347}\right) +  \left(-\beta _{1346}+\beta _{1348}+\beta _{1356}+\beta _{2346}\right)\\
	& + & (-\beta _{3568}+\beta _{1356}+\beta _{3578}+\beta _{4568})  + (-\beta _{2578}+\beta _{1257}+\beta _{2678}+\beta _{3578})\\
	& - & (\beta _{1257}+\beta _{1356}+\beta _{3578}).
\end{eqnarray*}
	Replacing each $\beta_{abcd}$ with the planar kinematic invariant $\eta_{abcd}$ and then evaluating on root kinematics gives 
	$$\eta_{\mathbf{c}}\vert_{Rt} = \gamma _{1258}+\gamma _{1357}+\gamma _{2346}+\gamma _{2348}+\gamma _{4678},$$
	where the right hand side is uniquely characterized by the two requirements, that (1) all coefficients should be positive, and (2) the index set 
	$$\{1258,1357,2346,2346,4678\}$$
	should be noncrossing: it is in $\mathbf{NC}_{4,8}$.  Therefore the corresponding coarsest matroidal weighted blade arrangement has noncrossing degree $6$.
	\begin{figure}[h!]
		\centering
		\includegraphics[width=1\linewidth]{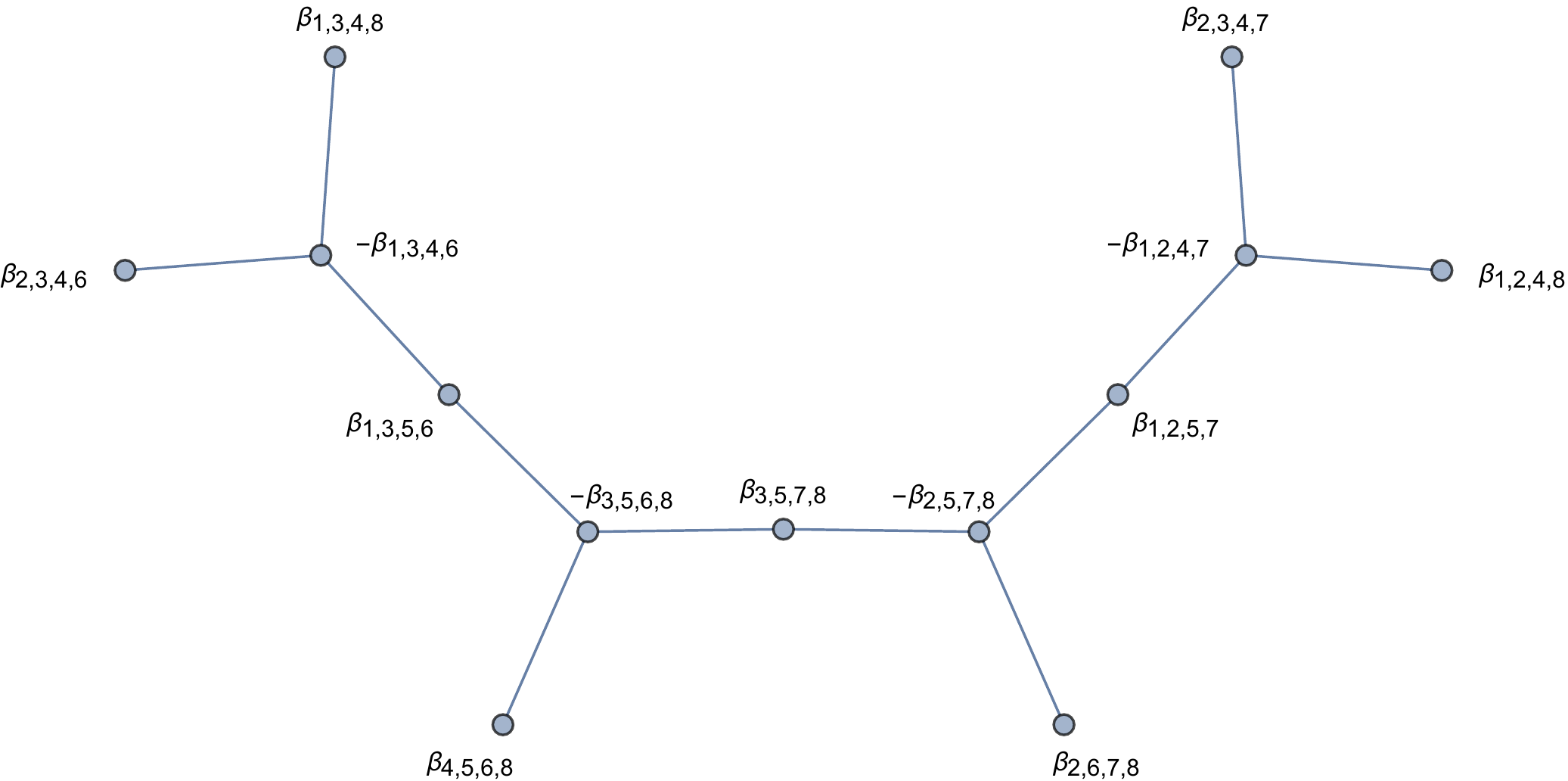}
			\caption{Pinning to the 1-skeleton of $\Delta_{4,8}$ the matroidal weighted blade arrangement from Example \ref{example: 48 large}; as a pole of $m^{(4)}_8$ this evaluates on Root Kinematics to  the noncrossing decomposition $\eta_{\mathbf{c}}\vert_{Rt} = \gamma _{1258}+\gamma _{1357}+\gamma _{2346}+\gamma _{2348}+\gamma _{4678}.$}
		\label{fig:zonotopaltessellationpole48d}
	\end{figure}

\end{example}
	\begin{figure}[h!]
		\centering
		\includegraphics[width=0.6\linewidth]{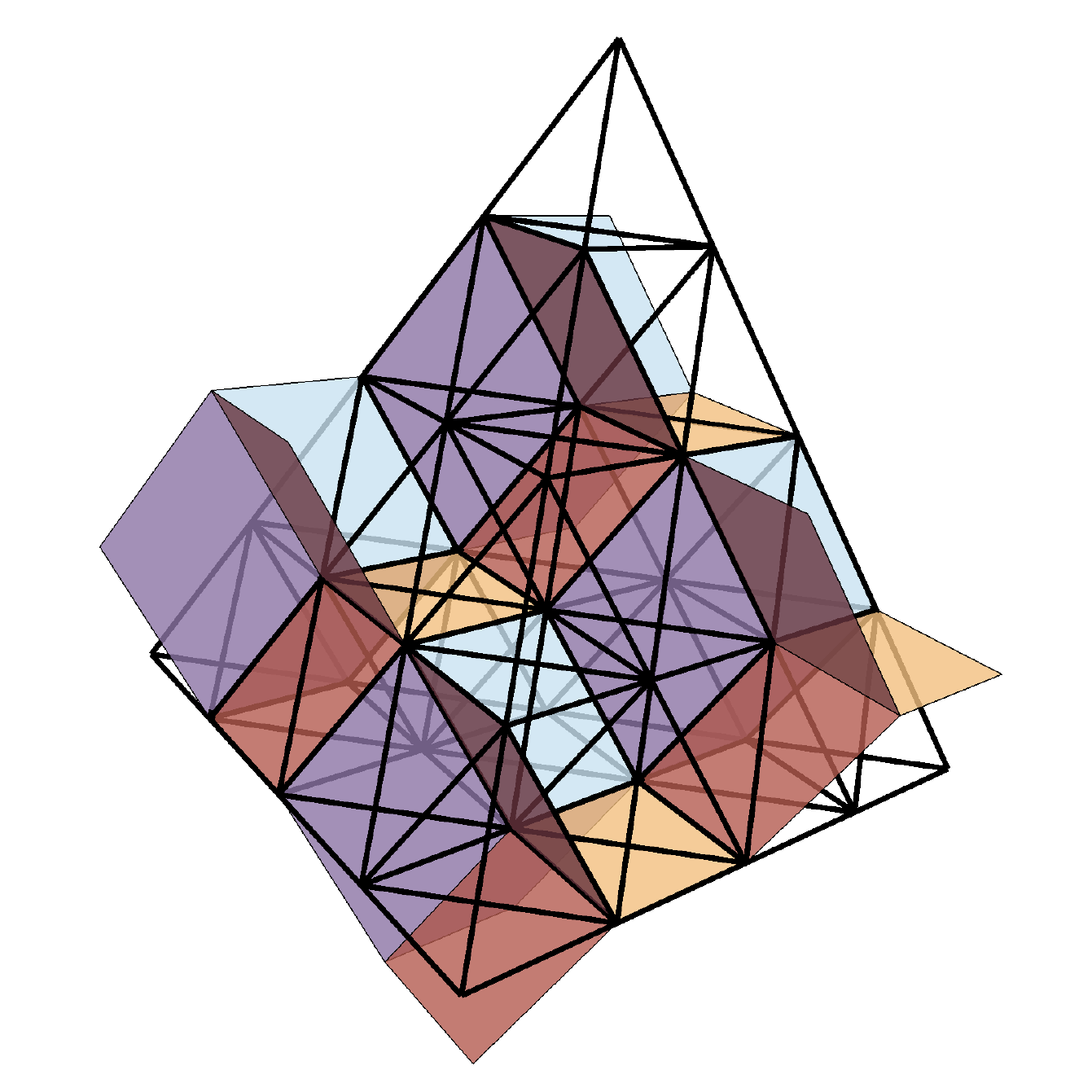}
		\caption{The propagator $\eta_\mathbf{c}$ in Example \ref{example: 48 large}, same as in Figure \ref{fig:zonotopaltessellationpole48d}, shown (schematically) as a blade arrangement in the fourth dilate of a tetrahedron, $4\Delta_{1,4}$.}
		\label{fig:ZonotopalTessellationPole48B}
	\end{figure}

\begin{example}
	Finally, on the Root Kinematics subspace of $\mathcal{K}(5,10)$ we have
	\begin{eqnarray*}
		& & \left(-3 \eta _{1,3,5,7,9}+\eta _{1,3,5,7,10}+\eta _{1,3,5,8,9}+\eta _{1,3,6,7,9}+\eta _{1,4,5,7,9}+\eta _{2,3,5,7,9}\right)\big\vert_{Rt} \\
		& = & \left(\alpha _{1,1}+\alpha _{1,2}+\alpha _{2,2}+\alpha _{2,3}+\alpha _{3,3}+\alpha _{3,4}+\alpha _{4,2}+2 \alpha _{4,3}+2 \alpha _{4,4}+\alpha _{4,5}\right)\\
		& = & \left(\gamma _{1,4,5,8,9}+\gamma _{2,3,6,7,10}\right),
	\end{eqnarray*}
	where $\{\{1,4,5,8,9\}, \{2,3,6,7,10\}\}$ is a noncrossing pair that is not weakly separated.
\end{example}

\begin{example}\label{example: am36 free triangulation computation}
	Let us investigate what happens to two of the GFD's for the generalized biadjoint scalar $m^{(3)}(\mathbb{I}_6,\mathbb{I}_6)$ as studied in \cite{CEGM2019}; one of these corresponds to one of the bipyramidal cones in $\text{Trop}^+G(3,6)$, and the other to a simplicial maximal cone.  We show, ``inside'' the amplitude, how the two cones join and the split apart into two simplicial cones when the kinematics is restricted to the Root Kinematics subspace of the kinematic space $\mathcal{K}(3,6)$.  Then we shall see that upon restriction to $\mathcal{H}_{3,6}$ the amplitude vanishes.  By a direct translation of the poles in \cite{CEGM2019} into the planar kinematic invariants $\eta_J$ we have
	\begin{eqnarray}\label{eq:36 amplitude Noncrossing Example GFDs}
		& & m^{(3)}(\mathbb{I}_6,\mathbb{I}_6)\\
		& = & \frac{\eta_{124} + \eta_{346} + \eta_{124}}{\eta_{124}\eta_{346}\eta_{246}(-\eta_{246} + \eta_{124} + \eta_{346} + \eta_{256})} + \frac{1}{\eta_{256}\eta_{346}\eta_{356}(-\eta_{246} + \eta_{124} + \eta_{346} + \eta_{256})}+ \cdots.\nonumber
	\end{eqnarray}
	For our present purposes this may be considered as definition; we shall study only the two visible terms at the beginning of Equation \eqref{eq:36 amplitude Noncrossing Example GFDs}.  The first corresponds to one of the bipyramids in the positive tropical Grassmannian $\text{Trop}^+G(3,6)$.
	
	Our plan is to reorganize these first two contributions to the amplitude using noncrossing collections.
	
	The first fraction splits apart as
	$$\frac{\eta_{124} + \eta_{346} + \eta_{124}}{\eta_{124}\eta_{346}\eta_{246}(-\eta_{246} + \eta_{124} + \eta_{346} + \eta_{256})}= \frac{1}{\eta_{124}\eta_{346}\eta_{256}\eta_{246}} +\frac{1}{\eta_{124}\eta_{346}\eta_{256}(-\eta_{246} + \eta_{124} + \eta_{346} + \eta_{256})}.$$
	Now
	\begin{eqnarray*}
		& & \frac{1}{\eta_{256}\eta_{346}\eta_{356}(-\eta_{246} + \eta_{124} + \eta_{346} + \eta_{256})} + \frac{1}{\eta_{124}\eta_{346}\eta_{256}(-\eta_{246} + \eta_{124} + \eta_{346} + \eta_{256})}\\
		& = & \frac{\eta_{124} + \eta_{356}}{\eta_{256}\eta_{346}\eta_{356}(-\eta_{246} + \eta_{124} + \eta_{346} + \eta_{256})},
	\end{eqnarray*}
	but evaluating on Root Kinematics, see Definition \ref{def: Root Kinematics}, where $\eta_J\vert_{Rt} = \gamma_J$, gives
	$$(-\eta_{246} + \eta_{124} + \eta_{346} + \eta_{256})\vert_{Rt} = (\eta_{124} + \eta_{356})\vert_{Rt},$$
	so that 
	\begin{eqnarray*}
		\frac{\eta_{124} + \eta_{356}}{\eta_{124}\eta_{256}\eta_{346}\eta_{356}(-\eta_{246} + \eta_{124} + \eta_{346} + \eta_{256})}\bigg\vert_{Rt} & = & \frac{1}{\eta_{124}\eta_{256}\eta_{346}\eta_{356}}\bigg\vert_{Rt}.
	\end{eqnarray*}
	Consequently, Equation \eqref{eq:36 amplitude Noncrossing Example GFDs}, restricted to the Root Kinematics subspace, becomes
	\begin{eqnarray}
		m^{(3)}_6(\mathbb{I}_6,\mathbb{I}_6)\bigg\vert_{Rt} & = & \left(\frac{1}{\eta_{124}\eta_{346}\eta_{256}\eta_{246}} + \frac{1}{\eta_{124}\eta_{256}\eta_{346}\eta_{356}} + \cdots \right) \bigg\vert_{Rt},
	\end{eqnarray}
	and we recognize that both 
	$$\{124,346,256,246\} \text{ and } \{124,256,346,356\}$$
	are maximal noncrossing collections in $\mathbf{NC}_{3,6}$, noting, however, that the second collection is not weakly separated.
	
	To summarize: if we repeat this process for all Generalized Feynman Diagrams, we see that on the Root Kinematics subspace, $m^{(3)}_6(\mathbb{I}_6,\mathbb{I}_6) $ is completely determined by the noncrossing complex $\mathbf{NC}_{3,6}$.  In actuality, due to the many linear relations among the restricted poles $\eta_J\vert_{Rt}$ we have the dramatic simplification
	\begin{eqnarray*}
		m^{(3)}_{6}(\mathbb{I}_6,\mathbb{I}_6)\vert_{Rt} & = & \frac{(\eta_{134} + \eta_{245} + \eta_{356})}{\eta_{134}\eta_{245} \eta_{356}}\frac{(\eta_{124} + \eta_{235} + \eta_{346})}{\eta_{124}\eta_{235} \eta_{346}}\bigg\vert_{Rt}\\
		& = & \frac{\alpha_{1,1}+ \alpha_{1,2}+\alpha_{1,3}}{\alpha_{1,1}\alpha_{1,2}\alpha_{1,3}}\frac{\alpha_{2,1}+ \alpha_{2,2}+\alpha_{2,3}}{\alpha_{2,1}\alpha_{2,2}\alpha_{2,3}}.
	\end{eqnarray*}
	Indeed, the amplitude vanishes identically when $\alpha$ is restricted to the subspace $\mathcal{H}_{3,6}$ of $\mathbb{R}^{(3-1)\times (6-3)}$, where 
	$$\alpha_{i,1} + \alpha_{i,2} + \alpha_{i,3} =0$$
	for $i=1,2$.  This can be seen for the potential function as well: when $\sum_{t=1}^3\alpha_{i,t}=0$, then one can check that the function has no critical points!

\end{example}

\section{Resolving Maximal Minors}\label{sec: minors resolved delta polynomials}

Here we study the evaluation of the minors $p_J$ on the image of the positive parameterization, and we introduce their resolutions using compound determinants.

For any triple of sets $(\{i_1,i_2\},\{i_3,i_4\},\{i_5,i_6\})$, where $i_1,\ldots, i_6\in \{1,\ldots, n\}$ are distinct, define a compound determinant
$$X_{(i_1,i_2),(i_3,i_4),(i_5,i_6)} = \det\left(u_{i_1}\times u_{i_2},u_{i_3}\times u_{i_4},u_{i_5}\times u_{i_6}\right).$$
where we are temporarily denoting by $u_j$ the $j$th column of a given $3\times n$ matrix.

For any $1\le i<j<k\le n$, as a generalization of the construction in \cite{ScottGrassmannian} for the ``A'' cluster variable, define a compound determinant
$$A_{ijk} =  \det\left((u_1\times u_2)\times (u_i\times u_{i+1}),u_j,(u_{j+1}\times u_{j+2})\times (u_{k}\times u_1)  \right).$$

\begin{defn}\label{defn: resolved Pluckers}
	For each $2\le j<k\le n$ define the \textit{resolved minor} 
	$$\hat{p}_{1jk} = p_{1jk},$$
	while for $2\le i<j<k\le n$ put 
	$$\hat{p}_{ijk} = \frac{A_{ijk}}{p_{1,2,i+1}p_{1,j+1,j+2}}.$$
\end{defn}
\begin{rem}
	There is an apparent problem in the formulation of $\hat{p}_{i,n-1,n}$, as the compound determinant $A_{i,n-1,n}$ involves the column $u_{n+1}$, but $u_{n+1}$ is not a column in $X(3,n)$; however, notice that 
	\begin{eqnarray*}
		A_{i,j,j+1} & = &  X_{(1,2),(i,i+1),(j,j+1)}p_{1,j+1,j+2},
	\end{eqnarray*}
	so that in particular 
	\begin{eqnarray*}
		\hat{p}_{i,n-1,n} & = & \frac{A_{i,n-1,n}}{p_{1,2,i+1}p_{1,n,n+1}}\\
		& = & \frac{X_{(1,2),(i,i+1),(n-1,n)}}{p_{1,2,i+1}}
	\end{eqnarray*}
	so the factor $p_{1,n,n+1}$ cancels and the column $u_{n+1}$ plays no role.
	
	Consequently, in a slight abuse of notation we still write $\hat{p}_{i,n-1,n}$ in terms of $A_{i,n-1,n}$.
\end{rem}

\begin{prop}\label{prop: (3,n) resolved minors}
	The resolved minors $\hat{p}_{ijk}$ can be simplified (and enumerated) as follows:
	\begin{enumerate}
		\item $\binom{n-1}{2}+(n-3) + (n-4) + (n-4) + (n-5) = \frac{1}{2} n (n+5)-15$ resolved Plucker coordinates satisfying $\hat{p}_{i,j,k} = p_{i,j,k}$.
		These are as follows:
		$$\left\{p_{1,i,j}: 2\le i<j\le n \right\},\ \left\{p_{2,i,i+1}: i=3,4,\ldots, n-1 \right\},\ \ \left\{p_{i,i+1,i+3}: i=2,\ldots, n-3 \right\}, $$
		$$\{p_{i,i+1,i+2}: i=3,\ldots, n-2\}, \{p_{2,i,i+2}: 4\le i<i+2\le n\}$$
		\item $\binom{n-4}{2}$ compound determinants of type $X$.   If $\{a,b,c\} = \{i,j,j+1\}$ for 
		$$3\le i<i+1<j<j+1\le n,$$
		then 
		\begin{eqnarray*}
			\hat{p}_{i,j,j+1}\big\vert_{BCFW} & = & \frac{X_{(1,2),(i,i+1),(j,j+1)}}{p_{1,2,i+1}}\big\vert_{BCFW}\\
			& = & \delta^{(1)}_{i-1,j-2}x_{2,i-2}x_{2,j-2}.
		\end{eqnarray*}
		\item $\binom{n-4}{2}$ compound determinants of type $X$.  If $\{a,b,c\} = \{i,i+1,j\}$ for 
		$$2\le i<i+3<j\le n,$$
		then  
		\begin{eqnarray*}
			\hat{p}_{i,i+1,j}\big\vert_{BCFW} & = & \frac{X_{(j,1),(i,i+1),(i+2,i+3)}}{p_{1,i+2,i+3}}\big\vert_{BCFW}\\
			& = &  \delta^{(1)}_{i-1,i,j-2}x_{2,i-2}.
		\end{eqnarray*}
		\item $\binom{n-5}{2}$ compound determinants of type $X$.  If $\{a,b,c\} = \{2,i,j\}$ for 
		$$\left\{(i,j): 4\le i< i+2 < j \le n\right\},$$
		then 
		\begin{eqnarray*}
			\hat{p}_{2,i,j}\big\vert_{BCFW} & = &  \frac{X_{(j,1),(2,i),(i+1,i+2)}}{p_{1,i+1,i+2}}\big\vert_{BCFW}\\
			& = & \delta^{(1)}_{1,i-1,j-2}.
		\end{eqnarray*}
		\item $\binom{n-4}{3}$ compound determinants of type $A_{i,j,k}$. For any 
		$$\{(i,j,k): 3\le i<i+1<j<j+2\le k\le n\},$$
		then
		\begin{eqnarray*}
			\hat{p}_{ijk}\big\vert_{BCFW} & = & \frac{A_{i,j,k}}{p_{1,2,i+1}p_{1,j+1,j+2}}\bigg\vert_{BCFW}\\
			& = & \delta^{(1)}_{i-1,j-1,k-2}x_{2,i-2}
		\end{eqnarray*}
	\end{enumerate}
	Tabulating the contributions gives all $\binom{n}{3}$ resolved minors, by inspection we have expressions for all $\binom{n}{3} - (3)(n-3)-1$ (as in the enumeration given in Proposition \ref{prop: number Minkowski summands}) face polynomials going into $\mathbb{K}^{(k)}_{n-k}$ in terms of minors or compound determinants.
\end{prop}

There is an elegant purely combinatorial condition encoding when we have $\hat{p}_J \not= p_J$.
\begin{cor}\label{cor: nontrivially resolved Pluckers}
	Given $J \in \binom{\lbrack n\rbrack}{3}$ then we have $\hat{p}_J \not= p_J$ if and only if there exists $I \in \binom{\lbrack n\rbrack}{3}$ such that $I \prec J$ ($I$ is lexicographically smaller than $J$) and the pair $\{I,J\} \in \mathbf{NC}_{3,n}$ is noncrossing but not weakly separated.
\end{cor}

Indeed, for the rank $k=3$ noncrossing complex $\mathbf{NC}_{3,n}$ it is not difficult to see that there are exactly $2\binom{n}{6}$ noncrossing pairs $(I,J)$ that are not weakly separated.  Namely, for each $L = \{\ell_1,\ldots, \ell_6\} \in \binom{\lbrack n\rbrack}{6}$ with $1\le \ell_1< \cdots <\ell_6\le n$ there are exactly two such pairs:
$$\left(\{\ell_1,\ell_2,\ell_4\}, \{\ell_3,\ell_5,\ell_6\}\right),\ \ \left(\{\ell_1,\ell_4,\ell_5\}, \{\ell_2,\ell_3,\ell_6\}\right).$$
Then one sees immediately that the set of all triples of the form either $\{\ell_2,\ell_3,\ell_6\}$ or $\{\ell_3,\ell_5,\ell_6\}$, for some $L\in \binom{\lbrack n\rbrack}{6}$, coincides with the elements $\hat{p}_{a,b,c}$ such that $\hat{p}_{abc} \not=p_{abc}$ in Proposition \ref{prop: (3,n) resolved minors}.

\begin{cor}
	For $(k,n) = (3,n)$ the number of nontrivially resolved minors $\hat{p}_{ijk}$ on $G(3,n)$ is given by 
	$$N_n = \binom{n-5}{3}+2(n-5)^2.$$
\end{cor}
For example, in the case (3,n) as above, one finds the enumerations
$$N_n = 2, 8, 19, 36, 60, 92, 133,\ldots$$
of nontrivially resolved Plucker coordinates for $n\ge 6$.

\begin{example}\label{example: Plucker resolution explicit 36}
	Let us evaluate some of the resolved Plucker coordinates on the image of the positive parameterization of $X(3,6)$.  In $\mathbf{NC}_{3,6}$ there are two noncrossing pairs which are not weakly separated.  These are
	$$\{\{1,2,4\},\{3,5,6\}\}\text{ and } \{\{1,4,5\},\{2,3,6\}\},$$
	and the following Plucker coordinates require resolution:
	$$p_{356},p_{236}.$$
	For the minor $p_{236}$ we have
	$$\hat{p}_{236}\vert_{BCFW} = x_{1,1}x_{2,123}+x_{1,2}x_{2,23},$$
	while prior to resolution it is
	$$p_{236}\vert_{BCFW} = x_{1,1}x_{2,123} + x_{1,2}x_{2,23}+x_{1,3}x_{2,3}.$$
	Similarly,
	$$p_{356}\vert_{BCFW} = \left(x_{1,2} x_{2,1}+x_{1,3} x_{2,1}+x_{1,3} x_{2,2}\right) x_{2,3}$$
	is resolved to
	$$\hat{p}_{356}\vert_{BCFW} = (x_{1,2}+x_{2,3})x_{2,1}x_{2,3}.$$
	
	These can be expressed in terms of Plucker coordinates as
	$$\hat{p}_{236} = p_{236} - \frac{p_{123}p_{456}}{p_{145}}.$$
	and 
	$$\hat{p}_{356} = p_{356} - \frac{p_{123}p_{456}}{p_{124}}.$$
	None of the other Plucker coordinates here requires resolution; indeed, $\{\{1,2,4\},\{3,5,6\}\}$ and $\{\{1,4,5\},\{2,3,6\}\}$ are the only pairs in $\mathbf{NC}_{3,6}$ which are noncrossing but not weakly separated.
	
\end{example}

\begin{example}\label{example: Plucker resolution explicit 37}
	There are 14 noncrossing pairs $\{I,J\}$ in $\mathbf{NC}_{3,7}$ that are not weakly separated.  These are:
	$$
	\begin{array}{cc}
		\{1,2,4\} & \{3,5,6\} \\
		\{1,2,4\} & \{3,5,7\} \\
		\{1,2,4\} & \{3,6,7\} \\
		\{1,2,5\} & \{3,6,7\} \\
		\{1,2,5\} & \{4,6,7\} \\
		\{1,3,5\} & \{4,6,7\} \\
		\{1,4,5\} & \{2,3,6\} \\
		\{1,4,5\} & \{2,3,7\} \\
		\{1,4,6\} & \{2,3,7\} \\
		\{1,5,6\} & \{2,3,7\} \\
		\{1,5,6\} & \{2,4,7\} \\
		\{1,5,6\} & \{3,4,7\} \\
		\{2,3,5\} & \{4,6,7\} \\
		\{2,5,6\} & \{3,4,7\} \\
	\end{array}.$$
	From this we read off the eight Plucker coordinates on $X(3,7)$ which require resolution.  These are labeled by the sets
	$$\{2,3,6\},\{2,3,7\},\{2,4,7\},\{3,4,7\},\{3,5,6\},\{3,5,7\},\{3,6,7\},\{4,6,7\}.$$
	Then for instance we resolve 
	$$p_{357}\vert_{BCFW} = (x_{1,23}x_{2,1} + x_{1,3}x_{2,2})x_{2,3} + (x_{1,234}x_{2,1} + x_{1,34}x_{2,2})x_{2,4},$$
	using Definition  \ref{defn: resolved Pluckers}, to 
	\begin{eqnarray*}
		\hat{p}_{357}\vert_{BCFW} & = & \frac{A_{357}}{p_{1,2,4}p_{1,6,7}}\bigg\vert_{BCFW}\\
		& = & (x_{1,2} x_{2,3}+x_{1,3} x_{2,3}+x_{1,2} x_{2,4}+x_{1,3} x_{2,4}+x_{1,4} x_{2,4})x_{2,1}.
	\end{eqnarray*}
\end{example}

\section{Facet classification for the PK polytope}\label{sec:facet PK polytope}
	
	Lemma \ref{lem: minimum value} gives one direction of the proof that the facet inequalities that were claimed to define $\Pi^{(k)}_{n-k}$ in Equation \eqref{eqn: facet inequalities intro} are indeed correct.

	\begin{lem}\label{lem: minimum value}
		For any $J \in \binom{\lbrack n\rbrack}{k}^{nf}$ and any $j\in \{1,\ldots, n-k-1\}$, then the minimum value of $\gamma_J$ on any $\text{Newt}(P_i)$ is 0.  The minimum value of $\gamma_J(\alpha)$ on the polyhedron $\text{Newt}(Q_j)$ is either 0 or 1.  Specifically, the minimum value of $\gamma_J(\alpha)$ on $Q_t$ is one exactly when we have 
		$$j_1\le t< t+k+1\le j_k,$$
		and otherwise this minimum is zero.		
		The minimum value of $\gamma_J$ on $\Pi^{(k)}_{n-k}$ is $-1$.
		
		We have the inclusion of the Newton polytope
		$$\Pi^{(k)}_{n-k} = \text{Newt}\left(\frac{P_1\cdots P_{k-1}Q_1\cdots Q_{n-k-1}}{\prod_{(i,j)\in \lbrack 1,k-1\rbrack \times \lbrack 1,n-k\rbrack}x_{i,j}}\right),$$
		into the polyhedron
		$$\left\{(\alpha) \in \mathbb{R}^{(k-1)\times (n-k)}: \sum_{j=1}^{n-k}\alpha_{i,j}=0;\  \gamma_J(\alpha)+1\ge 0,\ J \in \binom{\lbrack n\rbrack}{k}^{nf}\right\},$$
		and in particular, every affine hyperplane $\gamma_J+1=0$ contains a facet of $\Pi^{(k)}_{n-k}$.

	\end{lem}
	
\begin{proof}
	Fix $J \in \binom{\lbrack n\rbrack}{k}^{nf}$ and choose any $j_0\in \{1,\ldots, n-k-1\}$.
	
	Mindful of the interpretation of the variables $\alpha_{i,j}$ as coordinate functions on a matrix of size $(k-1)\times (n-k)$, we see that the Newton polytope of any fixed polynomial $Q_{j_0}$, for $j_0\in \{1,\ldots, n-k-1\}$, is a simplex of dimension $k-1$ which involves only the adjacent columns $j_0$ and $j_0+1$.  To be precise, we have that 
	$$\text{Newt}(Q_{j_0}) = \left\{\sum_{i=1}^{k-1}\alpha_{i,j_0}e_{i,j_0} + \alpha_{i,j_0+1} e_{i,j_0+1}: 1\ge \alpha_{1,j_0} \ge \alpha_{2,j_0} \ge \cdots \ge \alpha_{k-1,j_0} \ge 0,\ \alpha_{i,j_0} + \alpha_{i,j_0+1}=1\right\}.$$
	As for the linear function $\gamma_J$, notice that it involves at most one $\alpha$-coordinate from each of the columns $j_0$ and $j_0+1$.  Indeed, there are the following five possibilities.  If, in the expansion of $\gamma_J$ the columns $j_0$ and $j_0+1$ do not appear at all, then $\gamma_J$ takes the \textit{constant} value 0 on the whole simplex $\text{Newt}(Q_{j_0})$.  The remaining possibilities are as follows:
	\begin{eqnarray}
		\gamma_{J} & = &  \cdots  + \alpha_{i,j_0}\label{eq:line1},\\
		\gamma_{J} & = & \cdots  + \alpha_{i,j_0} + \alpha_{i,j_0+1} + \cdots\label{eq:line2}, \\
		\gamma_{J} & = & \cdots  + \alpha_{i,j_0} + \alpha_{i+t,j_0+1} + \cdots\label{eq:line3}, \\
		\gamma_{J} & = & \alpha_{i,j_0+1} + \cdots\label{eq:line4},
	\end{eqnarray}
	where $i+t\le k-1$.  Here the ellipses $\cdots$ are sums of $\alpha_{i,j}$'s (in the indicated intervals) that do not involve the columns $j_0,j_0+1$.  In the case of Equations \eqref{eq:line1} and \eqref{eq:line4}, $\gamma_J$ ranges between 0 and 1 on $\text{Newt}(Q_{j_0})$, and so the minimum value is zero.  In the case of Equation \eqref{eq:line2}, $\gamma_J$ is identically 1 on $\text{Newt}(Q_{j_0})$.  But in the case of Equation \eqref{eq:line3}, $\gamma_J$ ranges between 1 and 2 on $\text{Newt}(Q_{j_0})$.  In particular, the minimum value is 1.  This completes the proof that the minimum value of $\gamma_J$ on any $Q_j$ is either 0 or 1.
	
	It remains to observe that the minimum value of $\gamma_J$ on the polytope $\text{Newt}\left(\prod_{j=1}^{n-k-1}Q_j\right)$ is $m = j_k-j_1-k$; from this, dividing by the monomial factor $\prod_{(i,j) \in \lbrack 1,k-1\rbrack \times \lbrack 1,n-k\rbrack}x_{i,j}$ shifts the minimum value of $\gamma_J$ from $m$ to $-1$.  This also proves the claimed inclusion.  
	
	Now we already know that $\gamma_J+1=0$ defines a face of $\Pi^{(k)}_{n-k}$ of some positive codimension, and it is easy to see that this face will again be a Minkowski sum of simplices; a somewhat detailed -- but instructive -- dimension count can be used to show that the intersection of the affine hyperplane $\gamma_J+1=0$ with $\Pi^{(k)}_{n-k}$ has the desired codimension one.  For this, it is convenient to work with the $\prod_{j=1}^{n-k-1}Q_j$ itself, rather than the given Laurent polynomial.  
	
	There are two main steps.  First, we calculate the faces, respectively $P'_j$ and $Q'_i$, of each $P_j$ and each $Q_i$ that minimize $\gamma_J$; we know that on each, this minimum value is either zero or one.  In this way we write the facet as a Newton polytope
	$$\text{Newt}\left(P'_1\cdots P'_{k-1}Q'_1\cdots Q'_{n-k-1}\right).$$
	Second, by reorganizing the product $(P'_1\cdots P'_{k-1}Q'_1\cdots Q'_{n-k-1})$ and taking the Newton polytope, we find a subset of the Newton polytope which has the desired codimension one.  Let us illustrate the method by calculating the Newton polytope representation of the facet of $\Pi^{(4)}_5$ that minimizes $\gamma_{1458} = \alpha_{1,1 } + \alpha_{1,2} + \alpha_{3,3} + \alpha_{3,4}$.  Here our ambient space for the coordinate functions $\alpha_{i,j}$ for $(i,j) \in \lbrack 1,3\rbrack \times \lbrack 1,5\rbrack$ is $\mathcal{H}_{4,9}$. 
	
	We find
	\begin{eqnarray*}
	P'_1 & = & x_{1,3} + x_{1,4} + x_{1,5}\\
	P'_2 & = & x_{2,1} + x_{2,2} + x_{2,3} + x_{2,4} + x_{2,5}\\
	P'_3 & = & x_{3,1} + x_{3,2} + x_{3,5}\\
	Q'_1 & = & x_{1,1}x_{2,1}x_{3,1} + x_{1,1}x_{2,1}x_{3,2}+x_{1,1}x_{2,2}x_{3,2} + x_{1,2}x_{2,2}x_{3,2}\\
	Q'_2 & = & x_{1,2}x_{2,2}x_{3,2} + x_{1,3}x_{2,3}x_{3,3}\\
	Q'_3 & = & x_{1,3}x_{2,3}x_{3,3} + x_{1,3}x_{2,3}x_{3,4}+x_{1,3}x_{2,4}x_{3,4} + x_{1,4}x_{2,4}x_{3,4}\\
	Q'_4 & = & (x_{1,4} x_{2,4}+ x_{1,4}x_{2,5} + x_{1,5}x_{2,5})x_{3,5}.
\end{eqnarray*}
\begin{figure}[h!]
	\centering
	\includegraphics[width=0.85\linewidth]{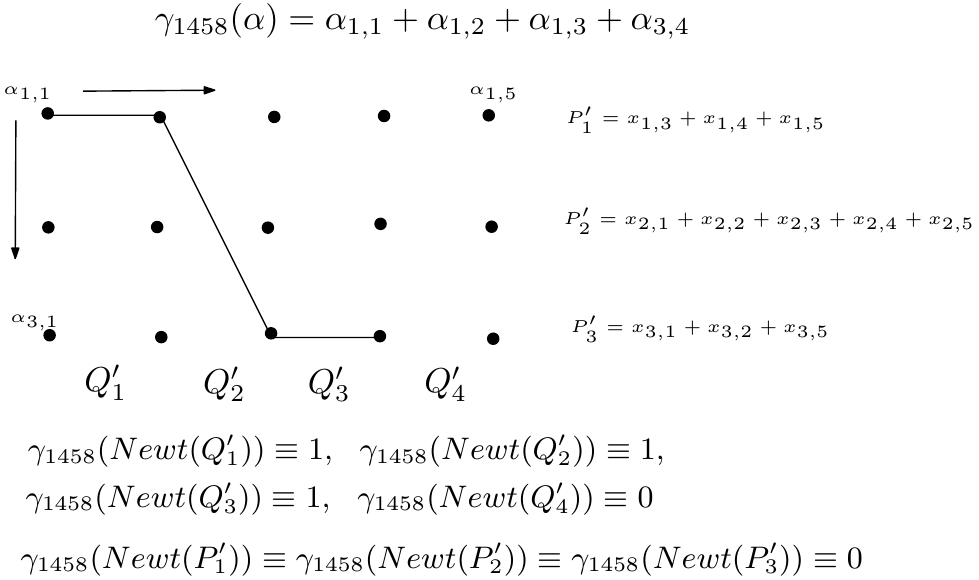}
	\caption{The function $\gamma_{1458}$ evaluates to the constant $3$ on the face $\text{Newt}(P'_1P'_2P'_3Q'_1Q'_2Q'_3Q'_4)$ of $\Pi^{(4)}_5$.  By a dimension count this face has codimension 1, and dimension $(4-1)(9-4-1)-1=11$.}
	\label{fig:facetgridexample}
\end{figure}
	Then for the facet, (before translating to the origin) we obtain
	\begin{eqnarray*}
		\text{Newt}\left(P'_1P'_2 P'_3 Q'_1 Q'_2 Q'_3 Q'_4\right)
		& = & \text{Newt} \left(Q'_1\right) \\
		& \boxplus & \text{Newt} \left((Q'_2)(x_{1,3} + \cdots )(\cdots +x_{2,2} + x_{2,3} + \cdots )(\cdots  + x_{3,2})\right)\\
		& \boxplus & \text{Newt}\left(Q'_3 \right)\\
		& \boxplus & \text{Newt}\left(Q'_4\right).
	\end{eqnarray*}
	Among these, by inspection of the ellipses $\cdots $ we can see that the four Newton polytopes are transverse, so that the dimension of $$\text{Newt}\left(P'_1P'_2 P'_3 Q'_1 Q'_2 Q'_3 Q'_4\right)$$ is the sum of the dimensions of the Minkowski summands; moreover, the Newton polytopes in the first three lines all have dimension three, while the dimension of the Newton polytope in the fourth line has been reduced from three (for $Q_4$) to two (for $Q'_4$).  Consequently, the whole facet has codimension one, as desired.
	
	Here we have used that all coefficients are nonnegative in our analysis of the Newton polytopes.
	
	Finally, by direct inspection (for instance, in the diagram in Figure \ref{fig:facetgridexample}, one has to count which pairs of adjacent columns are covered by the path connecting the variables $\alpha_{i,j}$ on the grid) one finds that when the integer $t$ satisfies $j_1\le t$ and $t+k+1\le j_k$, then the minimum value of $\gamma_J(\alpha)$ on $Q_t$ is equal to one, and otherwise the minimum is zero.
\end{proof}

	\begin{lem}\label{lem: minimized simultaneously gamma}
	Given any noncrossing pair $(I,J) \in \mathbf{NC}_{k,n}$, then $\gamma_I$ and $\gamma_J$ can be simultaneously minimized on any $\text{Newt}(Q_j)$.  That is to say, there exists a point $(\alpha_0) \in \text{Newt}(Q_j)$ such that $\gamma_I(\alpha) \ge \gamma_I (\alpha_0)$ and $\gamma_J(\alpha) \ge \gamma_J (\alpha_0)$ for all $\alpha \in \text{Newt}(Q_j)$.
\end{lem}
\begin{proof}
	Given any $I \in \binom{\lbrack n\rbrack}{k}^{nf}$, let $c_{I}(Q_j) = -\min\{\gamma_I(\alpha): (\alpha) \in \text{Newt}(Q_j) \}$.  Let us suppose that $(I,J) \in \mathbf{NC}_{k,n}$ and $I',J' \in \binom{\lbrack n\rbrack}{k}$ such that 
	$$\gamma_I + \gamma_J = \gamma_{I'} + \gamma_{J'}.$$
	First we claim that 
	$$c_{I'}(Q_t) + c_{J'}(Q_t) \ge c_{I}(Q_t) + c_{J}(Q_t)$$
	for each $t=1,\ldots, n-k-1$. 
	
	Indeed, by Lemma (11.1) the value of each constant $c_L(Q_j)$ with $L = \{\ell_1,\ldots, \ell_l\}$ is independent of the interior indices $\ell_{2},\ldots, \ell_{k-1}$. 
	
	This means, in particular, that we have reduced our calculation to the case $k=2$.  Indeed, one has the following straight forward computation: to verify that whenever $\{i_1i_2,j_1j_2\}$ is noncrossing and we have
	$$\gamma_{i_1'i_2'} + \gamma_{j'_1j'_2} = \gamma_{i_1i_2} + \gamma_{i_2j_2},$$
	then
	$$\left(c_{i'_1i'_2}(x_t+x_{t+1}) + c_{j'_1j'_2}(x_t+x_{t+1})\right) - \left(c_{i_1i_2}(x_t+x_{t+1}) + c_{j_1j_2}(x_t+x_{t+1})\right) \in \{0,1\}.$$
	It follows that the normal fan to the PK polytope is a coarsening of the noncrossing complete simplicial fan of Theorem \ref{thm: noncrossing subdivision Wkn}.  Consequently $\gamma_I$ and $\gamma_J$ can be simultaneously minimized on $\Pi^{(k)}_{n-k}$.
\end{proof}

\begin{rem}
	Repeating Lemma \ref{lem: minimized simultaneously gamma} when the PK polytope is replaced by the PK associahedron $\mathbb{K}^{(k)}_{n-k}$, and the polynomials $P_i,Q_j$ are completed to the full set of polynomials $\delta^{(i)}_{J}$ (or, equivalently, with the planar face polynomials $\tau_J$) could be an interesting question to pursue, but it is beyond the scope of the paper.  This would give a proof that the face poset of the $\mathbb{K}^{(k)}_{n-k}$ is (anti)-isomorphic to the noncrossing complex $\mathbf{NC}_{k,n}$, and in particular that $\mathbb{K}^{(k)}_{n-k}$ is an explicit metrization of the Grassmann associahedron of \cite{GrassmannAssociahedron}.  The problem is left to future work.
\end{rem}

\begin{conjecture}
	The face poset of $\mathbb{K}^{(k)}_{n-k}$ is anti-isomorphic to the noncrossing complex $\mathbf{NC}_{k,n}$.
\end{conjecture}

	Corollary \ref{cor: facets Pikn} completes our analysis of the facets of the PK polytope: it says that the facet hyperplanes of the PK polytope are given by exactly the equations $\gamma_J+1=0$, as $J$ runs over the set of nonfrozen subsets $\binom{\lbrack n\rbrack}{k}^{nf}$.  
	\begin{cor}\label{cor: facets Pikn}
		The PK polytope has exactly $\binom{n}{k}-n$ facets, one in each affine hyperplane $\gamma_J+1=0$, for $J \in \binom{\lbrack n\rbrack}{k}^{nf}$.  In particular, we have
\begin{eqnarray}\label{eq:facet inequalities Pikn}
	\Pi^{(k)}_{n-k} & = & \left\{(\alpha) \in \mathbb{R}^{(k-1)\times (n-k)}: \sum_{j=1}^{n-k}\alpha_{i,j}=0;\  \gamma_J(\alpha)+1\ge 0,\ J \in \binom{\lbrack n\rbrack}{k}^{nf}\right\},
\end{eqnarray}
		and every inequality is facet-defining.
	\end{cor}
\begin{proof}
	From Lemma \ref{lem: minimized simultaneously gamma} it follows that the normal fan to the PK polytope is a coarsening of the noncrossing complete simplicial fan of Theorem \ref{thm: noncrossing subdivision Wkn}.  In other words, each facet of $\Pi^{(k)}_{n-k}$ minimizes one of the generalized roots $\gamma_J$.  Taking into account Lemma \ref{lem: minimum value} completes the proof.
\end{proof}

By combining Corollary \ref{cor: facets Pikn} with Corollary \ref{cor: triangulation root polytope} and using that $\mathcal{R}^{(k)}_{n-k}$ and $\Pi^{(k)}_{n-k}$ are in duality, we use \cite[Claim 1]{AHL2019Stringy} to deduce Corollary \ref{cor: mkn value PK Catalan}.

\begin{cor}\label{cor: mkn value PK Catalan}
At the Planar Kinematics (PK) point $(s) \in \mathcal{K}(k,n)$, characterized by $\eta_J(s) = 1$ for all $J\in \binom{\lbrack n\rbrack}{k}^{nf}$, that is to say
\begin{eqnarray*}\label{planarK cor}
	s_{12\ldots k}= s_{23\ldots k+1}=\ldots =s_{n1\ldots k-1}& =& 1, \\
	s_{n1\ldots k-2,k} = s_{12\ldots k-1,k+1}=\ldots = s_{n-1,n\ldots k-3,k-1}& = & -1,\nonumber
\end{eqnarray*}
where all other $s_J$ are set to zero, then we have that 
$$m^{(k)}_n(\mathbb{I},\mathbb{I}) = C^{(k)}_{n-k}.$$
\end{cor}

\begin{proof}
	Equation \eqref{eq:facet inequalities Pikn} in Corollary \ref{cor: facets Pikn} says that the facets of $\Pi^{(k)}_{n-k}$ are given by exactly the $\binom{n}{k}-n$ equations $\gamma_J + 1 \ge 0$.  On the other hand, the root polytope $\mathcal{R}^{(k)}_{n-k}$ is the convex hull of all roots $\gamma_J$.  Now both $\Pi^{(k)}_{n-k}$ and $\mathcal{R}^{(k)}_{n-k}$ are lattice polytopes; and indeed, the facet inequalities $\gamma_J + 1\ge0$ define the dual to $\mathcal{R}^{(k)}_{n-k}$.  In fact, here duality is an involution and so 
	$$(\Pi^{(k)}_{n-k})^\ast = \mathcal{R}^{(k)}_{n}.$$
	Now, from \cite[Claim 1]{AHL2019Stringy} we have an expression for $m^{(k)}_n$ in terms of the relative volume of the dual polytope, that is to say
	$$m^{(k)}_n(\mathbb{I},\mathbb{I}) = \text{Vol}\left((\Pi^{(k)}_{n-k})^\ast\right) = \text{Vol}\left(\mathcal{R}^{(k)}_{n-k}\right) = C^{(k)}_{n-k},$$
	where in the last equality we have invoked Corollary \ref{cor: triangulation root polytope}, and we are done.
\end{proof}

\section{Generalized Worldsheet associahedra}\label{sec: generalized worldsheet associahedron}

We define the generalized worldsheet associahedron using a set of binary equations and then propose a parameterization of the solution space, in the planar face ratios $u_{ijk}$.
\begin{defn}
	The generalized worldsheet associahedron $\mathcal{W}^+_{3,n}$ is the set of all points $(u_J) \in \lbrack 0,1\rbrack^{\binom{n}{3}-n}$ such that the following set of equations hold: for each $J\in \binom{\lbrack n\rbrack}{3}^{nf}$, one equation
	\begin{eqnarray}\label{eq: binary relations worldsheet associahedron defn}
		u_J & =& 1 - \prod_{\{I:\ (I,J) \not\in\mathbf{NC}_{3,n}\}}u^{c_{I,J}}_I,
	\end{eqnarray}
	where for any crossing pair $(i_1i_2i_3,j_1j_2j_3) \not\in \mathbf{NC}_{3,n}$, 
	\begin{eqnarray}\label{eq: binary equations exponents worldsheet associahedron defn}
		c_{(i_1,i_2,i_3),(j_1,j_2,j_3)} = \begin{cases}
			2 & \text{if } i_1<j_1<i_2<j_2<i_3<j_3\text{ or } j_1<i_1<j_2<i_2<j_3<i_3\\
			1 & \text{otherwise}.
		\end{cases}
	\end{eqnarray}
\end{defn}
Let us recall the construction of the planar face polynomials $\tau_{ijk}$ and the planar face ratios $u_{ijk}$. We have $\tau_{1,2,j}=1$ for all $j=3,\ldots, n$ while for $\{1,j,k\}$ such that $3\le j<k \le n$, then
\begin{eqnarray*}
	\tau_{1,j,k} & := & \sum_{a \in \lbrack j-1,k-2\rbrack} x_{2,a}.
\end{eqnarray*}
Whenever $\{i,j,k\}$ satisfies $2\le i\le j<k\le n$, then
\begin{eqnarray*}
	\tau_{i,j,k} & := & \sum_{\{(a,b)\in \lbrack i-1,j-1\rbrack \times \lbrack j-2,k-3\rbrack,\ a\le b\}} x_{1,a}x_{2,b}
\end{eqnarray*}
The planar face ratios $u_{ijk}$ are as follows.  For any $\{i,j,k\} \subset \{1,\ldots, n\}$ that is not one of the $n$ cyclic intervals $\{j,j+1,j+2\}$, define
\begin{eqnarray*}
	u_{ijk}(x) = \begin{cases}
		\frac{\tau_{i,n-1,n}}{\tau_{i-1,n-1,n}}, & (i,j,k) = (i,n-1,n)\\
		\frac{\tau_{i+1,j,k}\tau_{i,j+1,j+2}}{\tau_{i,j,k}\tau_{i+1,j+1,j+2}}, & j+1<k,\ \ k=n,\\
		\frac{\tau_{i+1,j,k}\tau_{i,j,k+1}}{\tau_{i,j,k}\tau_{i+1,j,k+1}}, & k <n.
	\end{cases}
\end{eqnarray*}
\begin{conjecture}\label{conjecture: worldsheet compactification}
	The statement has two components.
	\begin{enumerate}
		\item The planar face ratios $u_{ijk}(\tau)$ solve Equations \eqref{eq: binary relations worldsheet associahedron defn}.  More precisely, Equations \eqref{eq: binary relations worldsheet associahedron defn} generate the ideal of relations satisfied by $u_{ijk}(x)$.
		\item The strata in the generalized worldsheet associahedron $\mathcal{W}^+_{3,n}$ are in bijection with pairwise noncrossing collections $\mathcal{J} \in \mathbf{NC}_{3,n}$.  
	\end{enumerate}

\end{conjecture}
We have checked Conjecture \ref{conjecture: worldsheet compactification} algebraically for all $(3,n)$ with $n\le 15$.

Naturally one would hope that Conjecture \ref{conjecture: worldsheet compactification} will extend to $G(k,n)\slash (\mathbb{C}^\ast)^n$ for higher $k\ge 4$; but several obstacles remain to be conquered before that can be achieved.  For example, for each $k=3,4,\ldots, $ a direct line of attack would require an analog of the exponent rule given in Equation \eqref{eq: binary equations exponents intro}.  It does appear that this may be possible using a certain compatibility degree.  Indeed, we give the general construction below.

First let us introduce homogeneous polynomials $\tau_{i,j,k,\ell}$.

Put $\tau_{1,2,3,j} = 1$.  For $4\le j<k \le n$, define
\begin{eqnarray*}
	\tau_{1,2,j,k} & = & \sum_{a \in \lbrack j-1,k-2\rbrack} x_{3,a}.
\end{eqnarray*}
For $3 \le i\le j<k \le n$, set
\begin{eqnarray*}
	\tau_{1,i,j,k} &  = & \sum_{\{(a,b)\in \lbrack i-1,j-1\rbrack \times \lbrack j-2,k-3\rbrack,\ a\le b\}} x_{2,a}x_{3,b},
\end{eqnarray*}
and otherwise for $2\le i\le j<k<\ell$, put 
\begin{eqnarray*}
	\tau_{i,j,k,\ell} & = & \sum_{\left\{(a,b,c) \in \lbrack i-1,j-1\rbrack \times \lbrack j-2,k-2\rbrack \times \lbrack k-3,\ell-4\rbrack:\ a\le b\le c\right\}} x_{1,a}x_{2,b}x_{3,c}.
\end{eqnarray*}

Let us define, as above, face ratios $u_J$ on $(\mathbb{CP}^{n-4-1})^{\times (4-1)}$.  For each nonfrozen $\{i,j,k,\ell\} \in \binom{\lbrack n\rbrack}{4}^{nf}$, let
\begin{eqnarray}\label{eqn: u variables 4n}
	u_{i,j,k,\ell} = \begin{cases}
		\frac{\tau_{i+1,n-2,n-1,n}}{\tau_{i,n-2,n-1,n}}, & (i,j,k,\ell) = (i,n-2,n-1,n),\ \ i \le n-4\\
		\frac{\tau_{i+1,j,n-1,n}\tau_{i,j+1,j+2,j+3}}{\tau_{i,j,n-1,n}\tau_{i+1,j+1,j+2,j+3}}, & (i,j,k,\ell) = (i,j,n-1,n),\ \ j \le n-3\\
		\frac{\tau_{i+1,j,k,n}\tau_{i,j,k+1,k+2}}{\tau_{i,j,k,n}\tau_{i+1,j,k+1,k+2}}, & (i,j,k,\ell) = (i,j,k,n),\ \ k\le n-2\\
		\frac{\tau_{i+1,j,k,\ell}\tau_{i,j,k,\ell+1}}{\tau_{i,j,k,\ell}\tau_{i+1,j,k,\ell+1}}, & k <n.
	\end{cases}
\end{eqnarray}
	From Equation \eqref{eqn: u variables 4n} it is relatively straightforward (if somewhat tedious) to extrapolate an all (k,n) formula for the u-variables.  We omit the general expression.

Let us formulate the all-(k,n) definition of the generalized worldsheet associahedron $\mathcal{W}^+_{k,n}$.

For any $I,J\in \binom{n}{k}^{nf}$, the \textit{compatibility degree}\footnote{C.f. \cite{FZ2003Y-systems}.}$c_{I,J}$ is the number of violations of the noncrossing\footnote{We thank Hugh Thomas for sharing insights about the non-kissing complex, which led us to formulate the all-(k,n) compatibility degree for the noncrossing complex.} condition in the pair $(I,J)$.  Let us be more precise, by paraphrasing Definition \ref{defn: noncrossing}.
\begin{defn}\label{defn: compatibility degree}
	Given $I,J\in \binom{n}{k}^{nf}$, then the compatibility degree $c_{I,J}\ge 0$ is the number of distinct pairs $(\{i_a,i_b\},\{j_c,j_d\})$ with $\{i_a,i_b\} \subseteq I$ and $\{j_c,j_d\} \subseteq J$, 
	such that $\{(i_a,i_b),(j_c,j_d)\}$ is crossing: the pair is not weakly separated and we have that $i_\ell = j_\ell$ for each\footnote{Note that this condition holds trivially whenever $b=a+1$ and $d=c+1$.} integer $\ell = a+1,a+2,\ldots, b-1$.
\end{defn}


\begin{defn}\label{defn: generalized worldsheet conclusion}
	The generalized worldsheet associahedron $\mathcal{W}^+_{k,n}$ is the set of all points $(u_J) \in \lbrack 0,1\rbrack^{\binom{n}{k}-n}$ such that the following set of equations hold: for each $J\in \binom{\lbrack n\rbrack}{k}^{nf}$, one equation
	\begin{eqnarray}\label{eq: binary relations worldsheet associahedron defn All kn}
		u_J & =& 1 - \prod_{\{I:\ (I,J) \not\in\mathbf{NC}_{k,n}\}}u^{c_{I,J}}_I,
	\end{eqnarray}
where $c_{I,J}$ is the compatibility degree from Definition \ref{defn: compatibility degree}.
\end{defn}

\begin{conjecture}\label{conjecture: binary relations All kn}
	Equations \eqref{eq: binary relations worldsheet associahedron defn All kn} generate the ideal of relations among the planar face ratios $u_J(\tau)$.
\end{conjecture}
We have checked nontrivial instances of Conjecture \ref{conjecture: binary relations All kn}, including $(k,n) = (5,14)$.  Let $J = \{2,4,6,8,12\}$.  Then one can check that $c_{J,I_i} >0$ for exactly $1293$ subsets $I_i \in \binom{\lbrack 14\rbrack}{5}$. 

In this case, we found that after a perfect cancellation of 1293 planar face ratios $u_{J}$ one has indeed 
$$u_{2,4,6,8,12} = 1 - \left(\prod_{j=1}^{1293} u_{J_j}^{c_{\{2,4,6,8,12\},J_j}}\right),$$
having substituted in (the $k=5$ analog of) Equation \eqref{eqn: u variables 4n}.

Let us illustrate with a more manageable (but still highly nontrivial) example.

\begin{example}
	One can check that for $(k,n) = (4,8)$ we have
	$$u_{2468} = 1 - \prod_{u_J^{c_{\{2,4,6,8\},J}}\in \mathcal{U}_{2468}}u_J^{c_{\{2,4,6,8\},J}},$$
	where 
	$$\mathcal{U}_{2468} = \left(
	\begin{array}{ccccccc}
		u_{1,2,3,7} & u_{1,2,4,7} & u_{1,2,5,6} & u_{1,2,5,7}^2 & u_{1,2,5,8} & u_{1,2,6,7} & u_{1,3,4,5} \\
		u_{1,3,4,6} & u_{1,3,4,7}^2 & u_{1,3,4,8} & u_{1,3,5,6}^2 & u_{1,3,5,7}^3 & u_{1,3,5,8}^2 & u_{1,3,6,7}^2 \\
		u_{1,3,6,8} & u_{1,3,7,8} & u_{1,4,5,6} & u_{1,4,5,7}^2 & u_{1,4,5,8} & u_{1,4,6,7} & u_{1,5,7,8} \\
		u_{2,3,4,7} & u_{2,3,5,6} & u_{2,3,5,7}^2 & u_{2,3,5,8} & u_{2,3,6,7} & u_{2,4,5,7} & u_{2,5,7,8} \\
		u_{3,4,5,7} & u_{3,4,7,8} & u_{3,5,6,7} & u_{3,5,6,8} & u_{3,5,7,8}^2 & u_{3,6,7,8} & u_{4,5,7,8} \\
	\end{array}
	\right).$$
	Here we are substituting in Equation \eqref{eqn: u variables 4n}.
\end{example}

	\section{Acknowledgements}

	This work has benefited from discussions with many people at various stages.  In particular, we thank Freddy Cachazo and Nima Arkani-Hamed for encouragement, helpful discussions and comments on a draft.  We are grateful to Benjamin Schroeter and Simon Telen for correspondence and assistance with numerical computations using the computer packages PolyMake and the Julia package, HomotopyContinuation.jl.  We also thank Chris Fraser, Thomas Lam, Sebastian Mizera, William Norledge, Bernd Sturmfels, Jenia Tevelev and Hugh Thomas for stimulating discussions at various stages of the work.
	
	We thank the Institute for Advanced study for excellent working conditions while this work was completed.
	
	This research was supported in part by a grant from the Gluskin Sheff/Onex Freeman Dyson Chair in Theoretical Physics and by Perimeter Institute. Research at Perimeter Institute is supported in part by the Government of Canada through the Department of Innovation, Science and Economic Development Canada and by the Province of Ontario through the Ministry of Colleges and Universities.
	
	\appendix
	
	\section{Blades}\label{sec: blades}
	In this Appendix, we review definitions and basic results from previous papers \cite{Early19WeakSeparationMatroidSubdivision,Early2019PlanarBasis,Early2020WeightedBladeArrangements}.
	
	If $J\in\binom{\lbrack n\rbrack}{k}$ is a $k$-element subset of $\{1,\ldots, n\}$, define $e_J = \sum_{j\in J}e_j$.

	The original definition of blades is due to A. Ocneanu; blades were first studied in \cite{EarlyBlades}, see also \cite{Early19WeakSeparationMatroidSubdivision,Early2019PlanarBasis,Early2020WeightedBladeArrangements,CE2020B}.
	
	\begin{defn}[\cite{OcneanuVideo}]\label{defn:blade}
		A decorated ordered set partition $((S_1)_{s_1},\ldots, (S_\ell)_{s_\ell})$ of $(\{1,\ldots, n\},k)$ is an ordered set partition $(S_1,\ldots, S_\ell)$ of $\{1,\ldots, n\}$ together with an ordered list of integers $(s_1,\ldots, s_\ell)$ with $\sum_{j=1}^\ell s_j=k$.  It is said to be of type $\Delta_{k,n}$ if we have additionally $1\le s_j\le\vert S_j\vert-1 $, for each $j=1,\ldots, \ell$.  In this case we write $((S_1)_{s_1},\ldots, (S_\ell)_{s_\ell}) \in \text{OSP}(\Delta_{k,n})$, and we denote by $\lbrack (S_1)_{s_1},\ldots, (S_\ell)_{s_\ell}\rbrack$ the convex polyhedral cone in the affine hyperplane in $\mathbb{R}^n$ where $\sum_{j=1}^n x_j=k$, that is cut out by the facet inequalities
		\begin{eqnarray}\label{eq:hypersimplexPlate}
			x_{S_1} & \ge & s_1 \nonumber\\
			x_{S_1\cup S_2} & \ge & s_1+s_2\nonumber\\
			& \vdots & \\
			x_{S_1\cup\cdots \cup S_{\ell-1}} & \ge & s_1+\cdots +s_{\ell-1}.\nonumber
		\end{eqnarray}
		These cones were called \textit{plates} by Ocneanu.
		Finally, the \textit{blade} $(((S_1)_{s_1},\ldots, (S_\ell)_{s_\ell}))$ is the union of the codimension one faces of the complete simplicial fan formed by the $\ell$ cyclic block rotations of $\lbrack (S_1)_{s_1},\ldots(S_\ell)_{s_\ell},\rbrack$, that is
		\begin{eqnarray}\label{eq: defn blade}
			(((S_1)_{s_1},\ldots, (S_\ell)_{s_\ell})) = \bigcup_{j=1}^\ell \partial\left(\lbrack (S_j)_{s_j},(S_{j+1})_{s_{j+1}},\ldots, (S_{j-1})_{s_{j-1}}\rbrack\right).
		\end{eqnarray}
	\end{defn}
	Let $\beta = ((1,2,\ldots, n))$ be the standard blade; as noted in \cite{Early19WeakSeparationMatroidSubdivision}, this is isomorphic to a tropical hyperplane.
	
	\begin{figure}[h!]
		\centering
		\includegraphics[width=0.7\linewidth]{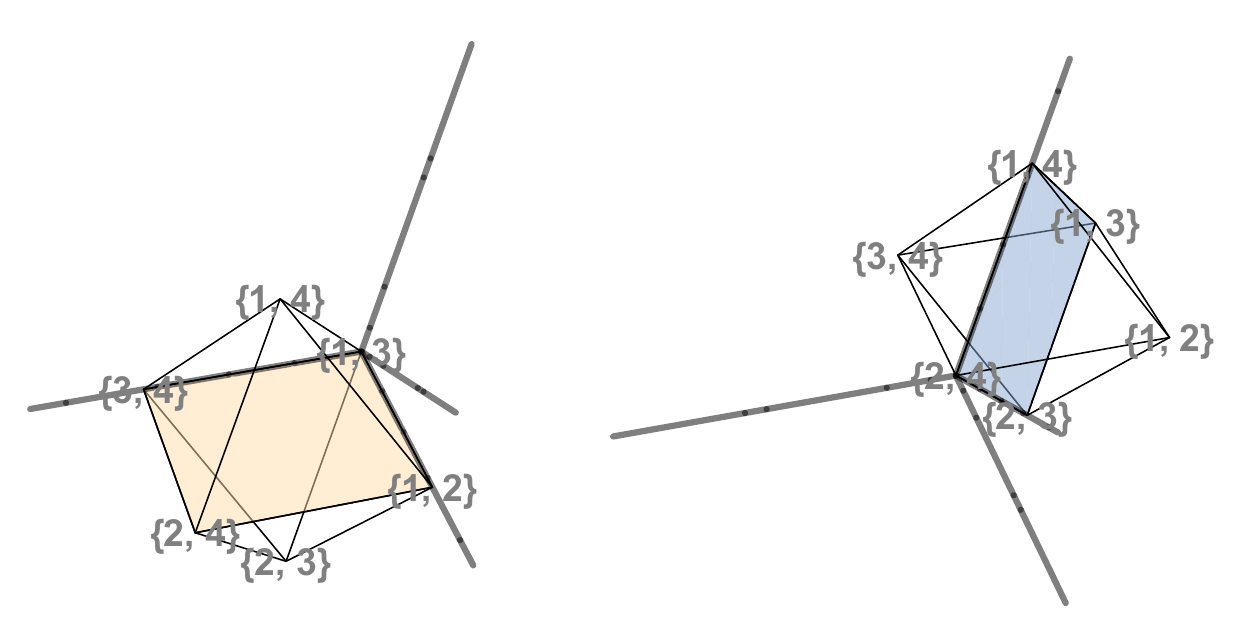}
		\caption{Arranging the blade $((1,2,3,4))$ on the octahedron induces the $t = \eta_{1,3} = s_{23}$ and $s = \eta_{2,4} = s_{12}$ channels (left and right, respectively)}
		\label{fig:12-7-2019bladearrangementoctahedron}
	\end{figure}
	\begin{figure}[h!]
		\centering
		\includegraphics[width=0.8\linewidth]{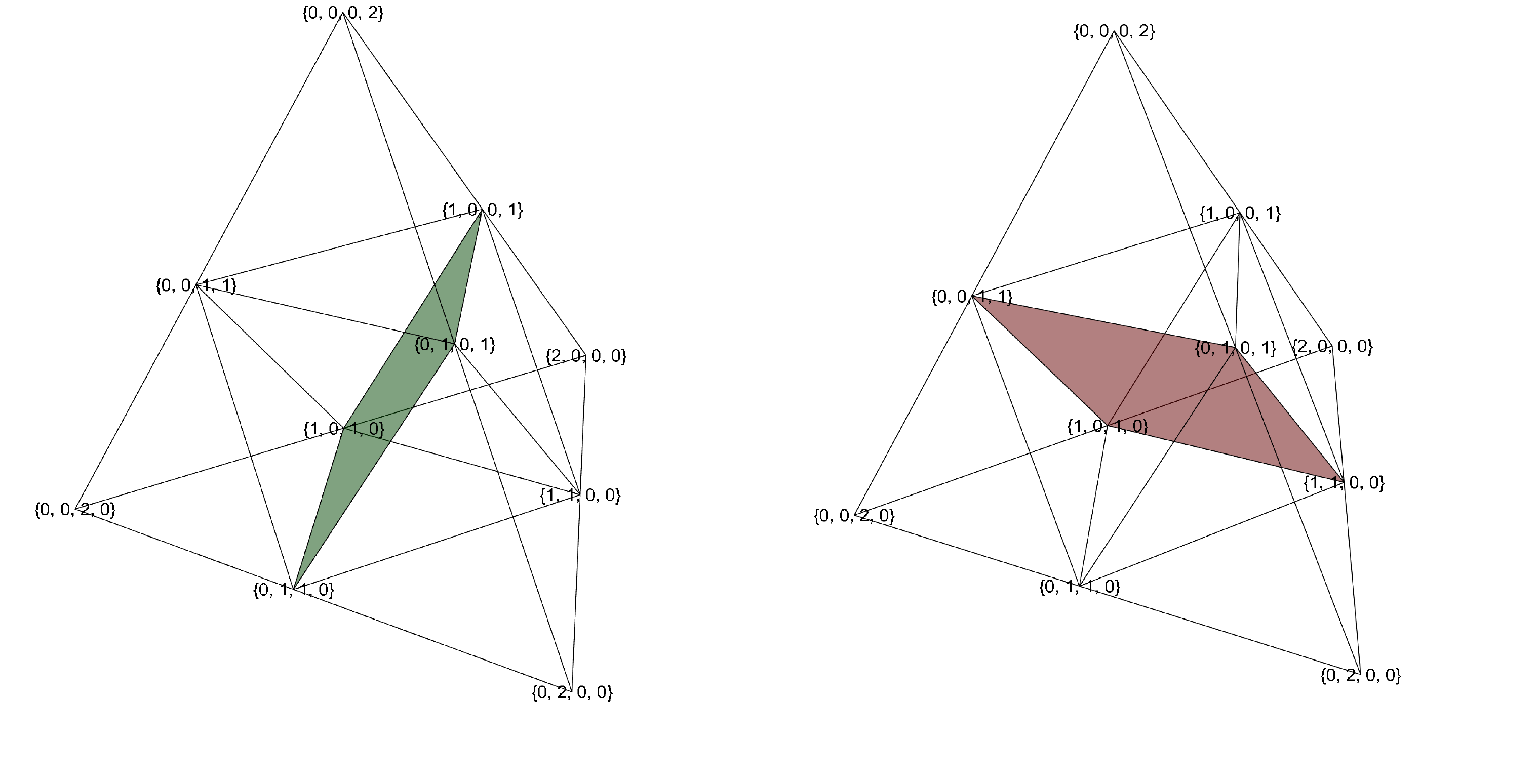}
		\caption{Second view: the two blade arrangements on the octahedron, embedded into a tetrahedron.}
		\label{fig:blades-tetrahedron-square-move}
	\end{figure}

	Any point $v\in \mathbb{R}^n$ gives rise to a translation $\beta_v$ of $\beta$ by the vector $v$.  When $v = e_J$ is a vertex of a hypersimplex $\Delta_{k,n}$, then we write simply $\beta_J$.
	In this paper we consider only translations of the single nondegenerate blade with labeled by the cyclic order $(1,2,\ldots, n)$, usually denoted $\beta := ((1,2,\ldots, n))$; however in \cite{Early19WeakSeparationMatroidSubdivision} it was shown that by pinning $\beta$ to a vertex $e_J$ of a hypersimplex $\Delta_{k,n}$, then that translated blade $\beta_{J}$ intersects the hypersimplex in a blade $(((S_1)_{s_1},\ldots, (S_\ell)_{s_\ell}))$ where now the pairs $(S_j,s_j)$ are uniquely determined and satisfy the condition from Definition \ref{defn:blade},
	$$1\le s_j\le \vert S_j\vert-1,$$
	or in short $((S_1)_{s_1},\ldots, (S_\ell)_{s_\ell}) \in \text{OSP}(\Delta_{k,n})$.  Additionally, we have that each $S_j = \{a,a+1,\ldots, b\}$ is cyclically contiguous.  We refer the reader to \cite{Early19WeakSeparationMatroidSubdivision} for a detailed explanation of the construction of the decorated ordered set partition.
	
	The number of blocks $\ell$ is equal to the number of cyclic intervals in the set $J$ and the contents of the blocks are determined by the set $J$ together with the cyclic order.  In particular, the number of blocks is equal to the number of maximal cells in the subdivision induced by the blade.
	
	It was further shown in \cite{Early19WeakSeparationMatroidSubdivision} that $(((S_1)_{s_1},\ldots, (S_\ell)_{s_\ell}))$ with $((S_1)_{s_1},\ldots, (S_\ell)_{s_\ell})\in \text{OSP}(\Delta_{k,n})$ induces a certain multi-split positroidal subdivision of $\Delta_{k,n}$ where the vertices of the maximal cells become bases of Schubert matroids, or nested matroids (for a recent review of multi-splits, we recommend \cite{SchroeterThesis} and references therein.).
	
	As in \cite{EarlyBlades}, blades can be studied using their indicator functions, in which case one starts to encounter certain linear relations that are encountered among rays defining the bipyramidal cone in the tropical Grassmannian $\text{Trop}G(3,6)$.  The basic linear relation among indicator functions of blades is exhibited in Figure \ref{fig:bladerelationdim3}.
	
	\begin{figure}[h!]
		\centering
		\includegraphics[width=1\linewidth]{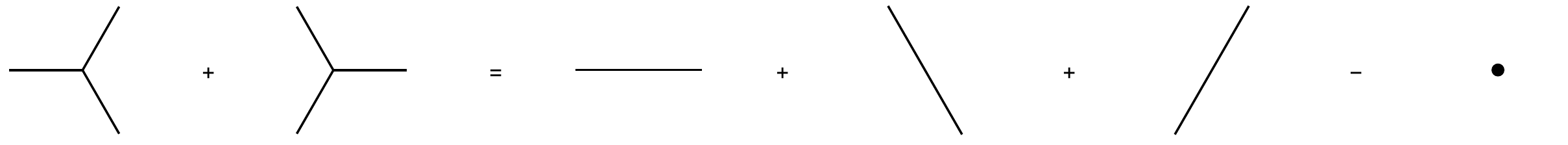}
		\caption{The indicator function relation for blades in $n=3$: $1_{((1,2,3))} + 1_{((1,3,2))} = 1_{((1,23))} + 1_{((13,2))} + 1_{((12,3))} - 1_{((1)(2)(3))}$.  Compare to the relation among primitive generators of the bipyramidal cone in $\text{Trop}G(3,6)$; see also \cite[Equation 2.29]{CEGM2019} for the identity among poles of $m^{(3)}(\mathbb{I}_6,\mathbb{I}_6)$.  In that case, only five terms appear: we are missing the ``dot.''}
		\label{fig:bladerelationdim3}
	\end{figure}

	We set $\beta^{(L)}_J=0$ when no subdivision is induced by $\beta_J$ on the face $\partial_L(\Delta_{k,n})$.  This is the case exactly when $J$ is frozen with respect to the gapped cyclic order on $\{1,\ldots, n\}\setminus L$ inherited from the standard cyclic order $(1,2,\ldots, n)$.  
	
	There is a natural action induced by restriction: define linear operators $\partial_j$ on the linear span of the symbols $\beta^{(L)}_J$ as follows.  
	
	\begin{itemize}
		\item If $j\in L$, then we set $\partial_j(\beta^{(L)}_J) = 0$.  
		\item If $j\not\in L$, set 
		$$\partial_j(\beta^{(L)}_J) = \beta^{(L\cup \{j\})}_{J\setminus\{\ell\}},$$
		where $\ell=j$ if $j\in J$, and otherwise $\ell$ is the cyclically next element of $\{1,\ldots, n\}$ that is in $J$.
	\end{itemize}
	
	Define 
	$$\partial = \partial_1+\cdots + \partial_n.$$
	
	Then we have the operator-theoretic identities for powers,
	$$\partial_j^2=0,\ \text{ and } \frac{1}{d!}\partial^{d} = \sum_{L\in\binom{\lbrack n\rbrack}{d}}\partial_{L},$$
	where we have defined $\partial_L = \partial_{\ell_1}\cdots \partial_{\ell_d}$ with $\ell_1<\cdots<\ell_d$, when $L=\{\ell_1,\ldots, \ell_d\}$.
	
	\begin{defn}\label{defn:Xkn YKn ZKn}

		Denote by 
		$$\mathcal{X}(k,n) = \left\{\beta_{\mathbf{c}} \in \mathfrak{B}_{k,n}: \text{supp }\partial_L(\beta_{\mathbf{c}})\text{ is a matroidal blade arrangement on } \partial_L(\Delta_{k,n}),L\in\binom{\lbrack n\rbrack}{k-2}\right\}$$
		the arrangement of (real) subspaces in $\mathfrak{B}_{k,n}$ consisting of linear combinations of blades $\beta_J$ whose support induces a (positroidal) subdivision on each second hypersimplicial face, and by 
		$$\mathcal{Y}(k,n)= \left\{\beta_{\mathbf{c}}  \in \mathfrak{B}_{k,n}: \text{coeff}\left(\partial_L(\beta_{\mathbf{c}}),\beta^L_{ij}\right)\ge0,L\in\binom{\lbrack n\rbrack}{k-2},\ \{i,j\}\in\binom{\lbrack n\rbrack\setminus L}{2}\text{ nonfrozen}\right\},$$
		the convex cone in $\mathfrak{B}_{k,n}$ with nonnegative curvature on every second hypersimplicial face $\partial_L(\Delta_{k,n})$.
		
		Let $\mathcal{Z}(k,n)$ denote their intersection:
		$$\mathcal{Z}(k,n)= \mathcal{X}(k,n) \cap \mathcal{Y}(k,n).$$
	\end{defn}

	\begin{thm}[\cite{Early2020WeightedBladeArrangements}]\label{Early2020WeightedBladeArrangements}
		There is a natural embedding $\text{Trop}^+ G(k,n) \hookrightarrow \mathcal{B}(k,n)$; its image is the set of matroidal weighted blade arrangements $\mathcal{Z}(k,n)$. 
	\end{thm}
	
	\section{Numerical Amplitude Calculation: $\hat{m}^{(3),NC}(\mathbb{I}_6,\mathbb{I}_6)$}\label{sec: numerical amplitude evaluation}
	
	In what follows, we use the parameterization of $X(3,6)$ obtained from 
	\begin{eqnarray*}
		\begin{bmatrix}
			1 & 0 & 0 & x_{1,1} x_{2,1}& x_{1,1}(x_{2,1}+x_{2,2}) + x_{1,2} x_{2,2} &x_{1,1}(x_{2,1}+x_{2,2}+x_{2,3}) + x_{1,2}(x_{2,2}+x_{2,3})  + x_{1,3}x_{2,3} \\
			0 & 1 & 0 &  x_{2,1} 	&x_{2,1}+x_{2,2}	& x_{2,1}+x_{2,2}+x_{2,3} \\
			0 & 0 & 1 & 1 & 1 & 1
		\end{bmatrix}
	\end{eqnarray*}
	by restricting the variables to $x_{1,1} = x_{2,1} = 1$.  Note that often one adds signs to some of the rows for sake of positivity, but this is not necessary in our context due to torus invariance of the potential function.
	
	Let $p_{j_1j_2j_3}$ be the determinant of the $3\times 3$ submatrix formed from the columns with indices $j_1,j_2,j_3$, and let $\hat{p}_{j_1j_2j_3}$ as in Example \ref{example: Plucker resolution explicit 36}.  Specifically, put all $\hat{p}_{j_1j_2j_3} = p_{j_1j_2j_3}$ except for 
		$$\hat{p}_{236} = p_{236} - \frac{p_{123}p_{456}}{p_{145}}\ \ \text{ and }\ \  \hat{p}_{356} = p_{356} - \frac{p_{123}p_{456}}{p_{124}}.$$
	
	Now define a potential function
	$$\hat{\mathcal{S}}_{3,6}(s) = \sum_{J}\log(\hat{p}_J)s_J,$$
	which, after straightening to the planar basis (see \cite{Early2019PlanarBasis,Early2020WeightedBladeArrangements}), becomes
	\begin{eqnarray*}
		\hat{\mathcal{S}}_{3,6}(s)& = & \sum_{J\in \binom{\lbrack 6 \rbrack}{3}^{nf}}\eta_J\log(\hat{w}_J)\\
		& = & \eta _{124} \log \left(\frac{\hat{p}_{125} \hat{p}_{134}}{\hat{p}_{124} \hat{p}_{135}}\right)+\eta _{125} \log \left(\frac{\hat{p}_{126} \hat{p}_{135}}{\hat{p}_{125} \hat{p}_{136}}\right)+\eta _{134} \log \left(\frac{\hat{p}_{135} \hat{p}_{234}}{\hat{p}_{134} \hat{p}_{235}}\right)+ \eta _{146} \log \left(\frac{\hat{p}_{156} \hat{p}_{246}}{\hat{p}_{146} \hat{p}_{256}}\right)\nonumber\\
		& + &\eta _{136} \log \left(\frac{\hat{p}_{146} \hat{p}_{236}}{\hat{p}_{136} \hat{p}_{246}}\right)+\eta _{145} \log \left(\frac{\hat{p}_{146} \hat{p}_{245}}{\hat{p}_{145} \hat{p}_{246}}\right)+\eta _{235} \log \left(\frac{\hat{p}_{236} \hat{p}_{245}}{\hat{p}_{235} \hat{p}_{246}}\right)+\eta _{236} \log \left(\frac{\hat{p}_{123} \hat{p}_{246}}{\hat{p}_{124} \hat{p}_{236}}\right)\nonumber\\
		& + &\eta _{245} \log \left(\frac{\hat{p}_{246} \hat{p}_{345}}{\hat{p}_{245} \hat{p}_{346}}\right)+\eta _{256} \log \left(\frac{\hat{p}_{125} \hat{p}_{356}}{\hat{p}_{135} \hat{p}_{256}}\right)+\eta _{346} \log \left(\frac{\hat{p}_{134} \hat{p}_{356}}{\hat{p}_{135} \hat{p}_{346}}\right)+\eta _{356} \log \left(\frac{\hat{p}_{135} \hat{p}_{456}}{\hat{p}_{145} \hat{p}_{356}}\right)\\
		& + & \eta _{135} \log \left(\frac{\hat{p}_{136} \hat{p}_{145} \hat{p}_{235} \hat{p}_{246}}{\hat{p}_{135} \hat{p}_{146} \hat{p}_{236} \hat{p}_{245}}\right)+ \eta _{246} \log \left(\frac{\hat{p}_{124} \hat{p}_{135} \hat{p}_{256} \hat{p}_{346}}{\hat{p}_{125} \hat{p}_{134} \hat{p}_{246} \hat{p}_{356}}\right)\nonumber\\
		& = & \hat{\eta} _{124} \log \left(\frac{\hat{p}_{125} \hat{p}_{134}}{\hat{p}_{124} \hat{p}_{135}}\right)+\hat{\eta} _{125} \log \left(\frac{\hat{p}_{126} \hat{p}_{135}}{\hat{p}_{125} \hat{p}_{136}}\right)+\hat{\eta} _{134} \log \left(\frac{\hat{p}_{135} \hat{p}_{234}}{\hat{p}_{134} \hat{p}_{235}}\right)+\hat{\eta} _{135} \log \left(\frac{\hat{p}_{136} \hat{p}_{235}}{\hat{p}_{135} \hat{p}_{236}}\right)\\
		& + & \hat{\eta} _{136} \log \left(\frac{\hat{p}_{145} \hat{p}_{236}}{\hat{p}_{136} \hat{p}_{245}}\right)+\hat{\eta} _{145} \log \left(\frac{\hat{p}_{146} \hat{p}_{245}}{\hat{p}_{145} \hat{p}_{246}}\right)+\hat{\eta} _{146} \log \left(\frac{\hat{p}_{156} \hat{p}_{246}}{\hat{p}_{146} \hat{p}_{256}}\right)+\hat{\eta} _{235} \log \left(\frac{\hat{p}_{145} \hat{p}_{236}}{\hat{p}_{146} \hat{p}_{235}}\right)\\
		& + & \hat{\eta} _{346} \log \left(\frac{\hat{p}_{124} \hat{p}_{356}}{\hat{p}_{125} \hat{p}_{346}}\right)+\hat{\eta} _{245} \log \left(\frac{\hat{p}_{246} \hat{p}_{345}}{\hat{p}_{245} \hat{p}_{346}}\right)+\hat{\eta} _{246} \log \left(\frac{\hat{p}_{256} \hat{p}_{346}}{\hat{p}_{246} \hat{p}_{356}}\right)+\hat{\eta} _{256} \log \left(\frac{\hat{p}_{124} \hat{p}_{356}}{\hat{p}_{134} \hat{p}_{256}}\right)\\
		& + & \hat{\eta} _{236} \log \left(\frac{\hat{p}_{123} \hat{p}_{146} \hat{p}_{245}}{\hat{p}_{124} \hat{p}_{145} \hat{p}_{236}}\right)+\hat{\eta} _{356} \log \left(\frac{\hat{p}_{125} \hat{p}_{134} \hat{p}_{456}}{\hat{p}_{124} \hat{p}_{145} \hat{p}_{356}}\right),
	\end{eqnarray*}
	where in the second equality we have used the following linear change of variable.  Define $\hat{\eta}_{i_1i_2i_3} = \eta_{i_1i_2i_3}$ unless $\{i_1,i_2,i_3\} \in \{\{1,2,4\},\{1,4,5\}\}$, in which case we set
	\begin{eqnarray*}
		\hat{\eta}_{124} + \eta_{356} & = & -\eta_{246} + \eta_{124} + \eta_{346} + \eta_{256}\\
		\hat{\eta}_{145} + \eta_{236} & = & -\eta_{135} + \eta_{235} + \eta_{145} + \eta_{136},
	\end{eqnarray*}
	The claim, which we check numerically, is that the $\hat{\eta}_J$ are now the poles of the resolved amplitude $\hat{m}^{(3),NC}(\mathbb{I}_6,\mathbb{I}_6)$.

	Now, compatibly with Corollary \ref{cor: nontrivially resolved Pluckers}, $\{124,356\}$ and $\{145,236\}$ are the two noncrossing pairs in $\mathbf{NC}_{3,6}$ that are not weakly separated.
	
	In what follows, we show, for a choice of kinematics such that the planar kinematic invariants $\eta_J$ evaluate to given large prime numbers determining a point $(s_0) \in \mathcal{K}_D(3,6)$, that the noncrossing scattering equations formulation of the noncrossing amplitude agrees with the sum over noncrossing Generalized Feynman Diagrams modulo a shift in the kinematics:
	\begin{eqnarray}\label{CEGM biadjoint scalar defin critical point}
		\sum_{(x)\in \text{critical points}(\hat{\mathcal{S}}_{3,6}(s))}\frac{1}{\det'\Phi(x)} \prod_{j=1}^n\left(\frac{1}{\hat{p}_{j,j+1,j+2}(x)}\right)^2= \sum_{\{J_1,J_2,J_3,J_4\}\in \mathbf{NC}_{3,6}}\prod_{j=1}^4\frac{1}{\hat{\eta}_{J_j}}.
	\end{eqnarray}
	Here $\det'\Phi$ is the so-called reduced determinant of the Hessian $\Phi$; see \cite{CEGM2019} for details.

	Let us fix the kinematic point $(s)$ determined by $\eta_{i,i+1,i+2} = 0$, together with fixed large prime integer values for the remaining planar kinematic invariants: 
	$$
	\begin{array}{cccc}
		\eta _{124} & 8087 & \eta _{235} & 17599 \\
		\eta _{125} & 8537 & \eta _{236} & 20333 \\
		\eta _{134} & 9227 & \eta _{245} & 23321 \\
		\eta _{135} & 10247 & \eta _{246} & 26737 \\
		\eta _{136} & 11657 & \eta _{256} & 30637 \\
		\eta _{145} & 13259 & \eta _{346} & 34679 \\
		\eta _{146} & 15277 & \eta _{356} & 39293. \\
	\end{array}
	$$
	
	Then, summing over the 35 critical points using the lefthand side of Equation \eqref{CEGM biadjoint scalar defin critical point}, we obtain
	$$\frac{123056338102581409136850198886105885604358154}{117823347678612917535483161041113226062939619903306798191335}.$$
	The denominator has the large prime factors
	$$\{8537,9227,10247,11657,15277,17599,20333,23321,26737,30637,34679,39293\}$$
	and an additional factor 
	$$87996755.$$
	Consequently we have found that all but two of the original 14 planar kinematic invariants $\eta_J$ are obviously present; however $\eta_{124}$ and $\eta_{145}$ are not present, a priori.  At the very least, they have been shifted; we claim that their new values are
	$$\hat{\eta}_{124} = 7373,\ \ \hat{\eta}_{145} = 11935.$$
	Notice that the shifts are exactly 
	$$\hat{\eta}_{124} = \eta_{124} + s_{356} $$
	and
	$$\hat{\eta}_{145} = \eta_{145} + s_{236},$$
	where
	$$-s_{356} = (\eta_{356} - \eta_{256} - \eta_{346} + \eta_{246})=714 $$
	and
	$$-s_{236} = (\eta_{236} - \eta_{235} - \eta_{136} + \eta_{135}) = 1324.$$

	We give an informal derivation, first evaluating\footnote{However, we urge caution as in general there is no reason a priori to expect that the resolved amplitude $m^{(3),NC}(\alpha,\beta)$ would be particularly well-behaved for arbitrary pairs of cyclic orders $\alpha,\beta$.} (at the same kinematic point) 
	$$\hat{m}^{(3),NC}((123456),(126435))=\frac{1}{44972074438939577} = \frac{1}{7373 \cdot 8537\cdot  23321\cdot  30637}.$$
	But in \cite{CEGM2019} it was shown (translating Equation 2.19) that the usual generalized biadjoint scalar amplitude satisfies
	$$m^{(3)}((123456),(126435)) = \frac{1}{\eta_{125}\eta_{245}\eta_{256}\eta_{124}}.$$
	
	If we allow ourselves to assume that $\hat{\eta}_{a,b,c} =\eta_{a,b,c}$ except possibly when $\{a,b,c\}\in \{\{1,2,4\},\{1,4,5\}\}$ and that, correspondingly, we have
	$$\hat{m}^{(3),NC}((123456),(126435)) = \frac{1}{\hat{\eta}_{125}\hat{\eta}_{245}\hat{\eta}_{256}\hat{\eta}_{124}}= \frac{1}{\eta_{125}\eta_{245}\eta_{256}\hat{\eta}_{124}},$$
	we find that $\hat{\eta}_{124} = 7373$.  Comparing with the first we derive that $\hat{\eta}_{145} = 11935$, as claimed.

	Let us now compute the sum over all maximal noncrossing collections in $\mathbf{NC}_{3,6}$; taking into account the kinematic shift, then we evaluate the poles of the resolve amplitude take the values
	$$
	\begin{array}{cccc}
		\hat{\eta} _{1,2,4} & 7373 & \hat{\eta} _{2,3,5} & 17599 \\
		\hat{\eta} _{1,2,5} & 8537 & \hat{\eta} _{2,3,6} & 20333 \\
		\hat{\eta} _{1,3,4} & 9227 & \hat{\eta} _{2,4,5} & 23321 \\
		\hat{\eta} _{1,3,5} & 10247 & \hat{\eta} _{2,4,6} & 26737 \\
		\hat{\eta} _{1,3,6} & 11657 & \hat{\eta} _{2,5,6} & 30637 \\
		\hat{\eta} _{1,4,5} & 11935 & \hat{\eta} _{3,4,6} & 34679 \\
		\hat{\eta} _{1,4,6} & 15277 & \hat{\eta} _{3,5,6} & 39293, \\
	\end{array}
	$$
	and immediately we obtain
	\begin{eqnarray*}
		\sum_{\{J_1,J_2,J_3,J_4\}\in \mathbf{NC}_{3,6}}\prod_{j=1}^4\frac{1}{\hat{\eta}_{J_j}} & = & \frac{1}{\hat{\eta} _{125} \hat{\eta} _{134} \hat{\eta} _{135} \hat{\eta} _{145}}+\frac{1}{\hat{\eta} _{134} \hat{\eta} _{135} \hat{\eta} _{136} \hat{\eta} _{145}}+\frac{1}{\hat{\eta} _{124} \hat{\eta} _{125} \hat{\eta} _{134} \hat{\eta} _{145}} + \cdots \\
		& + & \frac{1}{\hat{\eta} _{236} \hat{\eta} _{256} \hat{\eta} _{346} \hat{\eta} _{356}}+\frac{1}{\hat{\eta} _{124} \hat{\eta} _{256} \hat{\eta} _{346} \hat{\eta} _{356}}\\
		& = & \frac{123056338102581409136850198886105885604358154}{117823347678612917535483161041113226062939619903306798191335}
	\end{eqnarray*}
	which, indeed, coincides with the value of the non-crossing scattering equations computation at the unshifted kinematic point, affirming in hindsight the efficacy of our informal derivation.
	
	\begin{rem}
		With the help of Simon Telen, using the Julia package HomotopyContinuation.jl \cite{Brieding1,Brieding2}, we can also report the results of a calculation of the number of critical points for the potential function 
		$$\hat{\mathcal{S}}_{3,7} = \sum_{\{i,j,k\} \in \binom{\lbrack 7\rbrack}{3}}\log(\hat{p}_{ijk})s_{ijk}.$$  The output of the calculation is as follows:
		\begin{itemize}
			\item Solutions found: 3127    Time: 0:01:10
			\item tracked loops (queued):            24984 (0)
			\item solutions in current (last) loop:  0 (0)
			\item generated loops (no change):       8 (5)
		\end{itemize}
	It would be very interesting to explore further the enumeration of the number of critical points (for $\hat{\mathcal{S}}_{3,7}$, 3127 critical points) and, if possible, to give a combinatorial interpretation as is possible in the case $k=2$.
		
	\end{rem}

		\section{Resolved Cross-Ratios and Binary Relations}\label{sec: u-variables}
	We define resolved cross-ratios on $X(3,n)$ and we verify certain combinatorial relations of binary type.  These provide a generalization to $\mathbb{CP}^2$ of the three-term relations holding among cross-ratio coordinates on the partial dihedral compactification of the configuration space $M_{0,n}$ of $n$ distinct points in $\mathbb{CP}^1$, seen in the Generalized Veneziano model in \cite{KobaNielsen} and \cite{RobertsPh.D.}, and more recently in for instance \cite{Brown2006} and \cite{AHLT2019}.  
	
	We conjecture that, for any $J\in \binom{\lbrack n\rbrack}{3}^{nf}$, we have
	\begin{eqnarray}\label{eq: binary relations}
		u_J & =& 1 - \prod_{\{I:\ (I,J) \not\in\mathbf{NC}_{3,n}\}}u^{c_{I,J}}_I,
	\end{eqnarray}
	where for any crossing pair $(i_1i_2i_3,j_1j_2j_3) \not\in \mathbf{NC}_{3,n}$ we put 
	\begin{eqnarray}\label{eq: binary equations exponents}
		c_{(i_1,i_2,i_3),(j_1,j_2,j_3)} = \begin{cases}
			2 & \text{if } i_1<j_1<i_2<j_2<i_3<j_3\text{ or } j_1<i_1<j_2<i_2<j_3<i_3\\
			1 & \text{otherwise}.
		\end{cases}
	\end{eqnarray}
	We have confirmed explicitly that Equations \eqref{eq: binary relations} and \eqref{eq: binary equations exponents} hold for u-variables of type $(3,n)$ for all $n\le 15$.  

	In what follows, we give the complete solution for resolved planar kinematic invariants and u-variables for $n\le 9$.  
	
	Let us proceed first with resolved cross-ratios on $X(3,6)$.  Defining $\hat{\eta}_{124}$ and $\hat{\eta}_{145}$ by 
	\begin{eqnarray*}
		\hat{\eta}_{124} + \eta_{356} & = & -\eta_{246} + \eta_{124} + \eta_{346} + \eta_{256}\\
		\hat{\eta}_{145} + \eta_{236} & = & -\eta_{135} + \eta_{235} + \eta_{145} + \eta_{136},
	\end{eqnarray*}
	then straightening the potential function $\hat{\mathcal{S}}_{3,6}$ we find the resolved cross-ratios
	$$\begin{array}{ccc}
		u_{124} & = &  \frac{\hat{p}_{125} \hat{p}_{134}}{\hat{p}_{124} \hat{p}_{135}} \\
		u_{125} & = &  \frac{\hat{p}_{126} \hat{p}_{135}}{\hat{p}_{125} \hat{p}_{136}} \\
		u_{134} & = &  \frac{\hat{p}_{135} \hat{p}_{234}}{\hat{p}_{134} \hat{p}_{235}} \\
		u_{135} & = &  \frac{\hat{p}_{136} \hat{p}_{235}}{\hat{p}_{135} \hat{p}_{236}} \\
		u_{136} & = &  \frac{\hat{p}_{145} \hat{p}_{236}}{\hat{p}_{136} \hat{p}_{245}} \\
		u_{145} & = &  \frac{\hat{p}_{146} \hat{p}_{245}}{\hat{p}_{145} \hat{p}_{246}} \\
		u_{146} & = &  \frac{\hat{p}_{156} \hat{p}_{246}}{\hat{p}_{146} \hat{p}_{256}} \\
		u_{235} & = &  \frac{\hat{p}_{145} \hat{p}_{236}}{\hat{p}_{146} \hat{p}_{235}} \\
		u_{236} & = &  \frac{\hat{p}_{123} \hat{p}_{146} \hat{p}_{245}}{\hat{p}_{124} \hat{p}_{145} \hat{p}_{236}} \\
		u_{245} & = &  \frac{\hat{p}_{246} \hat{p}_{345}}{\hat{p}_{245} \hat{p}_{346}} \\
		u_{246} & = &  \frac{\hat{p}_{256} \hat{p}_{346}}{\hat{p}_{246} \hat{p}_{356}} \\
		u_{256} & = &  \frac{\hat{p}_{124} \hat{p}_{356}}{\hat{p}_{134} \hat{p}_{256}} \\
		u_{346} & = &  \frac{\hat{p}_{124} \hat{p}_{356}}{\hat{p}_{125} \hat{p}_{346}} \\
		u_{356} & = &  \frac{\hat{p}_{125} \hat{p}_{134} \hat{p}_{456}}{\hat{p}_{124} \hat{p}_{145} \hat{p}_{356}}. \\
	\end{array}$$
	These are easily seen to satisfy the following set of binary equations in Equation \eqref{eq: binary relations}, in the sense of \cite{AHL2019Stringy}:
	\begin{eqnarray*}
		u_{124}& = & 1 -u_{135} u_{136} u_{235} u_{236} \\
		u_{125}& = & 1-u_{136} u_{146} u_{236} u_{246} u_{346}\\
		u_{134} & = & 1-u_{235} u_{236} u_{245} u_{246} u_{256}\\
		u_{135}& = & 1-u_{124} u_{146} u_{236} u_{245} u_{256} u_{346} u_{246}^2\\
		u_{136} & = & 1 - u_{124} u_{125} u_{245} u_{246} u_{256}\\
		u_{145} & = & 1-u_{246} u_{256} u_{346} u_{356} \\
		u_{146} & = & 1-u_{125} u_{135} u_{235} u_{256} u_{356}\\
		u_{235} & = & 1-u_{124} u_{134} u_{146} u_{246} u_{346}\\
		u_{236} & = & 1-u_{124} u_{125} u_{134} u_{135}\\
		u_{245} &=&1-u_{134} u_{135} u_{136} u_{346} u_{356}\\
		u_{246}& = & 1-u_{125} u_{134} u_{136} u_{145} u_{235} u_{356} u_{135}^2 \\
		u_{256}& = & 1-u_{134} u_{135} u_{136} u_{145} u_{146} \\
		u_{346}& = & 1-u_{125} u_{135} u_{145} u_{235} u_{245} \\
		u_{356}& = & 1-u_{145} u_{146} u_{245} u_{246}. \\
	\end{eqnarray*}
	In this case all (nonzero) $c_{I,J}$ are equal to one, except for $c_{135,246} = c_{246,135}=2$.
	
	Now for $(k,n) = (3,7)$ we have
	$$
	\begin{array}{cc}
		
		\begin{array}{cc}
			u_{124} & \frac{\hat{p}_{125} \hat{p}_{134}}{\hat{p}_{124} \hat{p}_{135}} \\
		\end{array}
		& 
		\begin{array}{cc}
			u_{125} & \frac{\hat{p}_{126} \hat{p}_{135}}{\hat{p}_{125} \hat{p}_{136}} \\
		\end{array}
		\\
		
		\begin{array}{cc}
			u_{126} & \frac{\hat{p}_{127} \hat{p}_{136}}{\hat{p}_{126} \hat{p}_{137}} \\
		\end{array}
		& 
		\begin{array}{cc}
			u_{134} & \frac{\hat{p}_{135} \hat{p}_{234}}{\hat{p}_{134} \hat{p}_{235}} \\
		\end{array}
		\\
		
		\begin{array}{cc}
			u_{135} & \frac{\hat{p}_{136} \hat{p}_{235}}{\hat{p}_{135} \hat{p}_{236}} \\
		\end{array}
		& 
		\begin{array}{cc}
			u_{136} & \frac{\hat{p}_{137} \hat{p}_{236}}{\hat{p}_{136} \hat{p}_{237}} \\
		\end{array}
		\\
		
		\begin{array}{cc}
			u_{137} & \frac{\hat{p}_{145} \hat{p}_{237}}{\hat{p}_{137} \hat{p}_{245}} \\
		\end{array}
		& 
		\begin{array}{cc}
			u_{145} & \frac{\hat{p}_{146} \hat{p}_{245}}{\hat{p}_{145} \hat{p}_{246}} \\
		\end{array}
		\\
		
		\begin{array}{cc}
			u_{146} & \frac{\hat{p}_{147} \hat{p}_{246}}{\hat{p}_{146} \hat{p}_{247}} \\
		\end{array}
		& 
		\begin{array}{cc}
			u_{147} & \frac{\hat{p}_{156} \hat{p}_{247}}{\hat{p}_{147} \hat{p}_{256}} \\
		\end{array}
		\\
		
		\begin{array}{cc}
			u_{156} & \frac{\hat{p}_{157} \hat{p}_{256}}{\hat{p}_{156} \hat{p}_{257}} \\
		\end{array}
		& 
		\begin{array}{cc}
			u_{157} & \frac{\hat{p}_{167} \hat{p}_{257}}{\hat{p}_{157} \hat{p}_{267}} \\
		\end{array}
		\\
		
		\begin{array}{cc}
			u_{235} & \frac{\hat{p}_{145} \hat{p}_{236}}{\hat{p}_{146} \hat{p}_{235}} \\
		\end{array}
		& 
		\begin{array}{cc}
			u_{236} & \frac{\hat{p}_{146} \hat{p}_{237}}{\hat{p}_{147} \hat{p}_{236}} \\
		\end{array}
		\\
		
		\begin{array}{cc}
			u_{237} & \frac{\hat{p}_{123} \hat{p}_{147} \hat{p}_{245}}{\hat{p}_{124} \hat{p}_{145} \hat{p}_{237}} \\
		\end{array}
		& 
		\begin{array}{cc}
			u_{245} & \frac{\hat{p}_{246} \hat{p}_{345}}{\hat{p}_{245} \hat{p}_{346}} \\
		\end{array}
		\\
		
		\begin{array}{cc}
			u_{246} & \frac{\hat{p}_{247} \hat{p}_{346}}{\hat{p}_{246} \hat{p}_{347}} \\
		\end{array}
		& 
		\begin{array}{cc}
			u_{247} & \frac{\hat{p}_{256} \hat{p}_{347}}{\hat{p}_{247} \hat{p}_{356}} \\
		\end{array}
		\\
		
		\begin{array}{cc}
			u_{256} & \frac{\hat{p}_{257} \hat{p}_{356}}{\hat{p}_{256} \hat{p}_{357}} \\
		\end{array}
		& 
		\begin{array}{cc}
			u_{257} & \frac{\hat{p}_{267} \hat{p}_{357}}{\hat{p}_{257} \hat{p}_{367}} \\
		\end{array}
		\\
		
		\begin{array}{cc}
			u_{267} & \frac{\hat{p}_{124} \hat{p}_{367}}{\hat{p}_{134} \hat{p}_{267}} \\
		\end{array}
		& 
		\begin{array}{cc}
			u_{346} & \frac{\hat{p}_{156} \hat{p}_{347}}{\hat{p}_{157} \hat{p}_{346}} \\
		\end{array}
		\\
		
		\begin{array}{cc}
			u_{347} & \frac{\hat{p}_{124} \hat{p}_{157} \hat{p}_{356}}{\hat{p}_{125} \hat{p}_{156} \hat{p}_{347}} \\
		\end{array}
		& 
		\begin{array}{cc}
			u_{356} & \frac{\hat{p}_{357} \hat{p}_{456}}{\hat{p}_{356} \hat{p}_{457}} \\
		\end{array}
		\\
		
		\begin{array}{cc}
			u_{357} & \frac{\hat{p}_{367} \hat{p}_{457}}{\hat{p}_{357} \hat{p}_{467}} \\
		\end{array}
		& 
		\begin{array}{cc}
			u_{367} & \frac{\hat{p}_{125} \hat{p}_{134} \hat{p}_{467}}{\hat{p}_{124} \hat{p}_{145} \hat{p}_{367}} \\
		\end{array}
		\\
		
		\begin{array}{cc}
			u_{457} & \frac{\hat{p}_{125} \hat{p}_{467}}{\hat{p}_{126} \hat{p}_{457}} \\
		\end{array}
		& 
		\begin{array}{cc}
			u_{467} & \frac{\hat{p}_{126} \hat{p}_{145} \hat{p}_{567}}{\hat{p}_{125} \hat{p}_{156} \hat{p}_{467}} \\
		\end{array}
		\\
	\end{array}$$
	Next let us give an example of the kind of calculation performed in order to arrive at the binary equations, as formulated in Conjecture \ref{conjecture: binary relations All kn}.

	First let us recall again the construction of the homogeneous polynomials $\tau_{i,j,k,\ell}$.
	
	Put $\tau_{1,2,3,j} = 1$.  For $4\le j<k \le n$, define
	\begin{eqnarray*}
		\tau_{1,2,j,k} & = & \sum_{a \in \lbrack j-1,k-2\rbrack} x_{3,a}.
	\end{eqnarray*}
	For $3 \le i\le j<k \le n$, set
	\begin{eqnarray*}
		\tau_{1,i,j,k} &  = & \sum_{\{(a,b)\in \lbrack i-1,j-1\rbrack \times \lbrack j-2,k-3\rbrack,\ a\le b\}} x_{2,a}x_{3,b},
	\end{eqnarray*}
	and otherwise for $2\le i\le j<k<\ell$, put 
	\begin{eqnarray*}
		\tau_{i,j,k,\ell} & = & \sum_{\left\{(a,b,c) \in \lbrack i-1,j-1\rbrack \times \lbrack j-2,k-2\rbrack \times \lbrack k-3,\ell-4\rbrack:\ a\le b\le c\right\}} x_{1,a}x_{2,b}x_{3,c}.
	\end{eqnarray*}
	
	Now recall the planar face ratios $u_J$ for $k=4$.  For each nonfrozen $\{i,j,k,\ell\} \in \binom{\lbrack n\rbrack}{4}^{nf}$, let
	\begin{eqnarray}\label{eqn: u variables 4n Appendix}
		u_{i,j,k,\ell} = \begin{cases}
			\frac{\tau_{i+1,n-2,n-1,n}}{\tau_{i,n-2,n-1,n}}, & (i,j,k,\ell) = (i,n-2,n-1,n),\ \ i \le n-4\\
			\frac{\tau_{i+1,j,n-1,n}\tau_{i,j+1,j+2,j+3}}{\tau_{i,j,n-1,n}\tau_{i+1,j+1,j+2,j+3}}, & (i,j,k,\ell) = (i,j,n-1,n),\ \ j \le n-3\\
			\frac{\tau_{i+1,j,k,n}\tau_{i,j,k+1,k+2}}{\tau_{i,j,k,n}\tau_{i+1,j,k+1,k+2}}, & (i,j,k,\ell) = (i,j,k,n),\ \ k\le n-2\\
			\frac{\tau_{i+1,j,k,\ell}\tau_{i,j,k,\ell+1}}{\tau_{i,j,k,\ell}\tau_{i+1,j,k,\ell+1}}, & k <n.
		\end{cases}
	\end{eqnarray}
	
	\begin{rem}
		After using techniques\footnote{We thank N. Arkani-Hamed for explanations.} developed in \cite{AHL2019Stringy,AHLT2019} to calculate explicitly u-variables in the case $(k,n) = (4,8)$, subject to the requirement that they should satisfy some binary-type equations which use the noncrossing rule, we were able to conjecture a formula for all $(k,n)$ and to test it in cases including $(k,n) \in \{(3,15),(4,15),(5,10)\}$.  For $k=3$ we found all binary equations.  For $k=4,5$, we did not immediately\footnote{However, thanks to an insight shared by Hugh Thomas for the so-called non-kissing complex, we are now able to formulate an all-(k,n) exponent rule, counting violations of the noncrossing conditions (See Definition \ref{defn: generalized worldsheet conclusion})!  We keep the example here to illustrate our original methods, for posterity.} achieve rules for the exponents analogous to those in Equation \eqref{eq: binary relations} which we confirmed for $k=3$ and $n\le 15$.
	\end{rem}

	To illustrate the nontriviality of our construction let us explicate in detail what happens for two example binary identities; one succeeds immediately, and one fails on the first attempt, but then succeeds after a suitable modification.
	
	\begin{example}\label{example: (4,8) binary identity}
		Let $J = \{2,3,6,7\}$.  Then, the set of 4-element subsets $I$ such that $(I,J)$ is \textit{not} noncrossing labels the set 
		\begin{eqnarray*}
			\mathcal{U}_J&=&\left\{u_I: I \in \binom{\lbrack 8\rbrack}{4}^{nf},\ \  (I,J) \not\in \mathbf{NC}_{4,8} \right\}\\
			& = & \left\{\begin{array}{ccccccc}
				u_{1245},& u_{1246},& u_{1247},& u_{1248},& u_{1256},& u_{1257},& u_{1258} \\
				u_{1345},& u_{1346},& u_{1347},& u_{1348},& u_{1356},& u_{1357},& u_{1358} \\
				u_{1468},& u_{1478},& u_{1568},& u_{1578},& u_{2468},& u_{2478},& u_{2568} \\
				u_{2578},& u_{3468},& u_{3478},& u_{3568},& u_{3578},& u_{4568},& u_{4578} \\
			\end{array}\right\}
		\end{eqnarray*}
		and one can easily check the following identity by substituting Equation \eqref{eqn: u variables 4n}:
		$$u_{2367} + \prod_{u_I\in \mathcal{U}_{\{2,3,6,7\}}}u_I =1.$$
		Here
		\begin{eqnarray*}
			& & u_{2367} =  \frac{\tau_{3367}\tau_{2368}}{\tau_{2367}\tau_{3368}}\\
			& &  =  \frac{(x_{2,2}+x_{2,3})P}{\left(x_{1,1} x_{2,1}+x_{1,1} x_{2,2}+x_{1,2} x_{2,2}+x_{1,1} x_{2,3}+x_{1,2} x_{2,3}\right) \left(x_{2,2} x_{3,3}+x_{2,3} x_{3,3}+x_{2,2} x_{3,4}+x_{2,3} x_{3,4}+x_{2,4} x_{3,4}\right)},
		\end{eqnarray*}
		where 
		\begin{eqnarray*}
			P & = & x_{1,1} x_{2,1} x_{3,3}+x_{1,1} x_{2,2} x_{3,3}+x_{1,2} x_{2,2} x_{3,3}+x_{1,1} x_{2,3} x_{3,3}+x_{1,2} x_{2,3} x_{3,3}+x_{1,1} x_{2,1} x_{3,4}+x_{1,1} x_{2,2} x_{3,4}\\
			& + & x_{1,2} x_{2,2} x_{3,4}+x_{1,1} x_{2,3} x_{3,4}+x_{1,2} x_{2,3} x_{3,4}+x_{1,1} x_{2,4} x_{3,4}+x_{1,2} x_{2,4} x_{3,4}
		\end{eqnarray*}
		and after much cancellation	
		\begin{eqnarray*}
			& & \prod_{u_I\in \mathcal{U}_{\{2,3,6,7\}}}u_I\\
			& = & \frac{x_{1,1} x_{2,1} x_{2,4} x_{3,4}}{\left(x_{1,1} x_{2,1}+x_{1,1} x_{2,2}+x_{1,2} x_{2,2}+x_{1,1} x_{2,3}+x_{1,2} x_{2,3}\right) \left(x_{2,2} x_{3,3}+x_{2,3} x_{3,3}+x_{2,2} x_{3,4}+x_{2,3} x_{3,4}+x_{2,4} x_{3,4}\right)},
		\end{eqnarray*}
		and the binary identity follows immediately.
		
		In contradistinction, using a subset with three cyclic consecutive intervals $J = \{2,3,6,8\}$ gives 
		\begin{eqnarray*}
			\mathcal{U}_J = \left\{
			\begin{array}{cccccc}
				u_{1237}, & u_{1245}, & u_{1246}, & u_{1247}, & u_{1248}, & u_{1256} \\
				u_{1257}, & u_{1258}, & u_{1267}, & u_{1345}, & u_{1346}, & u_{1347} \\
				u_{1348}, & u_{1356}, & u_{1357}, & u_{1358}, & u_{1367}, & u_{1457} \\
				u_{1478}, & u_{1578}, & u_{2347}, & u_{2357}, & u_{2457}, & u_{2478} \\
				u_{2578}, & u_{3457}, & u_{3478}, & u_{3578}, & u_{4578} &  \\
			\end{array}
			\right\}
		\end{eqnarray*}
		but the same recipe as above for $J = \{2,3,6,7\}$ does not quite work\footnote{However, Equation \eqref{eqn: example 48}, the set of nonzero monomials in the numerator coincides with the set of monomials in the expansion of the denominator.  As a general rule \cite{AHLT2019}, this suggests the possibility that the identity can be repaired with a suitable choice of exponents.}, in that
		\begin{eqnarray}\label{eqn: example 48}
			u_{2368} + \prod_{u_I\in \mathcal{U}_{\{2,3,6,8\}}}u_I \not=1,
		\end{eqnarray}
		but perhaps we would achieve an identity if exponents of some $u_I$ in the second term $\prod_{u_I\in \mathcal{U}_{\{2,3,6,8\}}}u_I$ were replaced with integers larger than 1. 
		
		This is indeed the case.  Here we have 
		\begin{eqnarray*}
			u_{2368} & = & \frac{\tau_{3368}\tau_{2378}}{\tau_{2368}\tau_{3378}},
		\end{eqnarray*}
		where
		\begin{eqnarray*}
			\tau_{3368} & = & x_{1,2} \left(x_{2,2} x_{3,3}+x_{2,3} x_{3,3}+x_{2,2} x_{3,4}+x_{2,3} x_{3,4}+x_{2,4} x_{3,4}\right)\\
			\tau_{2378} & = & \left(x_{1,1} x_{2,1}+x_{1,1} x_{2,2}+x_{1,2} x_{2,2}+x_{1,1} x_{2,3}+x_{1,2} x_{2,3}+x_{1,1} x_{2,4}+x_{1,2} x_{2,4}\right) x_{3,4}\\
			\tau_{2368} & = & x_{1,1} x_{2,1} x_{3,3}+x_{1,1} x_{2,2} x_{3,3}+x_{1,2} x_{2,2} x_{3,3}+x_{1,1} x_{2,3} x_{3,3}+x_{1,2} x_{2,3} x_{3,3}+x_{1,1} x_{2,1} x_{3,4}\\
			& +& x_{1,1} x_{2,2} x_{3,4}+x_{1,2} x_{2,2} x_{3,4}+x_{1,1} x_{2,3} x_{3,4}+x_{1,2} x_{2,3} x_{3,4}+x_{1,1} x_{2,4} x_{3,4}+x_{1,2} x_{2,4} x_{3,4}\\
			\tau_{3378} & = & x_{1,2} \left(x_{2,2}+x_{2,3}+x_{2,4}\right) x_{3,4}.
		\end{eqnarray*}
		
		Now, after much cancellation in Equation \eqref{eqn: example 48} we find that the second term simplifies to
		\begin{eqnarray}\label{eq: 2368 binary identity}
			\prod_{u_I\in \mathcal{U}_{\{2,3,6,8\}}}u_I  & =& \frac{x_{1,1} x_{2,1} x_{2,4} x_{3,3}}{D_1D_4}\frac{N_1N_2N_3}{D_2D_3},
		\end{eqnarray}
		where
		\begin{eqnarray*}
			N_1 & = & \left(x_{3,2}+x_{3,3}\right) \left(x_{3,1}+x_{3,2}+x_{3,3}\right)\\
			N_2 & = & x_{1,1} x_{2,1} x_{3,1}+x_{1,1} x_{2,1} x_{3,2}+x_{1,1} x_{2,2} x_{3,2}+x_{1,2} x_{2,2} x_{3,2}+x_{1,1} x_{2,1} x_{3,3}+x_{1,1} x_{2,2} x_{3,3}+x_{1,2} x_{2,2} x_{3,3}\\
			& + & x_{1,1} x_{2,1} x_{3,4}+x_{1,1} x_{2,2} x_{3,4}+x_{1,2} x_{2,2} x_{3,4}\\
			N_3 & = & x_{1,1} x_{2,1} x_{3,2}+x_{1,1} x_{2,2} x_{3,2}+x_{1,2} x_{2,2} x_{3,2}+x_{1,1} x_{2,1} x_{3,3}+x_{1,1} x_{2,2} x_{3,3}+x_{1,2} x_{2,2} x_{3,3}+x_{1,1} x_{2,3} x_{3,3}\\
			& +&  x_{1,2} x_{2,3} x_{3,3}+x_{1,1} x_{2,1} x_{3,4}+x_{1,1} x_{2,2} x_{3,4}+x_{1,2} x_{2,2} x_{3,4}+x_{1,1} x_{2,3} x_{3,4}+x_{1,2} x_{2,3} x_{3,4}\\
			D_1 & = & x_{2,2}+x_{2,3}+x_{2,4}\\
			D_2 & = & \left(x_{3,2}+x_{3,3}+x_{3,4}\right) \left(x_{3,1}+x_{3,2}+x_{3,3}+x_{3,4}\right)\\
			D_3 & = & x_{1,1} x_{2,1} x_{3,1}+x_{1,1} x_{2,1} x_{3,2}+x_{1,1} x_{2,2} x_{3,2}+x_{1,2} x_{2,2} x_{3,2}+x_{1,1} x_{2,1} x_{3,3}+x_{1,1} x_{2,2} x_{3,3}+x_{1,2} x_{2,2} x_{3,3}\\
			D_4 & = & x_{1,1} x_{2,1} x_{3,3}+x_{1,1} x_{2,2} x_{3,3}+x_{1,2} x_{2,2} x_{3,3}+x_{1,1} x_{2,3} x_{3,3}+x_{1,2} x_{2,3} x_{3,3}+x_{1,1} x_{2,1} x_{3,4}+x_{1,1} x_{2,2} x_{3,4}\\
			& + & x_{1,2} x_{2,2} x_{3,4}+x_{1,1} x_{2,3} x_{3,4}+x_{1,2} x_{2,3} x_{3,4}+x_{1,1} x_{2,4} x_{3,4}+x_{1,2} x_{2,4} x_{3,4}.
		\end{eqnarray*}
		Noticing now that 
		\begin{eqnarray*}
			1-u_{2368} & = & \frac{x_{1,1} x_{2,1} x_{2,4} x_{3,3}}{D_1D_4},
		\end{eqnarray*}
		it follows that the second factor
		$$\frac{N_1N_2N_3}{D_2D_3}$$
		in Equation \eqref{eq: 2368 binary identity} is suspect and we have not yet achieved the binary identity.  However, it is possible to modify Equation \eqref{eqn: example 48} to achieve the result.  After some experimentation we find that 	
		\begin{eqnarray}
			u_{2368} + (u_{1247} u_{1257} u_{1347} u_{1357})\prod_{u_I\in \mathcal{U}_{\{2,3,6,8\}}}u_I = 1,
		\end{eqnarray}
		noting that the front factor $(u_{1247} u_{1257} u_{1347} u_{1357})$ already appears in the product, so we have in effect simply changed some of the exponents to 2.
		
		In particular, this implies the following binary property: whenever $(I,\{2,3,6,8\}) \not\in \mathbf{NC}_{4,8}$ and $u_{I}=0$ then $u_{2368}=1$.	
	\end{example}

		\section{Newton Polytopes}\label{sec:Newton polytopes}
	In this Appendix, we summarize some computations: we study the Newton polytopes $\text{Newt}(\prod_J \tau_J)$ in several cases.

	We find the following f-vectors for respectively (3,6), ((3,7) and (4,7)), (3,8)
	\begin{eqnarray}\label{eq: f-vectors fibered associahedra}
		&&(1,42,84,56,14,1),\nonumber\\
		&&(1,462,1386,1596,882,238,28,1),\\
		&&(1, 6006, 24024, 39468, 34320, 16962, 4752, 708, 48, 1),\nonumber
	\end{eqnarray}
	and we confirmed in each case that, as expected, the facet inequalities for the Newton polytope are of the form $\gamma_J + c_J \ge 0$ for some integers $c_J$.
	
	Here the number of vertices is the multi-dimensional Catalan number $C^{(3)}_{n-3}$, that is the number of pairwise noncrossing collections of $(3-1)(n-3-1)$ \textit{nonfrozen} subsets\footnote{Recall that a subset is \textit{nonfrozen} if it consists of at least two cyclic intervals with respect to the standard cyclic order $\mathbb{I}_n = (1,2,\ldots, n)$.}, and the number of facets is $\binom{n}{3}-n$.  

	For $(k,n) = (3,6)$, only the two minors $p_{2,3,6}$ and $p_{3,5,6}$ are replaced by compound determinants.  All other Plucker coordinates are included without modification.  We exclude monomial factors as they only translate the Newton polytope.  
	
	Then define $g_{3,6}(x)$ to be the product
	$$\left(x_{1,1}+x_{1,2}\right) \left(x_{1,2}+x_{1,3}\right) \left(x_{1,1}+x_{1,2}+x_{1,3}\right) \left(x_{2,1}+x_{2,2}\right) \left(x_{2,2}+x_{2,3}\right) \left(x_{2,1}+x_{2,2}+x_{2,3}\right)$$
	$$\left(x_{1,1} x_{2,1}+x_{1,1} x_{2,2}+x_{1,2} x_{2,2}\right) \left(x_{1,2} x_{2,2}+x_{1,2} x_{2,3}+x_{1,3} x_{2,3}\right)$$
	$$\left(x_{1,1} x_{2,2}+x_{1,2} x_{2,2}+x_{1,1} x_{2,3}+x_{1,2} x_{2,3}+x_{1,3} x_{2,3}\right) \left(x_{1,1} x_{2,1}+x_{1,1} x_{2,2}+x_{1,2} x_{2,2}+x_{1,1} x_{2,3}+x_{1,2} x_{2,3}\right).$$
	Using SageMath we find that the f-vector of the Newton polytope $\text{Newt}(g_{3,6}(x))$ is 
	$$(1,42,84,56,14,1).$$

		\section{Generalized Roots: Facets of $\mathbb{K}^{(k)}_{n-k}$}

	Recall that, given a polytope $P$ and a linear function $f$, then $f$ is minimized on a unique face of $P$.  In what follows, we refine the Minkowski sum decomposition of faces of the PK associahedron $\mathbb{K}^{(k)}_{n-k}$, according to the minimum value of a given $\gamma_I$ on the Minkowski summands $\mathcal{F}^{(i)}_J$ of $\mathbb{K}^{(k)}_{n-k}$.  In the case $k=2$, the minimum value of $\gamma_{ab}$ is always zero or 1.  However, as $k$ increases larger minimum values are found.  For instance, the linear function
	$$\gamma_{1458} = \alpha_{1,1} + \alpha_{1,2} + \alpha_{3,3} + \alpha_{3,4}$$
	is identically $m=2$ on the planar face
	$$\mathcal{F}^{(1)}_{1256} = \mathcal{F}^{(\lbrack 1,3\rbrack)}_{\lbrack 1,2\rbrack, \lbrack 1,4\rbrack,\lbrack 3,4\rbrack}.$$

	For $m\ge 0$, denote by 
	$$\mathcal{E}_m(I) = \left\{ (i,J): \min(\{\gamma_I(v): v\in \mathcal{F}^{(i)}_J\}) = m  \right\}$$
	the set of pairs $(i,J)$ such that on $\mathcal{F}^{(i)}_J$ the minimum value of $\gamma_I$ is $m$.

	For a polyhedron $P$, let $\partial_{\gamma_I(\alpha)}\left(P\right)$
	be the face of $P$ where $\gamma_I$ attains its minimum.  Note that while $\partial_{\gamma_I(\alpha)}\left(P\right)$ is always a face of $P$, a priori it might have codimension 2 or more.  We conjecture that for all $2\le k\le n-2$, the (codimension 1) facets of $\mathbb{K}^{(k)}_{n-k}$ are exactly those that minimize the linear functions $\gamma_I$ as $I$ varies over all nonfrozen $k$-element subsets.
	
	Proposition \ref{prop: facets} adapts the standard result that for a Minkowski sum of polytopes $P = P_1\boxplus \cdots \boxplus P_\ell$, given a linear function $f$, then the face of $P$ that is minimized by $f$ is equal to the Minkowski sum of the faces of the $P_i$ that are minimized by $f$. 
	
	\begin{rem}
		By itself Proposition \ref{prop: facets} is not a priori extremely illuminating, but let us only remark that, in the case of $\mathbb{K}^{(k=2)}_{n-2}$ the two Minkowski summands (in this case, only $m=0,1$ contribute) in Proposition \ref{prop: facets} live in orthogonal subspaces; in fact \textit{this is tied to why factorization works so well for the cubic scalar theory}, in that when a subset $J$ of particles goes on-shell, that is we have $\sum_{i,j\in J}s_{ij}=0$, then the amplitude $m^{(2)}_n$ factors as a product of two amplitudes of the same type.  These two factors correspond, somewhat imprecisely speaking, to $m=0$ and $m=1$ in Proposition \ref{prop: facets}.
	\end{rem}
	
	However when $k\ge3$, it is not difficult to see that the Minkowski summands of faces of $\mathbb{K}^{(k)}_{n-k}$ are no longer in general orthogonal, and they are not in general simple polytopes.  For the last point, one can compute the facet of $\mathbb{K}^{(4)}_{8-4}$ that minimizes the linear function $\gamma_{1,4,5,8}$.  We return to this point after the Proposition.

	\begin{prop}\label{prop: facets}
		The face of $\mathbb{K}^{(k)}_{n-k}$ that minimizes a given $\gamma_I$ decomposes into a Minkowski sum of polyhedra, as 
		\begin{eqnarray}\label{eq: splitting facet}
			\partial_{\gamma_I(\alpha)}\left(\mathbb{K}^{(k)}_{n-k}\right) = \bigboxplus_{m\ge 0}\left(\bigboxplus_{(i,J)\in \mathcal{E}_m(I)}\partial_{\gamma_I(\alpha)}\left(\mathcal{F}^{(i)}_J\right)\right).
		\end{eqnarray}
	\end{prop}
	\begin{example}
		For $m^{(2)}_6$, when $\eta_{36} = s_{45}+s_{56}+s_{46} = \varepsilon$, then we calculate that 
		\begin{eqnarray*}
			m^{(2)}_6 & = & \frac{1}{\varepsilon}\left(\frac{1}{\eta_{26}} + \frac{1}{\eta_{13}}\right)\left(\frac{1}{\eta_{35}} + \frac{1}{\eta_{46}}\right) + \mathcal{O}(\varepsilon^0)\\
			& = & \frac{1}{\varepsilon}\left(\frac{1}{s_{12}} + \frac{1}{s_{23}}\right)\left(\frac{1}{s_{45}} + \frac{1}{s_{56}}\right) + \mathcal{O}(\varepsilon^0).
		\end{eqnarray*}
		This corresponds to one of the faces of the associahedron of the form $\mathbb{K}^{(2)}_{(2)}\times \mathbb{K}^{(2)}_{2}$.
		
		On the other hand, for the Newton polytope we minimize the linear function $\gamma_{36} = \alpha_{1,3} + \alpha_{1,4}$ over the associahedron to calculate the facet as a Newton polytope.  The standard associahedron is the Newton polytope 
		$$\mathbb{K}^{(2)}_{4} = \text{Newt}\left(x_{1,12}x_{1,23}x_{1,34}x_{1,123}x_{1,234}x_{1,1234}\right).$$
		
	\end{example}
	
	We find a decomposition into a product of two polynomials whose Newton polytopes are in orthogonal subspaces,
	\begin{eqnarray}
		\partial_{\gamma_{36}(\alpha)} \left(\mathbb{K}^{(2)}_4\right)  & = & \text{Newt}\left(\left(x_{1,2}^2 \left(x_{1,1}+x_{1,2}\right)^3\right)\left(x_{1,3}+x_{1,4}\right)\right),
	\end{eqnarray}
	where the first and second factors correspond to $m=0,1$ respectively in Proposition \ref{prop: facets}.  This recovers (metrically), the facet $\mathbb{K}^{(2)}_{2}\times \mathbb{K}^{(2)}_{2}$ of $\mathbb{K}^{(2)}_4$ which minimizes $\gamma_{36} = \alpha_{1,3} + \alpha_{1,4}$.
	
	Let us sound a cautionary note: it would be tempting to ask if the respective summands in the decomposition are all simple polytopes; but this is not the case.  To this end, we can study the facet that minimizes $\gamma_{1458} = \alpha_{1,12} + \alpha_{3,34}$, by giving the Minkowski factorization in Proposition \ref{prop: facets}, for the PK associahedron $\mathbb{K}^{(4)}_{4}$.  We find three groups $m=0,1,2$, according to the minimum $m$ value of $\gamma_{1458}$ on each.  The single contribution where $\gamma_{1458}$ takes the minimum value $m=2$ (identically, in fact) is given by 
	$$\text{Newt}(\tau_{2368}) = \mathcal{F}^{\lbrack 1,3\rbrack}_{\lbrack 1,2\rbrack,\lbrack 1,4\rbrack,\lbrack 3,4\rbrack}.$$
	However, one can easily check (using for instance SageMath) that this is not a simple polytope.
	
	\begin{conjecture}
		For any $2\le k\le n-2$, the PK associahedron $\mathbb{K}^{(k)}_{n-k}$ has exactly $\binom{n}{k}-n$ facets.  These are given by the set $\left\{\partial_{\gamma_I(\alpha)}\left(\mathbb{K}^{(k)}_{n-k} \right): I \in \binom{\lbrack n \rbrack}{k}^{nf}\right\}$.		
	\end{conjecture}
	Noting that the PK polytope is a Minkowski sub-summand of $\mathbb{K}^{(k)}_{n-k}$, we can deduce that the $\gamma_J$ are minimized on (codimension 1) facets; however what does not seem so obvious is to show that no new facets are present in the full Newton polytope, the PK associahedron $\mathbb{K}^{(k)}_{n-k}$!

	\begin{example}

		In what follows, we calculate the 14 facets of the Newton polytope of the following product of polynomials
		$$\prod_{\{ijk\} \in \binom{6}{3}}\tau_{ijk},$$
		as Newton polytopes.  In particular, modulo relabeling there are only three facets.  Forgetting monomial factors, this equals the product of the following ten irreducible polynomials:
		$$\left(x_{1,1}+x_{1,2}\right), \left(x_{1,2}+x_{1,3}\right), \left(x_{1,1}+x_{1,2}+x_{1,3}\right), \left(x_{2,1}+x_{2,2}\right), \left(x_{2,2}+x_{2,3}\right), \left(x_{2,1}+x_{2,2}+x_{2,3}\right),$$
		$$\left(x_{1,1} x_{2,1}+x_{1,1} x_{2,2}+x_{1,2} x_{2,2}\right), \left(x_{1,2} x_{2,2}+x_{1,2} x_{2,3}+x_{1,3} x_{2,3}\right),$$
		$$\left(x_{1,1} x_{2,2}+x_{1,2} x_{2,2}+x_{1,1} x_{2,3}+x_{1,2} x_{2,3}+x_{1,3} x_{2,3}\right),$$
		$$ \left(x_{1,1} x_{2,1}+x_{1,1} x_{2,2}+x_{1,2} x_{2,2}+x_{1,1} x_{2,3}+x_{1,2} x_{2,3}\right).$$
		It is not difficult to verify from scratch using SageMath, say, that the fourteen linear functions $\gamma_{ijk}$ are minimized exactly on the 14 (codimension 1) facets of the Newton polytope $\text{Newt}\left(\prod_{\{ijk\} \in \binom{6}{3}}\tau_{ijk}\right)$; therefore by explicitly minimizing the functions $\gamma_{ikj}$ over the vertices of the whole Newton polytope, one derives the following Newton polytopal expressions of the codimension 1 strata of $\mathcal{W}_{3,6}$.
		\begin{figure}[h!]
			\centering
			\includegraphics[width=0.33\linewidth]{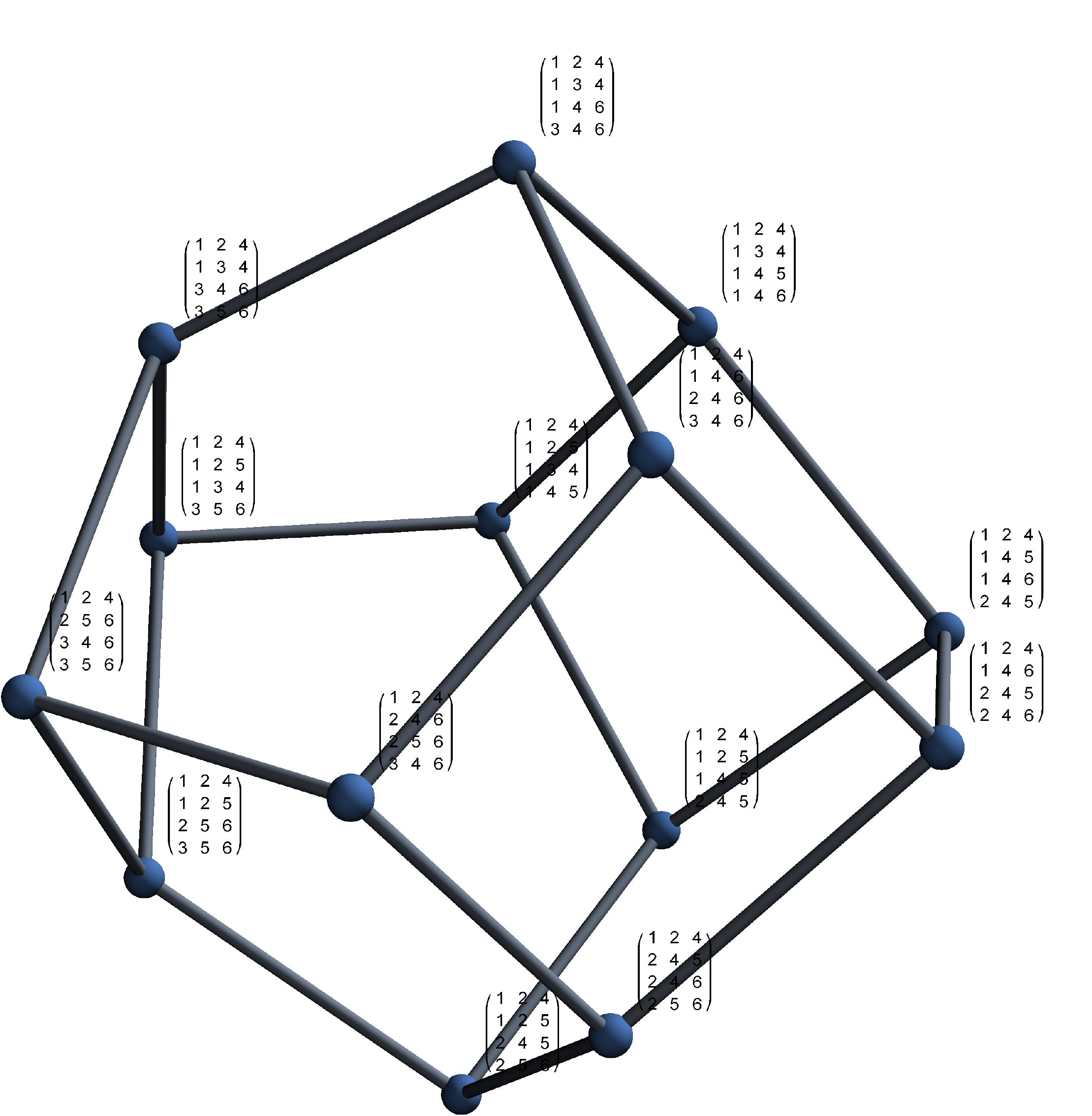}
			\includegraphics[width=0.33\linewidth]{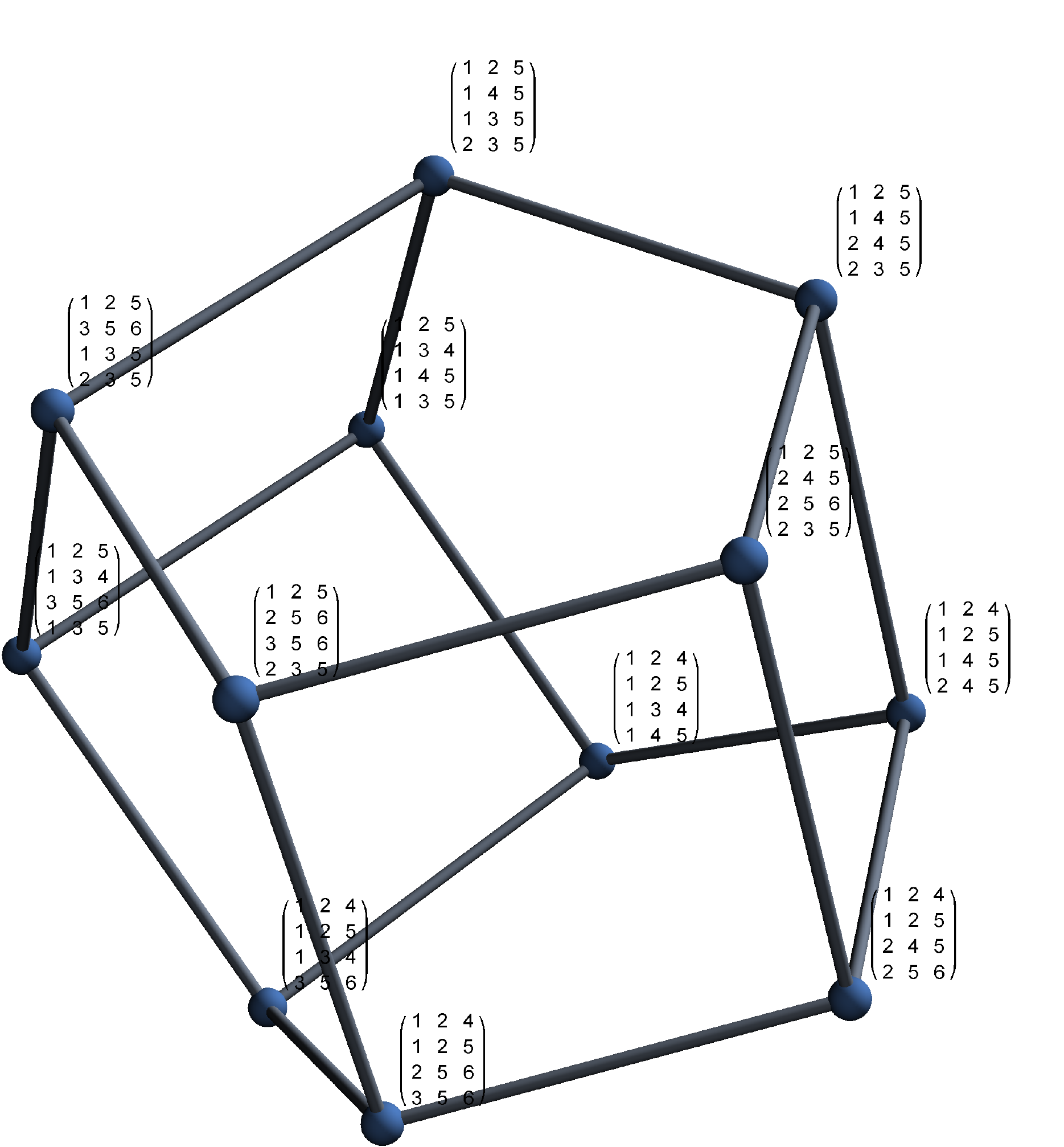}
			\includegraphics[width=0.32\linewidth]{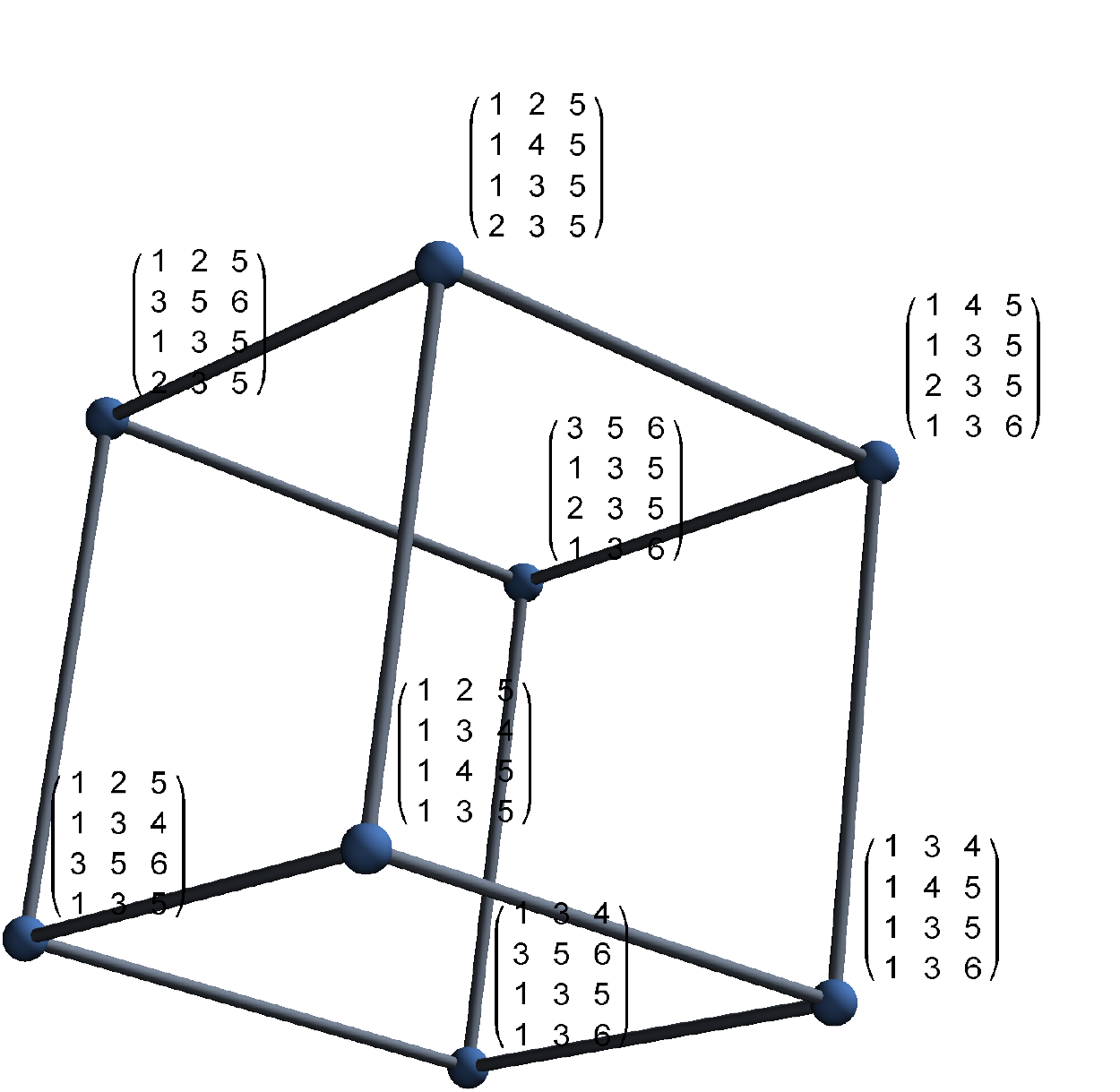}
			\caption{We use the noncrossing complex $\mathbf{NC}_{3,6}$ to classify, modulo index relabeling, the codimension one strata in the generalized worldsheet associahedron $\mathcal{W}^+_{3,6}$.  Left: $\{1,2,4\}$.  Middle: $\{1,2,5\}$.  Right: $\{1,3,5\}$.  In particular, they are simple polytopes.}
			\label{fig:factorizationA36}
		\end{figure}
		Consider the Newton polytope found by minimizing $\gamma_{124} = \alpha_{2,1}$, with polynomial the product of the following six irreducible polynomials:
		$$\left(x_{1,1}+x_{1,2}\right){}^3 \left(x_{1,2}+x_{1,3}\right) \left(x_{1,1}+x_{1,2}+x_{1,3}\right) x_{2,2}^2 \left(x_{2,2}+x_{2,3}\right){}^3$$
		$$\left(x_{1,2} x_{2,2}+x_{1,2} x_{2,3}+x_{1,3} x_{2,3}\right) \left(x_{1,1} x_{2,2}+x_{1,2} x_{2,2}+x_{1,1} x_{2,3}+x_{1,2} x_{2,3}+x_{1,3} x_{2,3}\right).$$
		One can argue informally, by simply counting faces of the Newton polytope that it is then isomorphic to the usual 3-dimensional associahedron; moreover, see Figure \ref{fig:factorizationA36} (left) for the corresponding noncrossing collections, as can be reconstructed from the binary relations for the $u_{ijk}$'s.

		One similarly concludes that the stratum of $\mathcal{W}^+_{3,6}$ characterized by $u_{125}=0$ is isomorphic to the Newton polytope found by minimizing $\gamma_{125} = \alpha_{2,1} + \alpha_{2,2}$; this gives the polynomial
		$$\left(x_{1,1}+x_{1,2}\right){}^2 \left(x_{1,2}+x_{1,3}\right){}^2 \left(x_{1,1}+x_{1,2}+x_{1,3}\right){}^2 \left(x_{2,1}+x_{2,2}\right) \left(x_{1,1} x_{2,1}+x_{1,1} x_{2,2}+x_{1,2} x_{2,2}\right)x_{2,3}^5$$
		
		The stratum of $\mathcal{W}_{3,6}$ characterized by $u_{135} = 0$ is a cube, isomorphic to the Newton polytope found by minimizing $\gamma_{135} = \alpha_{1,1} + \alpha_{2,2}$, giving
		$$x_{1,2}^2 \left(x_{1,2}+x_{1,3}\right){}^4 x_{2,1} \left(x_{1,1} x_{2,1}+x_{1,2} x_{2,2}\right) x_{2,3}^4 \left(x_{2,1}+x_{2,3}\right).$$
		
		Note not only the presence of the exponents on some of the factors above, but also the factor $\left(x_{1,1} x_{2,1}+x_{1,2} x_{2,2}\right)$, which is not among the ten irreducible polynomials!  A similar polynomial 
		$$x_{1,1}^4 x_{1,3} \left(x_{1,1}+x_{1,3}\right) x_{2,2}^2 \left(x_{2,1}+x_{2,2}\right){}^4 \left(x_{1,2} x_{2,2}+x_{1,3} x_{2,3}\right)$$
		can be found for the facet that minimizes the function $\gamma_{246}$.
	
		Using the command ``.Hrepresentation()'' in SageMath one can easily obtain the facet inequalities for the Newton polytope of the product of the irreducible face polynomials $\delta^{(i)}_J$.
		
		One finds
		\begin{eqnarray*}
			\alpha_{1,1} + \alpha_{1,2} + \alpha_{1,3} = \gamma_{134} + \gamma_{245} + \gamma_{356} & = & 7\\
			\alpha_{2,1} + \alpha_{2,2} + \alpha_{2,3} = \gamma_{124} + \gamma_{235} + \gamma_{346} & = & 7\\
			\alpha_{1,1} = \gamma_{134} & \ge & 0\\
			\alpha_{1,2} = \gamma_{245} & \ge & 0\\
			\alpha_{1,3} = \gamma_{356} & \ge & 0\\
			\alpha_{2,1} = \gamma_{124} & \ge & 0\\
			\alpha_{2,2} = \gamma_{235} & \ge & 0\\
			\alpha_{2,3} = \gamma_{346} & \ge & 0\\
			\alpha_{1,1}+\alpha_{1,2} = \gamma_{145} & \ge & 3\\
			\alpha_{1,2} + \alpha_{1,3}  = \gamma_{256} & \ge & 2\\
			\alpha_{2,1} + \alpha_{2,2} = \gamma_{125} & \ge & 2\\
			\alpha_{2,2} + \alpha_{2,3} = \gamma_{236}  & \ge & 3\\
			\alpha_{1,1} + \alpha_{2,2}  = \gamma_{135}& \ge & 1\\
			\alpha_{1,2} + \alpha_{2,3} = \gamma_{246} & \ge & 1\\
			\alpha_{1,1} + \alpha_{2,2} + \alpha_{2,3} = \gamma_{136} & \ge & 5\\
			\alpha_{1,1} + \alpha_{1,2} + \alpha_{2,3} = \gamma_{146} & \ge & 5.\\
		\end{eqnarray*}	
	\end{example}

	We present, in Figure \ref{fig:4-3-2021-x37facetshavinggamma147}, a (rather lossy) projection of the Newton polytope $\text{Newt}\left(\prod_{\{ijk\} \in \binom{\lbrack 7\rbrack}{3}} p_{ijk}\right)$.
	\begin{figure}[h!]
		\centering
		\includegraphics[width=0.4\linewidth]{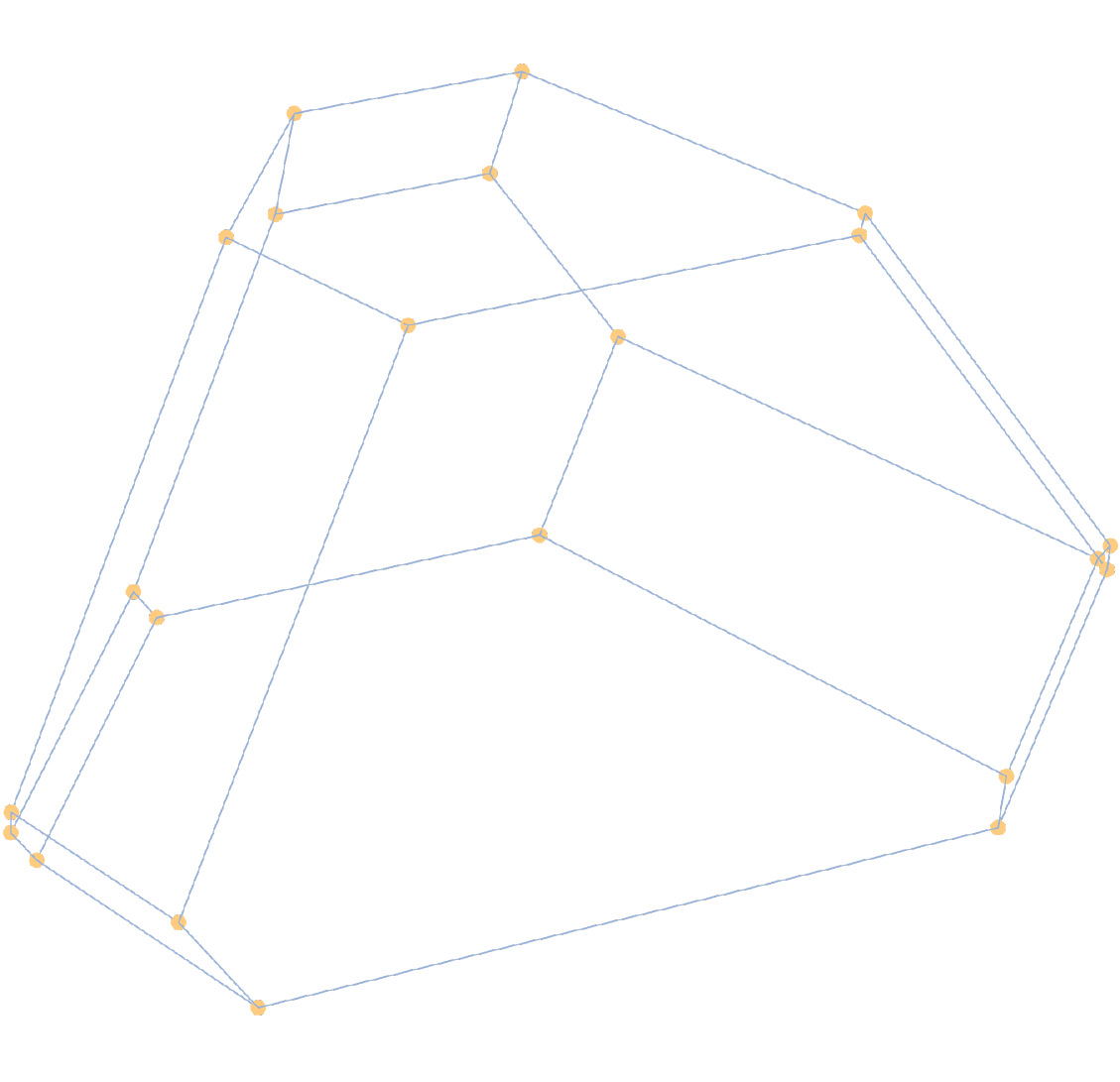}
		\includegraphics[width=0.4\linewidth]{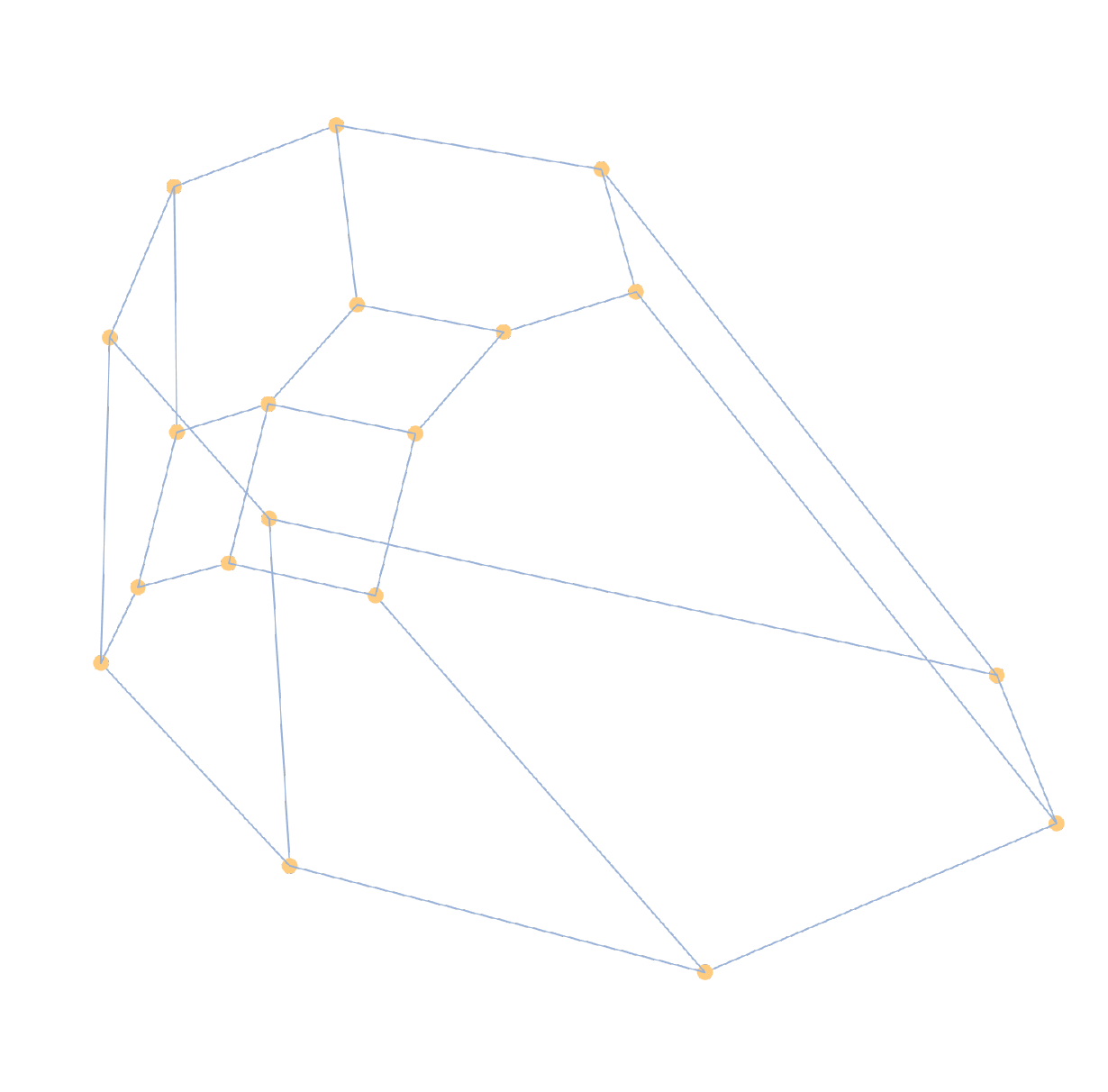}
		\caption{\textbf{Left:} A lossy projection $\mathbb{R}^{(3-1)\times (7-3)} \rightarrow \mathbb{R}^{7-3}$ given by $e_{i,j} \mapsto e_j$ of the (6-dimensional) Newton polytope of the product of all $3\times 3$ minors in $X(3,7)$ (in the positive parameterization).  Despite the lossiness, still one could start to extract some potentially very useful geometric intuition: for instance, just that the projection is the usual 3-d associahedron, but with some edges shaved off.  \textbf{Right:} Same projection, now of the (5-dimensional) facet of the $X(3,7)$ Newton polytope which minimizes the linear function $\gamma_{147} = \alpha_{1,1}+ \alpha_{1,2} + \alpha_{2,3} + \alpha_{2,4}$.  }
		\label{fig:4-3-2021-x37facetshavinggamma147}
	\end{figure}
	where the Newton polytope involves only minors, not resolved minors.

	To render the figure, we compute the facet of the Newton polytope that minimizes the linear function 
	$$\gamma_{147} = \alpha_{1,1}+ \alpha_{1,2} + \alpha_{2,3} + \alpha_{2,4}.$$
	This gives a 5-dimensional facet; then we project using $e_{i,j} \mapsto e_j$ to obtain a 3-dimensional polytope.  Both polytopes in the Figure are compatible with the notion that the Newton polytope of the product of all maximal $3\times 3$ minors,
	$$\text{Newt}\left(\prod_{\{i,j,k\}\in \binom{\lbrack 7\rbrack}{3}} p_{ijk}\right),$$
	in the positive parameterization, see Section \ref{sec: positive parameterization}, can be constructed from the Newton polytope $\mathbb{K}^{(3)}_{7-3}$ by shaving off some additional faces.  This shaving can be understood in terms of poles of the generalized biadjoint scalar $m^{(3)}_7$ that are not in the basis of planar kinematic invariants $\eta_J$, of the form $\sum_{\{i,j,k\}} c_{ijk} \eta_{ijk}$, for certain integers $c_{ijk}$.  It turns out that there are 14 such extra poles of $m^{(3)}_7$, see for instance \cite{arkani2020positive, BC2019,DrummondTropical2019A} for the full list of poles; in the present setting, the 14 extra poles corresponding to the pairs $\{I,J\}$ in the table in Example \ref{example: Plucker resolution explicit 37} take the form $\gamma_I + \gamma_J +c_{I,J}$ for some constant $c_{I,J}$.  Consequently the facet normal vectors are dual to linear functions which are (positive) linear combinations of $\gamma_J$ with at least two terms (however, it turns out that $(k,n) = (3,8)$ is the first time that one finds poles with three terms in the expansion in $\gamma_J$'s).
	
	\section{Kinematic Shift and Resolved Potential Function $\hat{\mathcal{S}}_{3,8}$}\label{sec:kinematic shift Appendix}
	To define the resolved potential function $\hat{\mathcal{S}}_{3,8}$ and the corresponding u-variables we introduce resolved planar kinematic invariants $\hat{\eta}_J$ and then straighten the potential function; we organize them suggestively in the style of weighted blade arrangements.  Then as usual the (resolved) planar cross-ratios are the arguments of the logarithms.  For each $I \in \binom{\lbrack 8\rbrack}{3}^{nf}$ such that $(I,J) \not \in \mathbf{NC}_{3,8}$ for some $J$ with $I$ lexicographically smaller than $J$, we define resolved planar kinematic invariants $\hat{\eta}_I$ by 
	\begin{eqnarray*}
		\hat{\eta}_{124}  + \eta_{378} & = & -\eta_{248} + \eta_{348} + \eta_{278} + \eta_{124}\\
		\hat{\eta}_{125}  + \eta_{378} & = & -\eta_{258} + \eta_{358} + \eta_{278} + \eta_{125}\\
		\hat{\eta}_{126}  + \eta_{378} & = & -\eta_{268} + \eta_{368} + \eta_{278} + \eta_{126}\\
		\hat{\eta}_{135}  + \eta_{478} & = & -\eta_{358} + \eta_{458} + \eta_{378} + \eta_{135}\\
		\hat{\eta}_{136}  + \eta_{578} & = & -\eta_{368} + \eta_{468} + \eta_{278} + \eta_{136}\\
		\hat{\eta}_{145}  + \eta_{238} & = & -\eta_{135} + \eta_{235} + \eta_{145} + \eta_{138}\\
		\hat{\eta}_{146}  + \eta_{238} + \eta_{578} & = & (-\eta_{136} + \eta_{236} + \eta_{146} + \eta_{138}) + (-\eta_{468} + \eta_{568} + \eta_{478} + \eta_{146}) - \eta_{146}\\
		\hat{\eta}_{147}  + \eta_{238} & = & -\eta_{137} + \eta_{237} + \eta_{147} + \eta_{138}\\
		\hat{\eta}_{156}  + \eta_{238} + \eta_{348} & = & (-\eta_{136} + \eta_{236} + \eta_{156} + \eta_{148}) + (-\eta_{146} + \eta_{346} + \eta_{156} + \eta_{148}) - \eta_{156}\\
		\hat{\eta}_{157}  + \eta_{238} + \eta_{348} & = & (-\eta_{137} + \eta_{237} + \eta_{157} + \eta_{148}) + (-\eta_{147} + \eta_{347} + \eta_{157} + \eta_{148}) - \eta_{157}\\
		\hat{\eta}_{167}  + \eta_{238} + \eta_{348} + \eta_{458} & = & (-\eta_{137} + \eta_{237} + \eta_{167} + \eta_{148}) + (-\eta_{147} + \eta_{347} + \eta_{167} + \eta_{148})\\
		& + & (-\eta_{157} + \eta_{457} + \eta_{167} + \eta_{158}) - 2\eta_{167}\\
	\end{eqnarray*}
	\begin{eqnarray*}
		\hat{\eta}_{235}  + \eta_{478} & = & (-\eta_{358} + \eta_{458} + \eta_{378} + \eta_{235})\\
		\hat{\eta}_{236} +\eta_{478} & = & (-\eta_{368} + \eta_{468} + \eta_{378} + \eta_{236})\\
		\hat{\eta}_{246} + \eta_{578} & = & (-\eta_{468} + \eta_{568} + \eta_{478} + \eta_{246})\\
		\hat{\eta}_{256} + \eta_{348} & = & (-\eta_{246} + \eta_{346} + \eta_{256} + \eta_{248})\\
		\hat{\eta}_{257} + \eta_{348} & = & (-\eta_{247} + \eta_{347} + \eta_{257} + \eta_{248})\\
		\hat{\eta}_{267}  + \eta_{348} + \eta_{458} & = & (-\eta_{247} + \eta_{347} + \eta_{267} + \eta_{248}) + (-\eta_{257} + \eta_{457} + \eta_{267} + \eta_{258}) - \eta_{267}\\
		\hat{\eta}_{346} + \eta_{578} & = & (-\eta_{468} + \eta_{568} + \eta_{478} + \eta_{346})\\
		\hat{\eta}_{367} + \eta_{458} & = & (-\eta_{357} + \eta_{457} + \eta_{367} + \eta_{358}).
	\end{eqnarray*}
	In all other cases we put $\hat{\eta}_J = \eta_J$.

	With this substitution, the potential function $\mathcal{S}_{3,8}$ is fully resolved and we find u variables from inside the logarithms in the potential function below.  
	\begin{eqnarray*}
		\hat{\mathcal{S}}_{3,8} &= & \sum_{J \in \binom{\lbrack 8\rbrack}{3}^{nf}}\hat{\eta}_J\log(u_J)\\
		&= &\hat{\eta} _{124} \log \left(\frac{\hat{p}_{125} \hat{p}_{134}}{\hat{p}_{124} \hat{p}_{135}}\right)+\hat{\eta} _{125} \log \left(\frac{\hat{p}_{126} \hat{p}_{135}}{\hat{p}_{125} \hat{p}_{136}}\right)+\hat{\eta} _{126} \log \left(\frac{\hat{p}_{127} \hat{p}_{136}}{\hat{p}_{126} \hat{p}_{137}}\right)+\hat{\eta} _{127} \log \left(\frac{\hat{p}_{128} \hat{p}_{137}}{\hat{p}_{127} \hat{p}_{138}}\right) \\ &+ 
		&\hat{\eta} _{134} \log \left(\frac{\hat{p}_{135} \hat{p}_{234}}{\hat{p}_{134} \hat{p}_{235}}\right)+\hat{\eta} _{135} \log \left(\frac{\hat{p}_{136} \hat{p}_{235}}{\hat{p}_{135} \hat{p}_{236}}\right)+\hat{\eta} _{136} \log \left(\frac{\hat{p}_{137} \hat{p}_{236}}{\hat{p}_{136} \hat{p}_{237}}\right)+\hat{\eta} _{235} \log \left(\frac{\hat{p}_{145} \hat{p}_{236}}{\hat{p}_{146} \hat{p}_{235}}\right) \\ &+ 
		&\hat{\eta} _{137} \log \left(\frac{\hat{p}_{138} \hat{p}_{237}}{\hat{p}_{137} \hat{p}_{238}}\right)+\hat{\eta} _{138} \log \left(\frac{\hat{p}_{145} \hat{p}_{238}}{\hat{p}_{138} \hat{p}_{245}}\right)+\hat{\eta} _{236} \log \left(\frac{\hat{p}_{146} \hat{p}_{237}}{\hat{p}_{147} \hat{p}_{236}}\right)+\hat{\eta} _{237} \log \left(\frac{\hat{p}_{147} \hat{p}_{238}}{\hat{p}_{148} \hat{p}_{237}}\right) \\ &+ 
		&\hat{\eta} _{145} \log \left(\frac{\hat{p}_{146} \hat{p}_{245}}{\hat{p}_{145} \hat{p}_{246}}\right)+\hat{\eta} _{146} \log \left(\frac{\hat{p}_{147} \hat{p}_{246}}{\hat{p}_{146} \hat{p}_{247}}\right)+\hat{\eta} _{147} \log \left(\frac{\hat{p}_{148} \hat{p}_{247}}{\hat{p}_{147} \hat{p}_{248}}\right)+\hat{\eta} _{238} \log \left(\frac{\hat{p}_{123} \hat{p}_{148} \hat{p}_{245}}{\hat{p}_{124} \hat{p}_{145} \hat{p}_{238}}\right) \\ &+ 
		&\hat{\eta} _{148} \log \left(\frac{\hat{p}_{156} \hat{p}_{248}}{\hat{p}_{148} \hat{p}_{256}}\right)+\hat{\eta} _{156} \log \left(\frac{\hat{p}_{157} \hat{p}_{256}}{\hat{p}_{156} \hat{p}_{257}}\right)+\hat{\eta} _{157} \log \left(\frac{\hat{p}_{158} \hat{p}_{257}}{\hat{p}_{157} \hat{p}_{258}}\right)+\hat{\eta} _{158} \log \left(\frac{\hat{p}_{167} \hat{p}_{258}}{\hat{p}_{158} \hat{p}_{267}}\right) \\ &+ 
		&\hat{\eta} _{167} \log \left(\frac{\hat{p}_{168} \hat{p}_{267}}{\hat{p}_{167} \hat{p}_{268}}\right)+\hat{\eta} _{168} \log \left(\frac{\hat{p}_{178} \hat{p}_{268}}{\hat{p}_{168} \hat{p}_{278}}\right)+\hat{\eta} _{245} \log \left(\frac{\hat{p}_{246} \hat{p}_{345}}{\hat{p}_{245} \hat{p}_{346}}\right)+\hat{\eta} _{246} \log \left(\frac{\hat{p}_{247} \hat{p}_{346}}{\hat{p}_{246} \hat{p}_{347}}\right) \\& + 
		&\hat{\eta} _{247} \log \left(\frac{\hat{p}_{248} \hat{p}_{347}}{\hat{p}_{247} \hat{p}_{348}}\right)+\hat{\eta} _{248} \log \left(\frac{\hat{p}_{256} \hat{p}_{348}}{\hat{p}_{248} \hat{p}_{356}}\right)+\hat{\eta} _{346} \log \left(\frac{\hat{p}_{156} \hat{p}_{347}}{\hat{p}_{157} \hat{p}_{346}}\right)+\hat{\eta} _{347} \log \left(\frac{\hat{p}_{157} \hat{p}_{348}}{\hat{p}_{158} \hat{p}_{347}}\right) \\ &+ 
		&\hat{\eta} _{256} \log \left(\frac{\hat{p}_{257} \hat{p}_{356}}{\hat{p}_{256} \hat{p}_{357}}\right)+\hat{\eta} _{257} \log \left(\frac{\hat{p}_{258} \hat{p}_{357}}{\hat{p}_{257} \hat{p}_{358}}\right)+\hat{\eta} _{258} \log \left(\frac{\hat{p}_{267} \hat{p}_{358}}{\hat{p}_{258} \hat{p}_{367}}\right)+\hat{\eta} _{348} \log \left(\frac{\hat{p}_{124} \hat{p}_{158} \hat{p}_{356}}{\hat{p}_{125} \hat{p}_{156} \hat{p}_{348}}\right) \\& + 
		&\hat{\eta} _{267} \log \left(\frac{\hat{p}_{268} \hat{p}_{367}}{\hat{p}_{267} \hat{p}_{368}}\right)+\hat{\eta} _{268} \log \left(\frac{\hat{p}_{278} \hat{p}_{368}}{\hat{p}_{268} \hat{p}_{378}}\right)+\hat{\eta} _{278} \log \left(\frac{\hat{p}_{124} \hat{p}_{378}}{\hat{p}_{134} \hat{p}_{278}}\right)+\hat{\eta} _{356} \log \left(\frac{\hat{p}_{357} \hat{p}_{456}}{\hat{p}_{356} \hat{p}_{457}}\right) \\ &+ 
		&\hat{\eta} _{357} \log \left(\frac{\hat{p}_{358} \hat{p}_{457}}{\hat{p}_{357} \hat{p}_{458}}\right)+\hat{\eta} _{358} \log \left(\frac{\hat{p}_{367} \hat{p}_{458}}{\hat{p}_{358} \hat{p}_{467}}\right)+\hat{\eta} _{457} \log \left(\frac{\hat{p}_{167} \hat{p}_{458}}{\hat{p}_{168} \hat{p}_{457}}\right)+\hat{\eta} _{458} \log \left(\frac{\hat{p}_{125} \hat{p}_{168} \hat{p}_{467}}{\hat{p}_{126} \hat{p}_{167} \hat{p}_{458}}\right) \\ &+ 
		&\hat{\eta} _{367} \log \left(\frac{\hat{p}_{368} \hat{p}_{467}}{\hat{p}_{367} \hat{p}_{468}}\right)+\hat{\eta} _{368} \log \left(\frac{\hat{p}_{378} \hat{p}_{468}}{\hat{p}_{368} \hat{p}_{478}}\right)+\hat{\eta} _{378} \log \left(\frac{\hat{p}_{125} \hat{p}_{134} \hat{p}_{478}}{\hat{p}_{124} \hat{p}_{145} \hat{p}_{378}}\right)+\hat{\eta} _{467} \log \left(\frac{\hat{p}_{468} \hat{p}_{567}}{\hat{p}_{467} \hat{p}_{568}}\right) \\ &+ 
		&\hat{\eta} _{468} \log \left(\frac{\hat{p}_{478} \hat{p}_{568}}{\hat{p}_{468} \hat{p}_{578}}\right)+\hat{\eta} _{478} \log \left(\frac{\hat{p}_{126} \hat{p}_{145} \hat{p}_{578}}{\hat{p}_{125} \hat{p}_{156} \hat{p}_{478}}\right)+\hat{\eta} _{568} \log \left(\frac{\hat{p}_{126} \hat{p}_{578}}{\hat{p}_{127} \hat{p}_{568}}\right)+\hat{\eta} _{578} \log \left(\frac{\hat{p}_{127} \hat{p}_{156} \hat{p}_{678}}{\hat{p}_{126} \hat{p}_{167} \hat{p}_{578}}\right).
	\end{eqnarray*}
	It is straightforward to verify exhaustively that the following binary relations from Equation \eqref{eq: binary relations} hold among the corresponding arguments of the logarithms above,
	\begin{eqnarray}
		u_J & =& 1 - \prod_{\{I:\ (I,J) \not\in\mathbf{NC}_{3,8}\}}\hat{u}^{c_{I,J}}_I,
	\end{eqnarray}
	where as above,
	$$c_{(i_1,i_2,i_3),(j_1,j_2,j_3)} = \begin{cases}
		2 & i_1<j_1<i_2<j_2<i_3<j_3\text{ or } j_1<i_1<j_2<i_2<j_3<i_3\\
		1 & \text{otherwise}
	\end{cases}$$
	whenever the pair $\{(i_1,i_2,i_3),(j_1,j_2,j_3)\}$ is crossing.
	
	Let us conclude by summarizing our findings more compactly.  For any alternating pair of triples $(A,B)$ with $(a_1<b_1<a_2<b_2<a_3<b_3)$ (or $(b_1<a_1<b_2<a_2<b_3<a_3)$), define (as in what we did for tripods $\tau_{I,J}$)
	$$\eta_{(\{a_1,a_2,a_3\},\{b_1,b_2,b_3\})} = -\eta_{a_1,a_2,a_3} + \eta_{b_1,a_2,a_3} + \eta_{a_1,b_2,a_3} + \eta_{a_1,a_2,b_3}.$$
	Note that $\eta_{A,B} \not = \eta_{B,A}$.
	
	Then for each $I = \{i_1,i_2,i_3\}\in \binom{\lbrack n\rbrack}{3}^{nf}$ such that $(I,J)$ is crossing for some $J \in \binom{\lbrack n\rbrack}{3}^{nf}$, define $\hat{\eta}_{i_1,i_2,i_3}$ by the equation
	\begin{eqnarray}\label{eq: kinematic shift 3n}
		& & \hat{\eta}_{i_1,i_2,i_3} + \sum_{\{j:\ i_1<j<j+1<i_2<i_3<n\}} \eta_{j,j+1,n} + \sum_{\{j:\ i_1<i_2<j<i_3<n-1<n \}} \eta_{j,n-1,n}\nonumber\\
		& = & \sum_{\{j:\ i_1<j<j+1<i_2<i_3<n\}}\eta_{(i_1,j+1,i_3),(j,i_2,n)} +  \sum_{\{j:\ i_1<i_2<j<i_3<n-1<n\}}\eta_{(i_2,i_3,n),(i_1,j,n-1)}\\
		& - & (n_{I}-1)(\eta_{i_1,i_2,i_3})\nonumber
	\end{eqnarray}
	where $n_{I} = (i_3-i_2-1) + (i_2-i_1-2) = i_3-i_1-3$ is the total number of summands in the second line.  Here the last term $\eta_{i_1,i_2,i_3}$ is included to compensate for over-counting.
	
	For all other $I = \{i_1,i_2,i_3\}$ define 
	$$\hat{\eta}_{i_1,i_2,i_3} = \eta_{i_1,i_2,i_3}.$$
	\begin{question}\label{conjecture: noncrossing triangulation}
		Suppose that $(s_0) \in \mathcal{K}_D(3,n)$ is interior.  Suppose that $(I,J) \in \mathbf{NC}_{3,n}$ and $I',J' \in \binom{\lbrack n\rbrack}{3}^{nf}$ such that 
		$$\gamma_I + \gamma_J = \gamma_{I'} + \gamma_{J'}.$$
		Do we have that 
		$$(\hat{\eta}_{I'}(s_0) + \hat{\eta}_{J'}(s_0))-(\hat{\eta}_{I}(s_0)+\hat{\eta}_{J}(s_0)) >0?$$
	\end{question}
		This would say that we could follow the methods of \cite{GrassmannAssociahedron}, this time using the height function $\mathfrak{t}_{s_0}(v_J) = \hat{\eta}_J(s_0)$ to induce the noncrossing triangulation (Corollary \ref{cor: triangulation root polytope}) of the generalized root polytope $\mathcal{R}^{(3)}_{n-3}$, so that the face poset of the triangulation would be isomorphic to $\mathbf{NC}_{3,n}$.  See Example \ref{example: noncrossing inequalities} and surrounding discussion for more details.

	We stress that Equation \eqref{eq: kinematic shift 3n} has limited validity as it is the outcome of our attempt to write down a closed form expression from computational results only for $n\le 9$.  It would be interesting and important either to give a combinatorial proof that it holds true for general $(3,n)$, perhaps using matroidal weighted blade arrangements \cite{Early2020WeightedBladeArrangements}, or to find a counterexample and explain what goes wrong.  In either case, such questions are left to future work.

\end{document}